\documentclass[a4paper,USenglish]{article}
\usepackage{microtype}
\usepackage{amsthm,amssymb,amsmath}
\usepackage{epsfig,color}
\usepackage{graphicx}
\usepackage{hyperref}
\usepackage[shortlabels]{enumitem}

\usepackage{ifthen}

\usepackage{tikz}
\usetikzlibrary{shapes}
\usetikzlibrary{arrows.meta}
\usetikzlibrary{calc}
\tikzstyle{every picture} = [>=latex]

\newcommand{\RR}{\mathcal{R}}
\newcommand{\PP}{\mathcal{P}}
\newcommand{\TT}{\mathcal{T}}
\newcommand{\QQ}{\mathcal{Q}}

\newcommand{\fnlab}[1]{f_{\ref{#1}}}

\theoremstyle{plain}
\newtheorem{theorem}{Theorem}[section]
\newtheorem{lemma}[theorem]{Lemma}
\newtheorem{corollary}[theorem]{Corollary}
\newtheorem{proposition}[theorem]{Proposition}

\newtheorem{observation}[theorem]{Observation}
\newtheorem{definition}[theorem]{Definition}
\newtheorem{claim}{Claim}

\newenvironment{subproof}{%
  \begin{proof}[Subproof]%
}{%
  \end{proof}%
}

\def\crg{\mathop{\rm cr}}
\let\sem\setminus

\title{Structure and generation of crossing-critical graphs\thanks{This is an extended version of the paper with the same title presented
at 34th International Symposium on Computational Geometry (SoCG 2018).  Compared to the conference version, we have strengthened the
structural description by showing not only that every crossing-critical graph can be obtained by a sequence of expansions,
but also that the expansions only produce crossing-critical graphs. Furthermore, we have added details to many of the proofs, and implemented a number of
simplifications and corrections.}}
\author{Zden\v ek Dvo\v r\'ak\thanks{Charles University, Prague, Czech Republic.
	\url{rakdver@iuuk.mff.cuni.cz}
	Supported by the Center of Excellence -- Institute for Theoretical Computer Science, Prague, project P202/12/G061 of the Czech Science Foundation.
	The new results for the journal version were obtained with the support of the
	ERC-CZ project LL2328 (Beyond the Four Color Theorem) of the Ministry of Education of Czech Republic.}
\and Petr Hlin\v en\'y\thanks{Faculty of Informatics, Masaryk University, Brno, Czech Republic.
        \url{hlineny@fi.muni.cz}. Supported by the Center of Excellence --
 	Institute for Theoretical Computer Science, Brno, project P202/12/G061 of the Czech Science Foundation.}
\and Bojan Mohar\thanks{Department of Mathematics, Simon Fraser University, Burnaby, BC ~V5A 1S6, Canada.
        \url{mohar@sfu.ca}
	Supported in part by the NSERC Discovery Grant R611450 (Canada), by the Canada Research Chairs program, 
	and by the Research Project J1-8130 of ARRS (Slovenia).}}

\begin{document}

\maketitle

\begin{abstract}
We study $c\,$-crossing-critical graphs, which are 
the minimal graphs that require at least $c$ edge-crossings when drawn in the plane.
For $c=1$ there are only two such graphs without \mbox{degree-$2$} vertices, $K_5$ and $K_{3,3}$,
but for any fixed $c>1$ there exist infinitely many $3$-connected $c\,$-crossing-critical graphs.
It has been previously shown that
$c\,$-crossing-critical graphs have bounded path-width and contain only a bounded
number of internally disjoint paths between any two vertices.  

We expand on these results, providing a more detailed description of the structure of
crossing-critical graphs.  
On the way towards this description, we prove a new structural
result on plane graphs of bounded path-width.
Then we show that every $c\,$-crossing-critical
graph can be obtained from a $c\,$-crossing-critical graph of bounded size by
replicating bounded-size parts that already appear in narrow ``bands'' or
``fans'' in the graph.  
This also gives an algorithm to generate all $c$-cros\-sing-critical 
graphs of at most given order $n$ in polynomial time per each generated graph.
\end{abstract}

\section{Introduction}
\label{sec:intro}

Minimizing the number of edge-crossings in a graph drawing in the plane
(the {\em crossing number} $\crg(G)$ of the graph $G$) is considered one of the most important attributes of a ``nice drawing'' of a graph,
and this question has found numerous other applications
(for example, in VLSI design~\cite{leighton1983complexity}
and in discrete geometry~\cite{szekely1997crossing}).
Consequently, a great deal of research work has been invested into
understanding what forces the graph crossing number to be high.
There exist strong quantitative lower bounds, such as the famous Crossing
Lemma \cite{ajtai1982crossing,leighton1983complexity}.
However, the quantitative bounds show their strength typically in
dense graphs, and hence they do not shed much light on
the structural properties of sparse graphs of high crossing number.

Let us remark that on the positive side, for every fixed positive integer~$c$, it is possible
to decide whether an input graph has crossing number at most~$c$ in polynomial time~\cite{grohe2001computing,kawarabayashi2007computing,FPTcr-SODA25}.
However, in case the answer is negative, these algorithms do not reveal much
about the reasons why the crossing number is high: They proceed by first reducing the input instance
to a subinstance(s) of small tree-width, then applying a general meta-algorithmic result of Courcelle~\cite{courcelle}
(in \cite{grohe2001computing,kawarabayashi2007computing}), or using a specialized algorithm to solve the whole problem alongside the tree-width reduction routine \cite{FPTcr-SODA25}.
That is, one can only conclude from this that if the crossing-number
is large, it is because of the presence of a subgraph of small tree-width but with large crossing number.

This indicates that to understand the structural reasons for large crossing number, we need to study
minimal obstructions to drawing with less than $c$ crossing, called \emph{$c\,$-crossing-critical} graphs, in more detail.

\begin{definition}[crossing-critical]
\label{def:crosscritical}
Let $c$ be a positive integer.
A graph $G$ is \emph{$c\,$-crossing-critical} if $\crg(G)\ge c$, but every proper
subgraph $G'$ of $G$ has $\crg(G')<c$.
\end{definition}

By Kuratowski's theorem, the $1$-crossing-critical graphs are exactly the subdivisions of $K_5$
and $K_{3,3}$.  Thus, one could at first glance hope that we might be able to show that
(up to subdivisions), there are only finitely many $c\,$-crossing-critical graphs for every fixed $c$.
However, it has been known from \v{S}ir\'{a}\v{n}'s~\cite{vsiravn1984infinite}
and Kochol's~\cite{kochol1987construction} constructions that the structure of $c\,$-crossing-critical graphs
is quite rich for any $c\geq 2$, and already the first non-trivial case of $c=2$ shows a dramatic increase in
the complexity of the problem.  Bokal, Oporowski, Richter and Salazar recently succeeded in obtaining a
full description \cite{bokal2016characterizing} of all $2$-crossing-critical graphs 
up to finitely many small exceptions: They show that every sufficiently large $2$-crossing-critical graph of minimum degree at
least three is obtained by arranging any number of ``tiles'' from an explicit list of $42$ possibilities
in a M\"obius band fashion.

Our main result essentially shows that a characterization similar to~\cite{bokal2016characterizing}, i.e., every $c\,$-crossing-critical graph
consists of arbitrarily long band-like parts formed by concatenation of tiles connected through a central
subgraph of bounded size, holds for every fixed $c$.  However, unlike $c=2$, we do not provide an exact
description of the possible tiles. As we discuss later, there are complexity-theoretic reasons indicating that
a simple general description does not exist. In fact, given the increase in complexity already in the $c=2$
case, at the moment there does not seem to be any hope of extending the explicit description even to the $c=3$ case.

Thus, on a somewhat abstract level, we prove the following claims about sufficiently large $c\,$-crossing-critical graphs for any fixed $c$:
\begin{enumerate}[(S1)]
\item\label{it:twokinds}
There exist two kinds of local arrangements---bands and fans---such that any optimal drawing of
a sufficiently large $c\,$-crossing-critical graph contains at least one of them (Corollary~\ref{cor-tiles}).
Moreover, in the fan case, the crossing-criticality depends on the possibility to twist the drawing of the fan
analogously to the M\"obius band construction mentioned above (Observation~\ref{obs-fan-twisted}).
\item\label{it:reduction}
There are well-defined local operations (reductions) performed on such bands or fans that can
reduce any sufficiently large $c\,$-crossing-critical graph to one of (finitely many) ``basic'' $c$-cros\-sing-critical graphs
whose size is bounded by a function of $c$ (Corollary~\ref{cor-reducrit}).
\item\label{it:construction}
Importantly, the reductions are only applied in the case where the reduced parts
also appear elsewhere in the band or fan.
This property means that we can perform the converse---a well-defined bounded-size expansion
operation---to iteratively construct each $c$-cros\-sing-critical 
graph from a basic $c$-cros\-sing-critical graph of bounded size by only repeating
pieces that are already contained in the basic graph.

This yields a way to enumerate all the $c\,$-cros\-sing-critical 
graphs of at most given order $n$ in polynomial time per each generated graph (Theorem~\ref{thm-gen}).
More precisely, the total runtime is $O(n)$ times the output size.
\item\label{it:keepcritical}
Moreover, we show that the expansion operation actually preserves $c\,$-crossing-criticality (Corollary~\ref{cor-expcrit}).
This gives the desired structural characterization: The $c\,$-crossing-critical graphs are exactly those
that can be obtained from the basic ones with bounded number of vertices by expansions
which only replicate the previously present parts of the bands or fans (Theorem~\ref{thm:mainexpansion}).
Thus, the characterization can in principle be made precise for every fixed value of $c$: It suffices to list the basic $c\,$-crossing-critical
graphs of bounded size.  
\end{enumerate}

\begin{figure}
\centering\par
(a)\quad
\begin{tikzpicture}[scale=0.45]
\tikzstyle{every node}=[draw, shape=circle, minimum size=2.5pt,inner sep=0pt, fill=black]
\tikzstyle{every path}=[color=black]
\foreach \i in {0,...,3} {
\foreach \j in {1,-1} {
  \ifthenelse{\i<3 \OR \j=1}{
	\def\xx{6*\i+1.5-1.5*\j};
	\node at (\xx,1-\j) {};	\node at (\xx,1) {};
	\node at (\xx+1,1+\j) {};	\node at (\xx+2,1) {};
	\node at (\xx+2,1-\j) {};
 	\draw (\xx,1-\j) -- (\xx,1) -- (\xx+1,1+\j) -- (\xx+2,1) -- (\xx+2,1-\j);
 	\draw (\xx+2,1) -- (\xx+3,1);
  }{}
}}
\draw (-1,0) -- (21,0);	\draw (-1,1) -- (0,1);	\draw (-1,2) -- (21,2);
\foreach \i in {0,1} {
	\draw [dashed,-{Stealth[open]}] (24*\i-2,2*\i) -- (24*\i-2,1); 
	\draw [dashed,-{Stealth[open]}] (24*\i-2,1) -- (24*\i-2,2-2*\i); 
}
\end{tikzpicture}
\\(b)\quad
\begin{tikzpicture}[scale=0.52]
\useasboundingbox (1,-0.9) rectangle (22.03,4);
\tikzstyle{every node}=[draw, shape=circle, minimum size=2.5pt,inner sep=0pt, fill=black]
\tikzstyle{every path}=[color=black]
\foreach \i in {1,...,20} {
	\draw (\i,0) -- (\i+2,2);
	\draw (\i,2) -- (\i+2,0);
	\node at (\i+0.5,0.5) {};
	\node at (\i+1,1) {};
	\node at (\i+1.5,1.5) {};
}
\draw[color=white,fill=white] (0,0) rectangle (1.95,2);
\draw[color=white,fill=white] (21.05,0) rectangle (23,2);
\node (u) at (8,-1) {}; \node (u1) at (5,0) {}; \node (u2) at (12,0) {};
\node (v) at (13,3) {}; \node (v1) at (9,2) {}; \node (v2) at (16,2) {};
\draw[thick,dashed] (u) .. controls (10,0) and (11,2) .. (v);
\draw (u) to[bend right=12] (u2);
\draw (u) to[bend left=14] (u1);
\draw (v) to[bend left=12] (v2);
\draw (v) to[bend right=14] (v1);
\foreach \i in {1,22} {
	\draw [dashed,-{Stealth[open]}] (\i,0) -- (\i,1); 
	\draw [dashed,-{Stealth[open]}] (\i,1) -- (\i,2); 
}
\end{tikzpicture}
\caption{A schematic illustration of two basic methods of constructing crossing-critical graphs.
	(a)~The classical M\"obius-twist construction by Kochol~\cite{kochol1987construction}; note that the ends of the plane strip are joined together in a twisted way.
	(b)~An example construction in which the ends of a plane strip are joined together without a twist, but then a few added edges are forced to cross the strip.}
\label{fig:schemes}
\end{figure}

To give a more detailed (but still informal) explanation of these points, we
should review some of the key prior results.
First, the infinite $2$-crossing-critical family of Kochol~\cite{kochol1987construction}
explicitly showed one basic method of constructing crossing-critical
graphs---take a sequence of suitable small planar graphs (called \emph{tiles}, see~Section~\ref{sec:struc-tiles}),
concatenate them naturally into a plane strip and 
join the ends of this strip with the {\em M\"obius twist},
see the top part of Figure~\ref{fig:schemes} for an illustration.
Further constructions of this kind can be found, e.g., in
\cite{bokal2010infinite,pinontoan2003crossing,salazar2003infinite}.
In fact, \cite{bokal2016characterizing} essentially claims that such a
M\"obius twist construction is the only possibility for $c=2$;
there, the authors give an explicit list of $42$ tiles
which build in this way all the $2$-crossing-critical graphs
up to finitely many exceptions.

The second basic method of building crossing-critical graphs
was invented later by Hlin\v en\'y~\cite{hlinveny2002crossing};
it can be roughly described as constructing a suitable planar strip
whose ends are now joined without a twist (i.e., making a cylinder),
and adding to it a few edges which then have to cross the strip;
see the bottom part of Figure~\ref{fig:schemes} for an illustration.

A crossing-critical graph can also combine several smaller parts,
possibly arising from different constructions, or just being small ``sporadic''
crossing-critical graphs with no particular structure.
As an example, Bokal~\cite{bokal2010infinite} introduced
the so called {\em zip product} operation which combines two crossing-critical
graphs into a larger crossing-critical graph.

To complete the whole picture, a third method of building $c\,$-crossing-critical graphs was
discovered by Dvo\v r\'ak and Mohar in~\cite{dmcross}.  A more detailed analysis of this construction
is given in~\cite{bokal2022bounded}, where it is shown that it can be used only for $c\ge 13$.
The construction can be seen as a degenerate case of the M\"obius twist
construction where the whole strip shares a central high-degree vertex.

As we will see, the construction methods roughly represent
the local arrangements mentioned in \ref{it:twokinds}.
Hence, we can somewhat vaguely claim that no other method of constructing
infinite families of $c\,$-crossing-critical graphs is possible, for any fixed~$c$.

Moving on to statement~\ref{it:reduction},
we note that all three mentioned construction methods
involve long (and also ``thin'') planar strips, or {\em bands} as subgraphs 
(which degenerate into {\em fans} in the third kind of local arrangements;
see Definition~\ref{def:band-fan}).
In Corollary~\ref{cor-tiles}, we prove that such a long and ``thin'' planar 
band or fan must exist in any sufficiently large $c\,$-crossing-critical graph,
and we analyze its structure to identify elementary
connected tiles of bounded size forming the band.
We then argue that we can reduce repeated sections of the band (in the sense of Definition~\ref{def-redu})
while preserving $c\,$-crossing-criticality.

With respect to statements~\ref{it:construction} and \ref{it:keepcritical},
the converse expansion operation is described in Definition~\ref{def-expand}.
For a quick illustration, the simplest case of this expansion operation
is edge subdivision, that is replacing an edge with a path,
which clearly preserves $c\,$-crossing-criticality.

\subsection*{Paper organization}

We start with definitions and preliminary results about
crossing-critical graphs in Section~\ref{sec:prelim}.
In Section~\ref{sec:struc-tiles}, we focus on the structure of plane graphs of bounded path-width,
showing the existence of a long band or fan.
Then, in Section~\ref{sec:shorten-band}, we introduce the
reduction and expansion operations on bands and argue that they preserve crossing-criticality.
Next, we combine the results to give the structural theorem in Section~\ref{sec:main} and the algorithm
for generating the crossing-critical graphs in Section~\ref{sec:alg}.
Some final remarks are presented~in~Section~\ref{sec:conclusion}.

\section{Preliminaries}\label{sec:prelim}

We consider loopless multigraphs by default; i.e., throughout the paper,
graphs are allowed to have parallel edges, but not loops.  Let us remark that loops are
irrelevant when considering the crossing number, as we can always draw them without crossings,
and in particular, the crossing-critical graphs are loopless.
Similarly, parallel edges could be avoided (with a slight adjustment of definitions) by subdividing
them in order to make our graphs simple.  However, allowing parallel edges makes
some of our definitions simpler.  We generally follow the basic terminology of topological graph theory, see e.g.~\cite{mohthom}.

\subsection{Graph drawing and the crossing number}

A {\em drawing} of a graph $G$ in the plane consists of
\begin{itemize}
\item an injective function $\nu:V(G)\to \mathbb{R}^2$, assigning to each vertex a distinct point in the plane, and
\item a function $\varepsilon$ assigning to each edge $e=uv$ of $G$ a simple
curve $\varepsilon(e)$ from $\nu(u)$ to $\nu(v)$ and otherwise disjoint from the image of $\nu$,
\end{itemize}
such that for every point $p\in \mathbb{R}^2\setminus \nu(V(G))$, there exists at most two edges $e\in E(G)$
such that $p$ lies on the curve $\varepsilon(e)$; i.e., no three curves representing edges cross in a common point different from their ends.
A {\em crossing} is then any point of $\mathbb{R}^2\setminus \nu(V(G))$ lying on the curves $\varepsilon(e_1)$ and $\varepsilon(e_2)$
for two distinct edges $e_1,e_2\in E(G)$.
Each maximal connected subset of the complement of the drawing in the plane is a \emph{face} of the drawing;
let us remark that crossings may appear as part of the boundary of a face.
A drawing without crossings in the plane is called a \emph{plane drawing}.
A \emph{plane graph} is a graph with a fixed plane drawing.
A graph is \emph{planar} if it has a plane drawing.
If $H$ is a subgraph~of~$G$, a drawing $(\nu',\varepsilon')$ of $H$ is \emph{induced} by the drawing $(\nu,\varepsilon)$
if $\nu'$ is the restriction of $\nu$ to $V(H)$ and $\varepsilon'$ is the restriction of $\varepsilon$ to $E(H)$.

\begin{definition}[crossing number]
\label{def:crossingn}
The {\em crossing number} $\crg(G)$ of a graph $G$
is the minimum number of crossings of edges in a drawing of $G$ in the plane.
\end{definition}
Hence, a graph $H$ is planar if and only if~$\crg(H)=0$.
Let us remark that a $c\,$-crossing-critical graph may have no drawing
with precisely $c$ crossings (for example, the graph $C_3\times C_3$ is $2$-crossing-critical,
but has crossing number $3$).
The following result is classical.
\begin{theorem}[Richter and Thomassen~\cite{richter1993minimal}]\label{thm-crit-cno}
There exists a function $\fnlab{thm-crit-cno}:\mathbb{N}\to\mathbb{N}$ such that for every positive integer $c$,
every $c\,$-crossing-critical graph has crossing number at most $\fnlab{thm-crit-cno}(c)$.
\end{theorem}
Richter and Thomassen~\cite{richter1993minimal} proved this is true for the function $\fnlab{thm-crit-cno}(c)=\lceil 5c/2+16\rceil$.
This bound was recently improved by Barát and Tóth~\cite{barat2022improvement},
to $2c+8\sqrt{c}+47$.  Let us remark that the bound of Barát and Tóth~\cite{barat2022improvement} still seems to be far from optimal --
we do not know any $c\,$-crossing-critical graphs of crossing number more than $c+\Theta(\sqrt c)$.

Suppose that $\Gamma$ is a drawing of a graph $G$ in the plane with crossings, and let $G'$ be the
plane graph obtained from this drawing by replacing the crossings with new
vertices of degree $4$.  We say that $G'$ is the \emph{planarization of~$\Gamma$} (or of $G$, if the drawing is clear from the context), 
and the new vertices are the \emph{crossing vertices} of $G'$.

\subsection{2-connectivity}\label{ssec:2con}

The crossing number is additive over $(\le\!1)$-cuts.
More precisely, if $H$ is obtained from $H_1$, \ldots, $H_b$ by gluing over vertex cuts of size at most one, then
\begin{equation}\label{eq-addit}
\crg(H)=\crg(H_1)+\cdots+\crg(H_b),
\end{equation}
as can be seen by combining the drawings of $H_1$, \ldots, $H_b$ in the natural way so that they overlap
only on the cut vertices. More precisely, to show that this is possible, it is convenient to first observe that the circular inversion
of the plane can be used to modify any drawing of a graph $H$ so that a prescribed vertex $v\in V(H)$ is drawn in the boundary of the unbounded
region of the complement of the drawing. One can then choose a simple-closed curve $\gamma$ in the unbounded region so that the drawing
is contained in the closed interior of $\gamma$ and touches $\gamma$ exactly in this vertex $v$,
and then continuously deform the plane to turn $\gamma$ into an acute triangle with $v$ forming one of the vertices.
If $H'$ and $H''$ are two graphs intersecting in exactly
one vertex $v$, we can perform this transformation on the drawings of both $H'$ and $H''$, then simply move and glue the two triangles containing their drawings
so that they intersect exactly in $v$.

The additivity of the crossing number over $2$-connected blocks allows us to focus only on \emph{$2$-connected} crossing-critical graphs.
More precisely, we can argue about the validity of adding this constraint as follows.
For integers $c\ge 1$ and $c'\ge c$, let us say a graph is \emph{$(c,c')$-crossing-critical} if
it has crossing number exactly $c'$ and all proper subgraphs have crossing number less than $c$.
\begin{proposition}[folklore]\label{prop:2conn-cr}
Let $H$ be a graph and $c$ a positive integer.  The following claims are equivalent.
\begin{enumerate}[(i)]
\item\label{eqstmt-crit} The graph $H$ is $c\,$-crossing-critical.
\item\label{eqstmt-deltacrit} There exists an integer $c'$ such that $c\le c'\le \fnlab{thm-crit-cno}(c)$
and $H$ is $(c,c')$-crossing-critical.
\item\label{eqstmt-blocks} Letting $H_1$, \ldots, $H_b$ be the $2$-connected blocks of $H$
and letting $c'_i=\crg(H_i)$ for $i\in \{1,\ldots, b\}$, there exist
positive integers $c_1\le c'_1$, \ldots, $c_b\le c'_b$ such that
$$c\le c'_1+\ldots+c'_b\le c+\min\{c'_i-c_i:i\in \{1,\ldots,b\}\}$$ and for every $i\in \{1,\ldots,b\}$, the block $H_i$ is $c_i$-crossing-critical.
\end{enumerate}
\end{proposition}
\begin{proof}
If $H$ is $c\,$-crossing-critical, then it has crossing number $c'$ for some non-negative integer $c'\le \fnlab{thm-crit-cno}(c)$
by Theorem~\ref{thm-crit-cno}, and thus it is $(c,c')$-crossing-critical.

Suppose now that $H$ is $(c,c')$-crossing-critical for a non-negative integer $c'\le \fnlab{thm-crit-cno}(c)$.
For each block $H_i$ of $H$, let $c_i$ be the smallest positive integer strictly larger
than the crossing number of each proper subgraph of $H_i$.  Observe that $\crg(H_i)\ge c_i-1$, and that $H_i$
is $c_i\,$-crossing-critical if and only if $\crg(H_i)\ge c_i$.  Recall that $c'_i=\crg(H_i)$, and thus
we have $c'_i\ge c_i-1$, and $c'_i\ge c_i$ if and only if $H_i$ is $c_i\,$-crossing-critical.

By (\ref{eq-addit}), we have
$$c\le c'=\crg(H)=\sum_{j=1}^b \crg(H_j)=c'_1+\ldots+c'_b.$$
Consider any $i\in \{1,\ldots,b\}$ and an edge $e\in E(H_i)$ with $\crg(H_i-e)$ maximum;
the choice of $c_i$ implies that $\crg(H_i-e)=c_i-1$.
Therefore, (\ref{eq-addit}) gives
\begin{align*}
\crg(H-e)&=\crg(H_i-e)+\sum_{j\in \{1,\ldots,b\}\setminus\{i\}} \crg(H_j)\\
&=c'_1+\ldots+c'_b-c'_i+c_i-1=c'+c_i-c'_i-1.
\end{align*}
Since $H$ is $(c,c')$-critical, it follows that
$$c-1\ge \crg(H-e)=c'+c_i-c'_i-1,$$
and thus $c'_1+\ldots+c'_b=c'\le c+c'_i-c_i$, as required.
Moreover, $c\le c'\le c+c'_i-c_i$ implies that $c'_i\ge c_i$, and thus the block $H_i$ is $c_i$-crossing-critical.
Hence, all the conditions of \ref{eqstmt-blocks} hold.

Finally, suppose that the blocks $H_1$, \ldots, $H_b$ of $H$ satisfy \ref{eqstmt-blocks}
for some positive integers $c_1\le c'_1$, \ldots, $c_b\le c'_b$.
By (\ref{eq-addit}), we have
$$\crg(H)=\sum_{j=1}^b \crg(H_j)=c'_1+\ldots+c'_b\ge c.$$
Consider any edge $e\in E(H)$, contained in the block $H_k$ for some $k\in\{1,\ldots,b\}$.
Since $H_k$ is $c_k$-crossing-critical, we have $\crg(H_k-e)<c_k$, and thus by (\ref{eq-addit}),
\begin{align*}
\crg(H-e)&=\crg(H_k-e)+\sum_{j\in \{1,\ldots,b\}\setminus\{k\}} \crg(H_j)<c'_1+\ldots+c'_b-(c'_k-c_k)\\
&\le c+\min\{c'_i-c_i:i\in \{1,\ldots,b\}\}-(c'_k-c_k)\le c.
\end{align*}
We conclude that $H$ is $c\,$-crossing-critical.
\end{proof}
Hence, to obtain information on a $c\,$-crossing-critical graph $H$,
it suffices to study its $2$-connected blocks, where each such block $H_i$ is $c_i$-crossing-critical
for some $c_i\le \crg(H_i)\le \crg(H)\le \fnlab{thm-crit-cno}(c)$.
In conclusion, restricting ourselves only to $2$-connected crossing-critical graphs is without loss of generality.
Let us remark that our main result Theorem~\ref{thm:mainexpansion} actually gives a characterization of $2$-connected $(c,c')$-crossing-critical (rather than just $c$-crossing-critical) graphs, so this restriction has to be clarified.

Let us also note the following useful fact.

\begin{observation}\label{obs-planarization-2connected}
Let $G$ be a 2-connected graph drawn in the plane optimally, i.e., with exactly $\crg(G)$ crossings,
and let $G'$ be the planarization of $G$.  Then $G'$ is $2$-connected.
\end{observation}

\begin{proof}
Suppose for a contradiction that $G'$ has a cutvertex $v$, and let $x'_1$ and $x'_2$ be vertices of $G'$ contained in
different components of $G'-v$.

For $i\in \{1,2\}$, if $x'_i$ is a crossing vertex of $G'$, then let $x_i$ be an
edge~of~$G$ passing through the crossing; otherwise, let $x_i=x'_i$ be the corresponding vertex~of~$G$.  Since $G$ is $2$-connected, it contains
a cycle $K$ passing through $x_1$ and $x_2$, and the corresponding walk $K'$ in $G'$ must pass through $v$ twice.
We conclude that $v$ is a crossing vertex. Observe that the neighbors of $v$ in $G'$ are four distinct vertices of $G$, since adjacent edges do not cross in optimal drawings.

Let $G'_1$ and $G'_2$ be the components of $G'-v$ containing $x'_1$
and $x'_2$, respectively.  Since $G'$ is a plane graph, $G'_1$ and $G'_2$ are disjointly drawn plane graphs.  
Moreover, since $K'$ passes through $v$ twice, the vertex $v$ has two neighbors in $G'_1$ and two neighbors in $G'_2$.
Redraw $G'_2$ in the same
face of $G'_1$ but in a mirrored way, and add a matching of size two between the neighbors of $v$ in $G'_1$ and $G'_2$,
where the matching is drawn in a non-crossed way.  The resulting graph $G''$ is the planarization of a drawing of $G$ with fewer than $\crg(G)$ crossings, which
is a contradiction.
\end{proof}

\subsection{Forbidden structures for crossing-critical graphs}

Structural properties of crossing-critical graphs have been studied for more
than two decades, and we now briefly review 
some of the previous important results which we shall use.
Our approach to dealing with ``long and thin'' subgraphs
in crossing-critical graphs relies on the structural notion 
of \emph{path-width} of a graph.

\begin{definition}[path decomposition and path-width]\label{def:path-width}
A \emph{path decomposition} of a graph $G$ is a pair $(P,\beta)$, where $P$
is a path and $\beta$ is a function that assigns pairwise edge-disjoint subgraphs of $G$,
called \emph{bags}, to the nodes of $P$, such that
\begin{itemize}
\item for each edge $e\in E(G)$, there exists $x\in V(P)$ such that
$e\in E(\beta(x))$, and
\item for every $v\in V(G)$, the set $\{x\in V(P):v\in V(\beta(x))\}$ induces a non-empty connected subpath of $P$.
\end{itemize}
The \emph{width} of the decomposition is the maximum of $|V(\beta(x))|-1$ over all
nodes $x$ of $P$, and the \emph{path-width} of $G$ is the minimum width over all path
decompositions of~$G$.
\end{definition}
Let us remark that we refer to the vertices of the path $P$ as nodes to make it clear
they are distinct from the vertices of $G$.
It was proved that $c\,$-crossing-critical
graphs have path-width bounded in terms of~$c$.

\begin{theorem}[Hlin\v en\'y~\cite{Hl}]\label{thm:bounded-pw}
There exists a function $\fnlab{thm:bounded-pw}:\mathbb{N}\to\mathbb{N}$ such that
for every positive integer $c$, every $c\,$-crossing-critical graph has path-width at
most~$\fnlab{thm:bounded-pw}(c)$.
\end{theorem}

Hence, every crossing-critical graph can be seen as having a ``linear'' structure.
Another point of view is as follows: A graph is known to have bounded
path-width if and only if it avoids subdivisions of large complete binary
trees~\cite{bienstock1991quickly}.  Thus, subdivisions of large complete binary
trees are forbidden in crossing-critical graphs.

A concept important both for the proof of Theorem~\ref{thm:bounded-pw}
and for this work is that of \emph{nests} in a drawing of a graph; see Figure~\ref{fig:nests} for an illustration of the following definition.

\begin{figure}
\begin{center}
\includegraphics[width=0.67\textwidth]{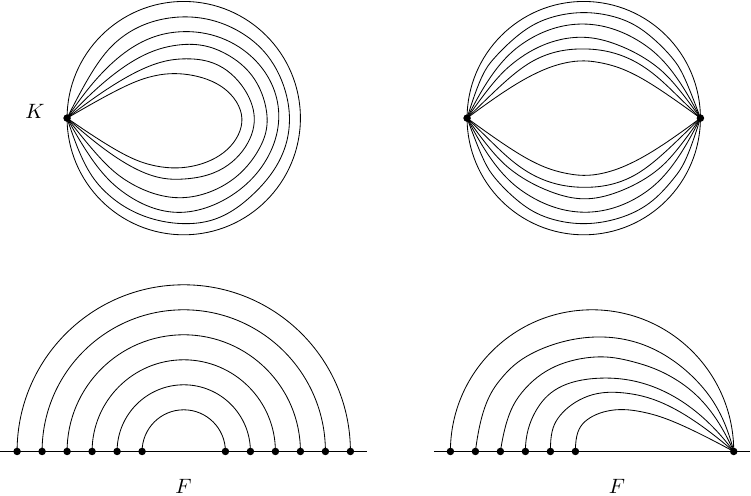}
\end{center}
\caption{An illustration of Definition~\ref{def:nests}:
	a $1$-nest, a $2$-nest, a proper $F$-nest, and a degenerate $F$-nest, each of depth~$6$.}
\label{fig:nests}
\end{figure}

\begin{definition}[nests]
\label{def:nests}
Let $G$ be a $2$-connected plane graph.  For an integer $k\ge 0$, a
\emph{$k$-nest} in $G$ of \emph{depth} $m$ is a sequence $(C_1, C_2,
\ldots, C_m)$ of pairwise edge-disjoint cycles such that for some set $K$ of
$k$ vertices and for every $i<j$, the cycle $C_i$ is drawn in the closed
disk bounded by $C_j$ and $V(C_i)\cap V(C_j)=K$.

Let $F$ be a face of $G$ and suppose that $v_1,v_2,\ldots,v_{2m}$ are some of
the vertices incident with~$F$ listed in the cyclic order along the face. 
Let $P_1$, \ldots, $P_m$ be pairwise vertex-disjoint paths in $G$ such that
for $1\le i\le m$, the path $P_i$ joins $v_i$ with $v_{2m+1-i}$.  Then, we say that
the sequence $(P_1, \ldots, P_m)$ forms a \emph{proper $F$-nest of depth $m$}.

Similarly, suppose that $v_1$, $v_2$, \ldots, $v_m$ are some of
the vertices incident with $F$ listed in the cyclic order along the face,
$u$ is a vertex incident with $F$ different from $v_1$, \ldots, $v_m$, and $P_1$, \ldots, $P_m$ are paths in $G$ such that
for $1\le i\le m$, the path $P_i$ joins $v_i$ with $u$.  If the paths pairwise intersect only in $u$, then we say
that the sequence $(P_1,\ldots, P_m)$ forms a \emph{degenerate $F$-nest of depth~$m$}.
\end{definition}
The nests were implicitly considered in \cite{Hl,HS}.
In particular, a 2-nest corresponds to two vertices joined by
a large number of paths with pairwise disjoint interiors, which was shown
not to occur in a crossing-critical graph by Hlin\v en\'y and Salazar~\cite{HS}.
The nests were explicitly defined by Hernandez-Velez et al.~\cite{crnest}
who concluded that no optimal drawing of a $c\,$-crossing-critical graph 
can contain a $0$-, $1$-, or $2$-nest whose depth is large compared to~$c$.
\begin{theorem}[Hernandez-Velez et al.~\cite{crnest}]\label{thm-nestdepth}
There exists a function $\fnlab{thm-nestdepth}:\mathbb{N}\to\mathbb{N}$ such that the following claim holds.
Let $c$ be a positive integer and let $G$ be a $2$-connected $c\,$-crossing-critical graph.
Let $G'$ be the planarization of an optimal drawing of $G$ and let $Y$ be the
set of its crossing vertices.  Then every $0$-, $1$-, or $2$-nest
in $G'$ disjoint from $Y$ has depth at most $\fnlab{thm-nestdepth}(c)$.
\end{theorem}
Let us remark that it suffices to take $\fnlab{thm-nestdepth}(c)=15c^2+105c+16$.
We now similarly exclude the existence of both proper and degenerate $F$-nests of large depth in crossing-critical graphs.

\begin{lemma}\label{lemma-fnest}
Let $G$ be a $2$-connected $c\,$-crossing-critical graph, let $\Gamma$ be a drawing of $G$ in the plane with the smallest number of crossings, and let $F$ be a face of this drawing.
Let $G'$ be the planarization of $\Gamma$ and let $Y$ be the set of its crossing vertices.
Then every proper or degenerate $F$-nest in $G'$ disjoint from $Y$ has depth at most $\fnlab{lemma-fnest}(c)=6c(\fnlab{thm-crit-cno}(c)+1)$.
\end{lemma}
\begin{proof}
Let
\begin{equation}\label{eq-defk_0}
k_0=\fnlab{thm-crit-cno}(c),
\end{equation}
so that by Theorem~\ref{thm-crit-cno} the drawing $\Gamma$ of $G$ has
at most $k_0$ crossings.  Let $(P_1, \ldots, P_t)$ be a proper or degenerate $F$-nest in $G'$ disjoint from $Y$;
since the paths of the $F$-nest do not contain the crossing vertices, they are also paths in $G$.
By performing a circular inversion of the plane if necessary, we can assume that $F$ is the outer face of the drawing $\Gamma$.
Moreover, we can without loss of generality assume that the starting vertices $v_1$, \ldots, $v_t$ of the paths $P_1$, \ldots, $P_t$
appear on the boundary of $F$ in the clockwise order.  In case that the $F$-nest is degenerate, let $u$ denote the common vertex of all the paths $P_1$, \ldots, $P_t$;
note that by starting the labeling of the paths of the $F$-nest in the first starting vertex after $u$ in the clockwise order along $F$, we can
without loss of generality assume that the clockwise order of the starting points and $u$ along $F$ is $v_1$, \ldots, $v_t$, $u$.

For any indices $1\le i_1<i_2\le t$, let $G_{i_1,i_2}$ be the subgraph of $G$ drawn between $P_{i_1}$ and $P_{i_2}$
(inclusive). More precisely, let $\gamma_{i_1,i_2}$ be a simple closed curve tracing $P_{i_1}$, then the boundary of $F$
in the counterclockwise order until it hits the end of $P_{i_2}$, then the reverse of $P_{i_2}$, then going counterclockwise
inside~$F$ infinitesimally close to its boundary until it reaches the starting point of $P_{i_1}$.  Let $\Delta_{i_1,i_2}$
be the closed disk bounded by $\gamma_{i_1,i_2}$, and let $G_{i_1,i_2}$ be the subgraph~of~$G$ drawn in $\Delta_{i_1,i_2}$.
Observe that if $1\le i_1<i_2\le i_3<i_4\le t$, then the subgraphs $G_{i_1,i_2}$ and $G_{i_3,i_4}$
do not share any crossing: If $i_2=i_3$, then they intersect in the path $P_{i_2}$ which does not contain any crossing,
and otherwise they are either disjoint or intersect exactly in the common vertex $u$ of all the paths, depending on
whether the $F$-nest is proper or degenerate.

Let $a$ and $b$ be indices such that $1\le a<b\le t$, the subgraph $G_{a,b}$ contains no crossings, and $q=b-a+1$ is maximum.
Let $i_j=1+(j-1)q$.  We claim that $i_{k_0+2}>t$; otherwise, the choice of $q$ would imply that the
subgraphs $G_{i_j,i_{j+1}}$ for $j\in \{1,\ldots, k_0+1\}$ together contain more than $k_0$ distinct crossings.
Therefore,
\begin{equation}\label{ineq-boundt}
t\le i_{k_0+2}-1=(k_0+1)q.
\end{equation}
Consider any index $i$ such that $a+1\le i\le b-2$.  A minor technical issue for the argument below is that
the graph $G-V(G_{i,i+1})$ may contain, in addition to the two components containing $P_{i-1}$ (or $P_{i-1}-u$) and $P_{i+2}$
(or $P_{i+2}-u$), further components only attaching to $P_i$ and $P_{i+1}$.
We deal with this issue by defining a supergraph $G_i$ of $G_{i,i+1}$ obtained by ``maximally pushing the paths $P_i$ and $P_{i+1}$
outwards, while fixing their endpoints'', thus including the parts of $G_{i-1,i}$ and $G_{i+1,i+2}$ that only attach to
$P_i$ and $P_{i+1}$, respectively.  More precisely, let $R_i=V(P_{i-1}\cup P_{i+2})$ if the considered $F$-nest is proper, and
let $R_i=V(P_{i-1}\cup P_{i+2})\setminus\{u\}$ if it is degenerate.
Let $G_i$ be the maximal $2$-connected subgraph of $G-R_i$ containing
$G_{i,i+1}$ (let us remark that the subgraph $G_{i,i+1}$ itself is $2$-connected, since $G$ is $2$-connected, and thus $G_i$ exists).
Observe that $G-V(G_i)$ has exactly two components, one containing $P_{i-1}$ (or $P_{i-1}-u$) and the other one containing $P_{i+2}$
(or $P_{i+2}-u$).  Let $C_i$ be the cycle bounding the outer face of $G_i$.  

Let $m=\lfloor \tfrac{q-1}{6}\rfloor$, so that $6m+1\le q\le 6(m+1)$.
If $m\le c-1$, then by (\ref{eq-defk_0}) and (\ref{ineq-boundt}), we obtain 
$$t\le (k_0+1)q\le 6(k_0+1)(m+1)\le 6c(\fnlab{thm-crit-cno}(c)+1),$$
matching the conclusion of this lemma.  Hence, for contradiction suppose that $m\ge c$.

Consider the cycles $C_{a+3i-2}$ for $1\le i\le 2m$; these cycles are defined, since $6m+1\le q$.
Let $e$ be an edge of $P_{a+3m}$.  Since $G$ is crossing-critical, the graph $G-e$ has a drawing $\Gamma'$ with $\crg(G-e)<c$
crossings. Since $m\ge c$, there exist indices $x\in \{1,\ldots,m\}$ and $y\in\{m+1,\ldots,2m\}$ such that there are no crossings
on the edges of the cycles $K_1=C_{a+3x-2}$ and of $K_2=C_{a+3y-2}$ in $\Gamma'$.

Let $H_1$ be the subgraph of $G$ consisting of $K_1$, the component 
$Z_1$ of $G-V(G_{a+3x-2})$ that does not contain $e$, and of the edges between them.  
Let $H_2$ be the subgraph of $G$ consisting of $K_2$ and the component 
$Z_2$ of $G-V(G_{a+3y-2})$ that does not contain $e$ and of the edges between them.
Let~$K$ be the cycle consisting of a path in $K_1$, a path in $K_2$ and
of two subpaths of the boundary~of~$F$ whose interior in the drawing $\Gamma$ of $G$ is disjoint from $K_1\cup K_2$;
if the $2$-nest is degenerate, then one of the two subpaths consists just of the vertex $u$.

Consider the drawings $\Gamma'_1$ and $\Gamma'_2$ of $H_1$ and $H_2$ induced by the drawing $\Gamma'$ of $G-e$.
For $i\in\{1,2\}$, recall that the subgraph $Z_i$ is connected and that no edge of the cycle $K_i$ is crossed in the drawing $\Gamma'$ of $G-e$,
and thus by performing the circular inversion of the plane
if necessary, we can modify the drawing $\Gamma'_i$ so that $Z_i$ is drawn outside of $K_i$ in $H_i$,
and moreover, so that the path $K\cap K_i$ is incident with the outer face of $H_i$.  Finally, we can continuously
deform the plane so that $K_i$ is drawn in exactly the same position as in the drawing $\Gamma$ of $G$, and so that the drawings
of $Z_1$ and $Z_2$ are disjoint from each other and from the closed disk bounded by $K$ in the drawing $\Gamma$ of $G$.
Let $\Gamma''$ be the resulting drawing of $H_1\cup H_2$.

Let $G_0$ be the subgraph of $G$ drawn in the closed disk bounded by $K$ in the drawing $\Gamma$.
Let $\Gamma'''$ be the drawing of the subgraph $G_{a+3x-2}\cup G_{a+3y-2}\cup G_0$ induced by $\Gamma$.
Then the combination of $\Gamma''$ and $\Gamma'''$ is a drawing of $G$ such that each crossing belongs to $H_1$ or $H_2$,
and thus corresponds to a distinct crossing in the drawing $\Gamma'$ of $G-e$.
Therefore, this drawing has at most $\crg(G-e)<c$ crossings, which contradicts the assumption that $G$ is $c\,$-crossing-critical.
\end{proof}

Finally, let us note the following well-known observation.
\begin{observation}\label{obs-multip}
Every edge of a $c$-crossing-critical graph $G$ has multiplicity at most $c$.
\end{observation}
\begin{proof}
Suppose for a contradiction that $e_1, \ldots, e_{c+1}$ are parallel edges of $G$,
and let $G'$ be the subgraph of $G$ obtained by deleting $e_{c+1}$.  Since $G$ is $c$-crossing-critical,
$G'$ has a drawing in the plane with less than $c$ crossings, and in particular there exists
$i\in \{1,\ldots,c\}$ such that the edge $e_i$ does not cross any edge of $G'$ in this drawing.
But then we can draw all edges $e_1$, \ldots, $e_{c+1}$ along $e_i$ so that there are no crossings
on them, obtaining a drawing of $G$ with less than $c$ crossings.
\end{proof}

\subsection{Framed graphs and their properties}

We will now consider several structural properties of the subgraphs induced by
the bags of a path decomposition of a plane graph, where the properties can refer to the
``boundary'' vertices shared with two consecutive bags on the path.
We are only going to need this in the case that these intersections all have the same size,
motivating the following definition.

A \emph{framed graph} is a graph $G$ together with tuples $L_0(G)=(l_1,\ldots,l_a)$, $R_0(G)=(r_1,\ldots,r_a)$
and $C(G)=(c_1,\ldots,c_b)$ of vertices of $G$, where the $2a+b$ vertices in these tuples are all distinct.  We say that $\sigma=(a,b)$ is the \emph{signature} of $G$ and define $|\sigma|=a+b$.
Let $L(G)$ be the concatenation of $L_0(G)$ and $C(G)$, and let $R(G)$ be the concatenation of $R_0(G)$ and $C(G)$.
We view $G$ as a graph with the boundary vertices $L(G)\cup R(G)$ divided into the left part $L(G)$ and the right part $R(G)$,
where the vertices of $C(G)$ are shared by both parts  of the boundary.
Framed graphs $G_1$ and $G_2$ are \emph{isomorphic} if there exists an isomorphism of $G_1$ and $G_2$
mapping $L(G_1)$ to $L(G_2)$ and $R(G_1)$ to $R(G_2)$.
A \emph{property} of framed graphs is a class $\RR$ of framed graphs
closed under isomorphisms.  A \emph{drawing-sensitive property} of plane framed graphs
is a class $\RR$ of plane framed graphs (not necessarily closed under homeomorphisms of the plane).
We say that the (plane) framed graphs in $\RR$ \emph{satisfy} the property $\RR$.

As an example, ``the first vertex of $L_0(G)$ and the first vertex of $R_0(G)$ are joined in $G$ by a path whose internal vertices
do not belong to $L(G)\cup R(G)$'' specifies a property of framed graphs, and ``there exists a cycle $C$ in $G$ such that the first
vertex of $L_0(G)$ is drawn in the open disk bounded by $C$ and the first vertex of $R_0(G)$ is drawn outside of the closed disk bounded by $C$''
specifies a drawing-sensitive property of plane framed graphs.

We are going to need the natural composition operation on (plane) framed graphs:  If $G_1$ and $G_2$
are framed graphs of the same signature such that $C(G_1)=C(G_2)$, $R_0(G_1)=L_0(G_2)$, and $G_1$ and $G_2$ only intersect in the vertices of $R(G_1)$ and $L(G_2)$,
then we say that $G_1$ and $G_2$ are \emph{compatible}; in case that $G_1$ and $G_2$ are plane framed graphs,
we additionally require that their drawings are disjoint except for the points representing the vertices of $R(G_1)=L(G_2)$.
The \emph{composition} $G_1+G_2$ of compatible (plane) framed graphs $G_1$ and $G_2$ is the (plane) framed graph $G=G_1\cup G_2$ with $L_0(G)=L_0(G_1)$, $R_0(G)=R_0(G_2)$,
and $C(G)=C(G_1)=C(G_2)$, where if $G_1$ and $G_2$ are plane framed graphs, then the drawing of $G$ is the union of the drawings of $G_1$ and $G_2$.

For convenience, let us also introduce a slight variation on the composition operation which does not require
$G_1$ and $G_2$ to share vertices:  Suppose $G_1$ and $G_2$ are 
framed graphs with the same signature.  Let $G'_2$ be a framed graph isomorphic to $G_2$ such that $L(G'_2)=R(G_1)$
and otherwise disjoint from $G_1$.  We then let $G_1\oplus G_2 = G_1+G'_2$; hence, the result of the operation
$\oplus$ is only determined uniquely up to isomorphism.

In general, the information whether $G_1$ and $G_2$ satisfy a property $\RR$ is not sufficient
to determine whether their composition $G_1+G_2$ satisfies $\RR$ or not.  However,
this is often the case if we maintain some additional information about $G_1$ and $G_2$.  
More precisely, let $A$ be a set and let $\circ: A^2\to A$ be a binary operation on $A$. Suppose that $f$ is a function assigning to each (plane)
framed graph $G$ a value $f(G)\in A$. The function $f$ \emph{determines} a (drawing-sensitive) property
$\RR$ if there exists a subset $A_\RR\subseteq A$ such that a (plane) framed graph $G$ satisfies the property $\RR$
if and only if $f(G)\in A_\RR$.  
We say that the function $f$ is \emph{composable via $\circ$} if $f(G_1+G_2)=f(G_1)\circ f(G_2)$
holds for all compatible (plane) framed graphs $G_1$ and $G_2$.

In the described situation, the values of $f$ on $G_1$ and $G_2$ determine the value on their composition $G_1+G_2$,
and thus also whether $G_1+G_2$ satisfies the property $\RR$ or not.  Of course, this
is only interesting if the set $A$ is small, e.g., its size is bounded by a function of $|\sigma|$ for the common signature $\sigma$ of $G_1$ and $G_2$; otherwise,
$f$ could simply be chosen as the identity function and $\circ$ as the composition of the framed graphs.

A pair $(A,\circ)$, where $\circ$ is a binary operation on $A$, is a \emph{semigroup}
if the operation $\circ$ is associative.  A function $f$ assigning to framed graphs values from a semigroup $A$ is \emph{isomorphism-invariant}
if $f(G)=f(G')$ for any two isomorphic framed graphs $G$ and $G'$. 
Let us note the following easy observation. 

\begin{observation}\label{obs-semi-abstract}
Let $\sigma$ be a pair of non-negative integers and let $f$ be a surjective isomorphism-invariant function
assigning to each framed graph of signature $\sigma$ a value from a set $A$.
If $f$ is composable via a binary operation~$\circ$, then $(A,\circ)$ is a semigroup.
\end{observation}
\begin{proof}
Consider any $a_1,a_2,a_3\in A$.  Since $f$ is surjective, for $i\in\{1,2,3\}$, there exists a framed graphs $G_i$ of signature $\sigma$
such that $f(G_i)=a_i$.  Since $f$ is isomorphism-invariant, we can furthermore assume
that $R(G_1)=L(G_2)$, $R(G_2)=L(G_3)$, and the graphs $G_1$, $G_2$, and $G_3$ are otherwise disjoint.
Thus, $G_1$ and $G_2$, $G_2$ and $G_3$, $G_1$ and $G_2+G_3$, and $G_1+G_2$ and $G_3$ are compatible.
Let $G=G_1+G_2+G_3$; the composition is associative, so the order of evaluation does not matter.
Since $f$ is composable via $\circ$, we conclude that
$$(a_1\circ a_2)\circ a_3=f((G_1+G_2)+G_3)=f(G)=f(G_1+(G_2+G_3))=a_1\circ(a_2\circ a_3).$$
It follows that the operation $\circ$ is associative, and thus $(A,\circ)$ is a semigroup.
\end{proof}

It seems at first clear that the same should hold for homeomorphism-invariant functions on plane framed graphs,
where the notion of homeomorphism-invariantness is defined analogously to isomorphism-invariantness.
However, in this case it is not clear that we can choose (say) the plane framed graphs $G_1$ and $G_2$
to be compatible, since the compatibility depends on their plane drawings.
To sidestep this issue, we only consider drawing-sensitive properties that can be expressed in terms of the properties of
auxiliary (non-plane) framed graphs, allowing us to use Observation~\ref{obs-semi-abstract}.
More precisely, an \emph{abstractifier} for a drawing-sensitive property $\RR$ of plane framed graphs of fixed signature $\sigma$ is
a 5-tuple $(\rho,f,A,\circ,A_\RR)$, where
\begin{itemize}
\item $\rho$ is a function mapping each plane framed graph of signature $\sigma$ to a framed graph of the same signature,
\item $f$ is a surjective isomorphism-invariant function on framed graphs of signature $\sigma$ with range $A$, composable via the binary operation $\circ$, 
\item $f(\rho(G_1+G_2))=f(\rho(G_1)\oplus\rho(G_2))$ for all compatible plane framed graphs $G_1$ and $G_2$ of signature $\sigma$, and
\item a plane framed graph $G$ of signature $\sigma$ has property $\RR$ if and only if $f(\rho(G))\in A_\RR$.
\end{itemize}
Using Observation~\ref{obs-semi-abstract}, we then straightforwardly obtain the following conclusion.

\begin{corollary}\label{cor-semi}
Let $\sigma$ be a pair of non-negative integers, and let $(\rho,f,A,\circ,A_\RR)$ be an abstractifier for a drawing-sensitive property
$\RR$ of plane framed graphs of signature $\sigma$.  Let $f'$ be the composition of the functions $\rho$ and $f$.
Then $f'$ is a function from plane framed graph of signature $\sigma$ to $A$ composable via $\circ$, $f'$ determines $\RR$, and $(A,\circ)$ is a semigroup.
\end{corollary}

If $\Theta$ is a subgraph of a graph $H$, a \emph{$\Theta$-bridge} of $H$ is either an edge of $H$
not belonging to $\Theta$ and with both ends in $\Theta$, or a connected component of $H-V(\Theta)$ together with all
the edges from this component to $\Theta$; and the \emph{attachments} of a $\Theta$-bridge are its
vertices in $\Theta$.  Moreover, for a set of vertices $X\subseteq V(H)$, an \emph{$X$-bridge} is a $\Theta$-bridge for the
edgeless graph with vertex set $X$.
We are going to consider the following properties, for a plane framed graph $G$ with a fixed signature $\sigma$.

\begin{enumerate}[(P1)]
\item\label{prop:bridge} For vertices $x,y\in L(G)\cup R(G)$, the property ``there exists an $(L(G)\cup R(G))$-bridge of $G$ containing both $x$ and $y$;
i.e., $x$ and $y$ are connected by a path in $G$ whose internal vertices do not belong to $L(G)\cup R(G)$''.
An abstractifier for this property in plane framed graphs of signature $\sigma=(p,q)$ can be defined as follows.
\begin{itemize}
\item $\rho$ is the function that just forgets the drawing of the plane framed graph $G$.
\item Let us fix pairwise disjoint tuples $L_0=(l_1,\ldots, l_p)$, $R_0=(r_1,\ldots,r_p)$, and $C=(c_1,\ldots, c_q)$ of vertices and
let $U$ be the set of vertices in $L_0$, $R_0$, and $C$.  Let $A$ be the set of all (up to isomorphism)
bipartite framed graphs $H'$ with $U\subseteq V(H')$, $L_0(H')=L_0$, $R_0(H')=R_0$, and $C(H')=C$, such that
both $U$ and $Z=V(H')\setminus U$ are independent sets in $H'$ and the vertices of $Z$ have pairwise different neighborhoods in $U$.
In particular, we have $|A|=2^{2^{2p+q}}$.
\item For a framed graph $H$ of signature $\sigma$, let $\pi$ be the function mapping the vertices of $L_0$ to $L_0(H)$,
the vertices of $R_0$ to $R_0(H)$, and the vertices of $C$ to $C(H)$ in order.  Let $f(H)=H'$,
where $H'\in A$ is the graph such that for every $U'\subseteq U$, $H'$ has a vertex $z\in Z$ with neighborhood~$U'$ if and only
of $H$ contains an $L(H)\cup R(H)$-bridge with attachments $\pi(U')$.  
\item Let $A_\RR$ consist of the set of the framed graphs $H\in A$ such that
there exists an $(L(H)\cup R(H)$-bridge of $H$ containing both $\pi^{-1}(x)$ and $\pi^{-1}(y)$.
\item For framed graphs $a_1,a_2\in A$, we define $a_1\circ a_2=f(a_1\oplus a_2)$.
\end{itemize}
\item\label{prop:contY} For a set $Y$ of vertices (of a supergraph of $G$), the property ``$Y\cap V(G)\subseteq C(G)$''.
The function $\rho$ just forgets the drawing.  We let $A=\{0,1\}$, and for a framed graph $H$, we let $f(H)=1$
if $Y\cap V(H)\subseteq C(H)$.  We let $A_\RR=\{1\}$.  Finally, we define $a_1\circ a_2=a_1a_2$.

\item\label{prop:sepcyc} Let $x$ and $y$ be distinct points in the plane, and let $\gamma$ be a simple curve between them.
We say that $\gamma$ is \emph{$G$-normal} if the drawing of $G$ is disjoint from $x$ and $y$,
$\gamma$ does not pass through vertices of $G$, $\gamma$ intersects the drawing of any edge $e$ of $G$ at most once,
and if it does intersect it once, then it does so transversally (and thus passes from one side of the edge to the other side).
We next consider the property ``$\gamma$ is $G$-normal and  $G$ contains a cycle separating $x$ from $y$''.
Note that the assumption of $G$-normality is just a technicality that simplifies the definition of the abstractifier.

The abstractifier is based on the following observation:  Suppose that $\gamma$ is $G$-normal and let $G'$ be the graph
obtained from $G$ by subdividing each edge not intersecting $\gamma$ once.  Then $G$ contains a cycle separating $x$ from $y$
if and only if $G'$ contains an odd-length cycle.  Indeed, a cycle~in~$G$ separates $x$ from $y$ if and only if it intersects $\gamma$
odd number of times, which is equivalent to the corresponding cycle in $G'$ having an odd length.
Moreover, $G'$ contains an odd-length cycle if and only if it contains an odd-length closed walk.
An abstractifier for this property thus can be defined as follows.
\begin{itemize}
\item Similarly to the property~\ref{prop:bridge},
let us fix pairwise disjoint tuples $L_0=(l_1,\ldots, l_p)$, $R_0=(r_1,\ldots,r_p)$, and $C=(c_1,\ldots, c_q)$ of vertices and
let $U$ be the set of vertices in $L_0$, $R_0$, and $C$. Let $A'$ be the set consisting of a special element $\bot$
and of all (up to isomorphism) framed graphs $G'$ with $U\subseteq V(G')$, $L_0(G')=L_0$, $R_0(G')=R_0$, and $C(G')=C$, such that
$Z=V(G')\setminus U$ is an independent set of vertices of degree two with pairwise different neighborhoods in $U$.
Let us remark that $G'$ can also contain edges between vertices of $U$.
In particular, we have $|A'|=1+2^{2\binom{2p+q}{2}}$.
\item Let $\pi$ be as before.  We define $\rho(G)=\bot$ if $\gamma$ is not $G$-normal.  Otherwise,
let $G'$ be the framed graph obtained from $G$ by subdividing once every edge that does not cross $\gamma$. 
Then, let $\rho(G)$ be the graph in $A'$ such that distinct vertices $u,v\in U$ are adjacent precisely when
$G'$ contains an odd-length walk from $\pi(u)$ to $\pi(v)$, and have a common neighbor in $Z$
if and only if $G'$ contains an even-length walk from $\pi(u)$ to $\pi(v)$.
\item The function $f$ is defined analogously, letting $G'=G$.  We let $A\subseteq A'$ be the image of $f$,
so that $f$ is surjective.
\item Let $A_\RR$ consist of the framed graphs in $A$ containing an odd-length cycle.
\item For framed graphs $a_1,a_2\in A$, we define $a_1\circ a_2=f(a_1\oplus a_2)$.  For every $a\in A$,
we define $a\circ\bot=\bot\circ a=\bot$.
\end{itemize}
\end{enumerate}

Let us note that given composable functions determining different properties, we can
naturally take their product, obtaining the following conclusion.

\begin{observation}\label{obs-proprod}
Let $a$ and $b$ be non-negative integers.  For $i\in \{1,\ldots,k\}$, let $\RR_i$ be a drawing-sensitive property of
plane framed graphs of signature $\sigma$ and let $f_i$ be a function with range $A_i$ determining $\RR_i$,
composable via an operation $\circ_i$.
Let $A=A_1\times \cdots\times A_k$ and $\circ=\circ_1\times \cdots\times \circ_k$,
and let $f$ be the function that to each plane framed graph $G$ of signature $\sigma$ assigns
the element $(f_1(G),\ldots,f_k(G))$ of $A$.  Then $f$ is a function with range $A$ determining all of the drawing-sensitive properties
$\RR_1$, \ldots, $\RR_k$ and composable via $\circ$.
\end{observation}
Of course, if $(A_1,\circ_1)$, \ldots, $(A_k,\circ_k)$ are semigroups, then so is $(A,\circ)$.
Finally, let us make a simple observation on a property implied by \ref{prop:bridge}.
\begin{observation}\label{obs-bridge-to-comp}
Let $G_1$ and $G_2$ be framed graphs with the same signature.  For each vertex $x\in L(G_1)\cup R(G_1)$, let $x'$ denote the
vertex of $L(G_2)\cup R(G_2)$ at the same position in the tuples.  Suppose that every pair $x,y\in L(G_1)\cup R(G_1)$ of vertices
satisfies \ref{prop:bridge} in $G_1$ if and only if the vertices $x'$ and $y'$ satisfy \ref{prop:bridge} in $G_2$.
If vertices $u,v\in L(G_1)\cup R(G_1)$ are joined in $G_1$ by a path with internal vertices not belonging to $C(G_1)$,
then the vertices $u'$ and $v'$ are joined in $G_2$ by a path with internal vertices not belonging to $C(G_2)$.
\end{observation}

\subsection{Linked and framed path decompositions}

Let us now introduce a few definitions concerning path decompositions, with the aim of applying the notions
presented in the previous section to their bags.

Let $(P,\beta)$ be a path decomposition of a graph $G$.  The \emph{order} of the decomposition is $|V(P)|$.
The path decomposition is \emph{proper} if $V(\beta(x))\not\subseteq V(\beta(y))$ holds for all distinct nodes $x,y\in V(P)$.
The \emph{adhesion} of the path decomposition is the maximum of $|V(\beta(x)\cap\beta(y))|$ over all pairs of distinct nodes $x,y\in V(P)$.
Let us remark that in both of these definitions, it suffices to consider adjacent nodes $x$ and $y$.

Let $s$ denote the first node and $t$ the last node of $P$.
The \emph{interior width} of the decomposition is the maximum of $|V(\beta(x))|-1$ over all
nodes $x$ of $P$ different from $s$ and $t$.
For $x\in V(P)\sem\{s\}$, let $l(x)$ be the node of $P$ preceding $x$, and let $L(x)=V(\beta(l(x))\cap \beta(x))$.
Similarly, for $x\in V(P)\sem\{t\}$, let $r(x)$ be the node of $P$ following $x$, and let $R(x)=V(\beta(r(x))\cap \beta(x))$.
The path decomposition is \emph{$p$-linked} if $|L(x)|=p$ for all $x\in V(P)\setminus\{s\}$
and $G$ contains $p$ pairwise vertex-disjoint paths from $R(s)$ to $L(t)$.
Note that this implies that for every $x\in V(P)\setminus\{s,t\}$, the subgraph $\beta(x)$ contains $p$ pairwise vertex-disjoint paths from $L(x)$ to $R(x)$.

The notion of linkedness of a path decomposition plays an important role in many arguments in
the graph minor theory.  The following standard observation is that any (proper) path decomposition can be restructured
into a linked one, while keeping control over its order and interior width (though we can lose the control
over width, since a large part of the decomposition may need to be included in the bags of $s$ and $t$).
A path decomposition $(P',\beta')$ of a graph $G$ is a \emph{coarsening} of a path decomposition $(P,\beta)$ if $P'=y_1\ldots y_m$ and $P$ can be expressed
as the concatenation of paths $P_1$, \ldots, $P_m$ such that $\beta'(y_i)=\bigcup_{x\in V(P_i)} \beta(x)$ for each $i\in\{1,\ldots, m\}$.
For a subpath $Q\subseteq P$, the \emph{restriction} of the path decomposition $(P,\beta)$ to $Q$ is the coarsening
$(Q,\beta')$ of $(P,\beta)$ such that $\beta'(x)=\beta(x)$ for all nodes $x$ of $Q$ distinct from its endpoints $s_Q$ and $t_Q$;
note that $\beta'(s_Q)$ is the union of the bags of the nodes of $P$ preceding or equal to $s_Q$,
and $\beta'(t_Q)$ is the union of the bags of the nodes of $P$ following or equal to $t_Q$.

\begin{lemma}\label{lemma-link}
Let $a$ and $w$ be non-negative integers and let $f_0:\mathbb{N}\to\mathbb{N}$ be an arbitrary non-decreasing function.
There exist integers $w_0$ and $n_0$ such that the following holds.
If a graph $G$ has a proper path decomposition of interior width at most $w$, adhesion at most
$a$, and order at least $n_0$, then for some $w'\le w_0$ and $p\le a$, the graph $G$ also has a $p$-linked proper path decomposition of
interior width at most $w'$ and order at least $f_0(w')$.
\end{lemma}
\begin{proof}
Let $(P,\beta)$ be a proper path decomposition of $G$ of interior width at most $w$ and adhesion at most $a$.
We prove the claim by induction on $a$.  If $a=0$, then $(P,\beta)$ is $0$-linked, and thus the claim holds with
$w_0=w$ and $n_0=f_0(w)$.
Hence, assume that $a\ge 1$.  Let $w_0'$ and $n'_0$ be the values of $w_0$ and $n_0$ for the inductive application of Lemma~\ref{lemma-link}
for $a-1$ and with the interior width of the decomposition bounded by $(2f(w)+1)$ rather than $w$.  Let $w_0=\max(w'_0,w)$ and $n_0=(2n'_0+4)f_0(w)+1$.

We say that a node $x$ of $P$ distinct from its endpoints is \emph{$a$-linked} if
$|L(x)|=|R(x)|=a$ and $\beta(x)$ contains $a$ pairwise vertex-disjoint paths from $L(x)$ to $R(x)$, and \emph{broken} otherwise.  By Menger's theorem,
if $x$ is broken, then there exist subgraphs $A_x,B_x\subseteq \beta(x)$
such that $A_x\cup B_x=\beta(x)$, $L(x)\subseteq V(A_x)$, $R(x)\subseteq V(B_x)$ and $|V(A_x\cap B_x)|\le a-1$;
that is, $V(A_x\cap B_x)$ is a cut of size at most $a-1$ separating $L(x)$ from $R(x)$ in $\beta(x)$.

If $P$ contains a subpath $Q$ of $f_0(w)$ consecutive $a$-linked nodes,
then the restriction of $(P,\beta)$ to $Q$ is an $a$-linked proper path decomposition of interior width at most $w$
and order at least $f_0(w)$.

Otherwise, at least one of each $f_0(w)$ consecutive internal nodes of $P$ is broken.  Let $x_0$, \ldots, $x_n$ be a maximal sequence
of broken nodes in order they appear on $P$, and observe that the distance between consecutive elements of this sequence
in $P$ is at most $f_0(w)$, and so are the lengths of the initial segment preceding $x_0$ and the final segment following $x_n$.
In particular, we have $n\ge \tfrac{n_0-1}{f_0(w)}-2\ge 2n'_0+2$.

Let $(P',\beta')$ be the path decomposition of $G$ obtained from $(P,\beta')$ by, for $i\in\{0,\ldots,n'_0\}$,
replacing the node $x_{2i+1}$ by nodes $a_i$ and $b_i$ with bags $\beta'(a_i)=A_{x_{2i+1}}$ and $\beta'(b_i)=B_{x_{2i+1}}$.
Let $(P'',\beta'')$ be the coarsening of $(P',\beta')$ with $P''=z_0z_1\ldots z_{n'_0+1}$, with
the node $z_0$ corresponding to the initial segment of $P'$ till $a_0$,
the node $z_i$ corresponding to the segment of $P'$ between $b_{i-1}$ and $a_i$ for $i\in \{1,\ldots, n'_0\}$, and
the node $z_{n'_0+1}$ corresponding to the final segment of $P'$ starting in $b_{n'_0}$.

Observe that the path decomposition $(P'',\beta'')$ is proper, since for distinct $i_1,i_2\in \{0,\ldots, n'_0+1\}$, we have
$\beta(x_{2i_1})\subseteq \beta''(z_{i_1})$ and $\beta(x_{2i_2})\subseteq \beta''(z_{i_2})$, and
the sets $V(\beta(x_{2i_1}))$ and $V(\beta(x_{2i_2}))$ are incomparable since the path decomposition $(P,\beta)$ is proper.
Since the distance between the consecutive elements of $x_0$, \ldots, $x_n$ in $P$ is at most $f_0(w)$ and the path decomposition $(P,\beta)$
has interior width at most $w$, we have $|V(\beta''(z_i))|\le (2f(w)+1)w$ for $i\in \{1,\ldots, n'_0\}$, and thus
the path decomposition $(P'',\beta'')$ has interior width at most $(2f(w)+1)w$.   Moreover, it clearly has adhesion at most $a-1$
and order at least $n'_0$.

By the induction hypothesis applied to $(P'',\beta'')$, we conclude that for some $w'\le w'_0\le w_0$ and $p\le a-1<a$,
the graph $G$ has a $p$-linked proper path decomposition of interior width at most $w'$ and order at least $f_0(w')$.
Hence, the conclusion of the lemma holds.
\end{proof}

Consider a $p$-linked path decomposition $(P,\beta)$ of a graph $G$, and let $s$ and~$t$ be the ends of $P$.
We would like to consider the bag $\beta(x)$ of each node $x\in V(P)\setminus\{s,t\}$
as a framed graph with boundary $L(x)\cup R(x)$, so that we can combine consecutive ones using the composition operation.
A technical issue preventing us from doing so is that the intersection $L(x)\cap R(x)$ can be different for each $x\in V(P)\setminus\{s,t\}$.
However, this is easily fixed.  We say that a path decomposition $(P,\beta)$ is \emph{$p$-framed} if it is $p$-linked and there exists a set $C\subseteq V(G)$
such that for every node $x\in V(P)\setminus\{s,t\}$, we have $L(x)\cap R(x)=C$.  We say that $C$ is the \emph{core} of the decomposition.

\begin{lemma}\label{lemma-framed}
Let $p$ and $w$ be non-negative integers and let $f:\mathbb{N}\to\mathbb{N}$ be an arbitrary non-decreasing function.
There exist integers $w_1$ and $n_1$ such that the following holds.
If a graph $G$ has a $p$-linked proper path decomposition $(P,\beta)$ of interior width at most $w$
and order at least $n_1$, then for some $w'\le w_1$, the graph $G$ also has a $p$-framed proper path decomposition of
interior width at most $w'$ and order at least $f(w')$.
\end{lemma}
\begin{proof}
Let $s$ and $t$ be the ends of $P$, and let $C=R(s)\cap L(t)$ and $c=|C|$.  We prove the claim by reverse induction on $c$.
For $c=p$, let $w'=w_1=w$ and $n_1=f(w)$.  Note that if $|C|=p$, then $L(x)\cap R(x)=C$ for every $x\in V(P)\setminus\{s,t\}$,
and thus the path decomposition $(P,\beta)$ is $p$-framed with the core $C$ and the conclusion of the lemma holds.
Hence, we can assume that $c<p$.  Let $w'_1$ and $n'_1$ be the values of $w_1$ and $n_1$ from the induction hypothesis
for $c+1$, and let $w_1=\max(w'_1, (w+2p)n'_1)$ and $n_1=n'_1f(w_1)$.

Let $Q_1$, \ldots, $Q_p$ be pairwise vertex-disjoint paths from $R(s)$ to $L(t)$ in~$G$, which exist by the $p$-linkedness of
the path decomposition; we choose the labeling so that $Q_1$, \ldots, $Q_c$ are single-vertex paths consisting of the vertices of $C$.
Suppose first that there exist a subpath $S_1\subseteq P$ with at least $n'_1$ nodes and an integer $i\in\{c+1,\ldots, p\}$
such that, letting $s'$ and $t'$ be the ends of $R_1$, we have $R(s')\cap V(Q_i)=L(t')\cap V(Q_i)$.  In this case,
let $(S_1,\beta_1)$ be the restriction of $(P,\beta)$ to $S_1$, and note that this path decomposition satisfies $|R(s')\cap L(t')|\ge c+1$.
By the induction hypothesis applied to $(S_1,\beta_1)$, we conclude that
for some $w'\le w'_1\le w_1$, the graph $G$ has a $p$-framed proper path decomposition of interior width at most $w'$ and order at least $f(w')$.

Suppose now that this is not the case for any subpath of $P$ with at least $n'_1$ nodes.
Let $w'=(w+2p)n'_1\le w_1$ and let $P_1$, \ldots, $P_{f(w'})$ be the partition of $P$ into pairwise vertex-disjoint
paths, where all except for $P_1$ have exactly $n'_1$ nodes and $P_1$ has at least $n'_1$ nodes;
this partition exists since $n_1\ge n'_1f(w')$.  Let $(S_2,\beta_2)$ be the corresponding
coarsening of $(P,\beta)$, with $S_1=y_1\ldots f_{f(w')}$, and note that this path decomposition
has interior width at most $w'$.  For $j\in \{2,\ldots,f(w')-1\}$ and $i\in \{1,\ldots,p\}$,
we have $L(y_j)\cap V(Q_i)=R(y_j)\cap V(Q_i)=C\cap V(Q_i)$ if $i\le c$ and $L(y_j)\cap V(Q_i)\neq R(y_j)\cap V(Q_i)$ if $i>c$.
Therefore, the path decomposition $(S_2,\beta_2)$ is $p$-framed with the core $C$.  Hence, the conclusion of the lemma holds.
\end{proof}

Let us now combine the two lemmas.
\begin{corollary}\label{cor-framed}
Let $w$ be a non-negative integer and let $f:\mathbb{N}\to\mathbb{N}$ be an arbitrary non-decreasing function.
There exist integers $w_2$ and $n_2$ such that the following holds.
If a graph $G$ has a proper path decomposition $(P,\beta)$ of interior width at most $w$
and order at least $n_2$, then for some $w'\le w_2$ and $p\le w$, the graph $G$ also has a $p$-framed proper path decomposition of
interior width at most $w'$ and order at least $f(w')$.
\end{corollary}
\begin{proof}
For each positive integer $w''$, let us define $s(w'')$ and $f_1(w'')$ as the maximum of the values of $w_1$ and $n_1$,
respectively, from Lemma~\ref{lemma-framed} applied for all non-negative integers $p\le w$, the integer $w=w''$ and the function $f$.
Let $w_0$ and $n_0$ be the integers from Lemma~\ref{lemma-link} applied for $a=w$, $w$, and $f_1$.
Let $n_2=n_0$ and let $w_2$ be the maximum of $s(w'')$ over all non-negative integers $w''\le w_0$.

Note that the adhesion of $(P,\beta)$ is bounded by its interior width $w$.
By Lemma~\ref{lemma-link} applied for $a=w$, $w$, $f_1$, and the path decomposition $(P,\beta)$,
for some non-negative integers $w''\le w_0$ and $p\le w$, the graph $G$ has a $p$-linked proper path decomposition $(P_1,\beta_1)$ of
interior width at most $w''$ and order at least $f_1(w'')$.  By Lemma~\ref{lemma-framed} applied for this $p$,
$w=w''$, $f$, and the path decomposition $(P_1,\beta_1)$,
for some $w'\le s(w'')\le w_2$, the graph $G$ also has a $p$-framed proper path decomposition of interior width at most $w'$
and order at least $f(w')$.
\end{proof}

Let $(P,\beta)$ be a $p$-framed path decomposition of a graph $G$, and let $s$ and~$t$ be the ends of $P$.
Let $c=|R(s)\cap L(t)|$.  A $p$-tuple $(Q_1,\ldots, Q_p)$ of pairwise vertex-disjoint paths from $R(s)$ to $L(t)$ in $G$ is a \emph{frame}
if the paths $Q_{p-c+1}$, \ldots, $Q_p$ are the single-vertex paths formed by the vertices of $R(s)\cap L(t)$.
Suppose a frame is fixed.  Then we can turn the bags of the path decomposition $(P,\beta)$ into framed graphs in the natural way;
let us give a more general definition, associating a framed graph with each subpath $P'$ of $P-\{s,t\}$:
Let $s'$ and $t'$ be the ends of $P'$.  Let $\beta[P']$ be the framed graph with the underlying graph $\bigcup_{x\in V(P')} \beta(x)$,
$L_0(\beta[P'])$ consisting of the vertices of $Q_1$, \ldots, $Q_{p-c}$ in $L(s')$ in order,
$R_0(\beta[P'])$ consisting of the vertices of $Q_1$, \ldots, $Q_{p-c}$ in $R(t')$ in order,
and $C(\beta[P'])$ consisting of the vertices forming the paths $Q_{p-c+1}$, \ldots, $Q_p$ in order.
In case that $G$ is a plane graph, then $\beta[P']$ inherits the plane drawing from $G$.
Note that if $P_1$ and $P_2$ are consecutive vertex-disjoint subpaths of $P-\{s,t\}$ and $P_3=P[V(P_1)\cup V(P_2)]$, then the (plane) framed graphs $\beta[P_1]$
and $\beta[P_2]$ are compatible and $\beta[P_3]=\beta[P_1]+\beta[P_2]$.

\subsection{Applying Simon's factorization forest}

Finally, we are going to need an algebraic tool from the semigroup theory.
Let $T$ be a rooted ordered tree, where by \emph{ordered}, we mean that the order of children of each vertex is fixed.
Let $f$ be a function assigning to each vertex of $T$ a string, such that
\begin{itemize}
\item for each leaf $v$ of $T$, the string $f(v)$ has length exactly one, and
\item for each non-leaf vertex $v$ of $T$, the string $f(v)$ is the concatenation of the strings assigned by $f$ to the children of $v$ in order.
\end{itemize}
We say that the pair $(T,f)$ \emph{yields the string} assigned to the root of $T$ by $f$.  
If the letters of the strings are elements of a semigroup $A$, 
then for each $v\in V(T)$, let $f_A(v)$ denote the product of the letters of $f(v)$
in~$A$.
An element $e$ of $A$ is \emph{idempotent} if $e^2=e$.  
The pair $(T,f)$ is an \emph{$A$-factorization tree}
if for every vertex $v$ of $T$ with more than two children, there exists an idempotent element $e\in A$ such that $f_A(x)=e$ for
each child $x$ of $v$ (and hence also $f_A(v)=e$). 
Simon~\cite{simon1990factorization} showed existence of bounded-depth $A$-factorization trees for every string;
the improved bound in the following theorem was proved by
Colcombet~\cite{colcombet2010factorization}.  Recall that the \emph{depth} of a rooted tree is the maximum length (number of edges)
of a path from its root to a leaf.

\begin{theorem}[Simon~\cite{simon1990factorization};
	Colcombet~\cite{colcombet2010factorization}]\label{thm-facttree}
For every finite semigroup $A$ and each string of elements of $A$, there
exists an $A$-factorization tree of depth at most $3|A|$ yielding this string.
\end{theorem}

We combine Theorem~\ref{thm-facttree} with the following easy observation, asserting that if $T$ is large,
then there necessarily exists a large-degree vertex $v\in V(T)$ such that all the subtrees below $v$ are relatively small.

\begin{lemma}\label{lemma-bigvert}
Let $f:\mathbb{N}\to\mathbb{N}$ be an arbitrary non-decreasing function and let $d$ be a positive integer.
There exist integers $k_0$ and $n_0$ such that the following statement holds.  
If $T$ is a rooted tree of depth at most $d$ with at least $n_0$
leaves, then for some $k\le k_0$, there exists a vertex $v$ of $T$ that has at least $f(k)$ children,
while the subtree of $T$ rooted at each child of $v$ has at most $k$ leaves.
\end{lemma}

\begin{proof}
We prove the claim by induction on $d$.  For $d=1$, it suffices to set $k_0=1$ and $n_0=f(1)$.  Suppose that $d\ge 2$ and Lemma~\ref{lemma-bigvert} holds
for $d-1$, and let $k'_0$ and $n'_0$ denote the values of $k_0$ and $n_0$ for this inductive application.  Let $k_0=\max(k'_0,n'_0)$ and $n_0=n'_0f(n'_0)$.

Let $T$ be rooted tree of depth at most $d$ with at least $n_0$ leaves.
If the subtree of $T$ rooted at a child of the root has at least $n'_0$ leaves, then the claim follows by the
induction hypothesis applied to this subtree.  Otherwise, the root has at least $n_0/n'_0\ge f(n'_0)$ children, and the subtree rooted
in each of them has at most $n'_0$ leaves.  Hence, we can let $v$ be the root and $k=n'_0$.
\end{proof}

For our application, we only need the following corollary obtained by combining Theorem~\ref{thm-facttree} and Lemma~\ref{lemma-bigvert}.

\begin{corollary}\label{cor-manyidem}
Let $\alpha$ be a positive integer and let $f:\mathbb{N}\to\mathbb{N}$ be an arbitrary non-decreasing function.
There exist integers $k_0$ and $n_0$ such that if $(A,\circ)$ is a finite semigroup of size at most $\alpha$
and $s$ is a string of elements of $A$ of length at least $n_0$,
then $s$ is the concatenation of strings $s_0,s_1,\ldots, s_m,s_{m+1}$ for some integer $m$, such that
\begin{itemize}
\item there exists a positive integer $k\le k_0$ such that $m\ge f(k)$ and the strings $s_1$, \ldots, $s_m$ have length at least one but at most $k$, and
\item the product of the elements of $A$ in each of the strings $s_1$, \ldots, $s_m$ is the same idempotent element of $A$.
\end{itemize}
\end{corollary}

\begin{proof}
Without loss of generality, we can assume that $f(k)>2$ for every positive integer $k$.
Let $k_0$ and $n_0$ be the values from Lemma~\ref{lemma-bigvert} for the function $f$ and for $d=3\alpha$.
By Theorem~\ref{thm-facttree}, there exists an $A$-factorization tree $(T,f)$ yielding $s$ of depth at most $d$.
By Lemma~\ref{lemma-bigvert}, there exists a positive integer $k\le k_0$ and a vertex $v\in V(T)$ with at least $f(k)$ children,
such that the subtree of $T$ rooted at each child of $v$ has at most $k$ leaves.
Let $v_1$, \ldots, $v_m$ be the children of $v$ in order, and for $i\in \{1,\ldots,m\}$, let $s_i=f(v_i)$.
Let $s_0$ and $s_{m+1}$ be the prefix of $s$ preceding $s_1$ and succeeding $s_m$, respectively.

By the choice of $v$, each of the strings $s_1$, \ldots, $s_m$ has length at most $k$,
but the number $m$ of children of $v$ is at least $f(k)$.  Moreover, since $(T,f)$ is an $A$-factorization tree and $m\ge f(k)>2$, we conclude
that for each $i\in \{1,\ldots,m\}$, the product $f_A(v_i)$ of the elements of $A$ in the string $s_i$ is equal to the idempotent element $f_A(v)$.
\end{proof}

Let $(P,\beta)$ be a $p$-framed path decomposition of a plane graph $G$,
and let us fix a frame $\varphi=(Q_1,\ldots,Q_p)$.  Let $s$ and $t$ be the ends of $P$.  
Let $\varphi_0=(Q'_1,\ldots,Q'_p)$ be a $p$-tuple of pairwise vertex-disjoint paths,
where for $i\in \{1,\ldots,p\}$, we have $Q_i\subseteq Q'_i$,
the initial segment of $Q'_i$ till the start of $Q_i$ is contained in $\beta(s)$,
and the final segment of $Q'_i$ from the end of $Q_i$ is contained in $\beta(t)$; we say that $\varphi_0$ is a \emph{superframe} of $\varphi$.
Let $Y$ be a set of vertices of $G$.
We say that $(P,\beta)$ is \emph{uniform with respect to the relevant properties for $Y$ and $\varphi_0$} if
there exists 
\begin{itemize}
\item a finite semigroup $(A,\circ)$ and a function $f$ with range $A$ determining the properties
\begin{itemize}
\item \ref{prop:bridge} for each pair of boundary vertices,
\item \ref{prop:contY} for a given subset $Y$ of the vertices of $G$, and
\item \ref{prop:sepcyc} for arbitrarily chosen $G$-normal curves $\gamma_i$ for $i\in\{1,\ldots, p\}$,
where $\gamma_i$ starts in a face of $G$ incident with one end of $Q'_i$ and ends in a face of $G$ incident with the other end of $Q'_i$,
\end{itemize}
where $f$ is composable via $\circ$, and
\item an idempotent element $a_0\in A$ such that $f(\beta[x])=a_0$ for every internal node $x$ of $P$.
\end{itemize}
Note that this also implies that $f(\beta[P'])=a_0$ for every subpath $P'$ of $P-\{s,t\}$.
By combining the results obtained so far, we obtain the following theorem.

\begin{theorem}\label{thm-eqdec}
Let $w$ be a non-negative integer, and let $h:\mathbb{N}\to\mathbb{N}$ be an arbitrary non-decreasing function.
There exist integers $w_3$ and $n_3$ such that the following claim holds.
Let $G$ be a plane graph and let $Y$ be a set of vertices~of~$G$.  If $G$ has a proper path decomposition of interior width at most $w$ 
and order at least $n_3$, then
for some $w'\le w_3$ and $p\le w$, it also has a $p$-framed proper path decomposition $(P'',\beta'')$
of interior width at most $w'$ and order at least $h(w')$ which is uniform with respect to the relevant properties for $Y$
and a superframe of $(P'',\beta'')$.
\end{theorem}
\begin{proof}
By Corollary~\ref{cor-semi} and Observation~\ref{obs-proprod},
for any signature $\sigma$ and a finite set $\Gamma$ of simple curves in the plane,
there exists a finite semigroup $(A_{\sigma,\Gamma},\circ_{\sigma,\Gamma})$
and a function $f_{\sigma,\Gamma}$ with range $A_{\sigma,\Gamma}$ determining the properties \ref{prop:bridge}, \ref{prop:contY}
for the set $Y$, and \ref{prop:sepcyc} for the curves in $\Gamma$, and composable via $\circ_{\sigma,\Gamma}$.
Moreover, observe that there exists a function $\alpha:\mathbb{N}^2\to\mathbb{N}$ such that the size of $A_{\sigma,\Gamma}$ is
at most $\alpha(|\sigma|,|\Gamma|)$ for every $\sigma$ and $\Gamma$.

For each positive integer $w''$, let us define $k(w'')$ and $f_1(w'')$ as the maximum of the values $k_0$ and $n_0$,
respectively, from Corollary~\ref{cor-manyidem} applied for all non-negative integers $i\le w$ with $\alpha(i,i)$ playing the role of $\alpha$
and with the function $f(k)=h(w''k)$.
Let $w_2$ and $n_2$ be the values from Corollary~\ref{cor-framed} applied for $w$, and the function $f(w'')=f_1(w'')+2$.
Let $n_3=n_2$ and let $w_3$ be the maximum of $w''k(w'')$ over all non-negative integers $w''\le w_2$.

Let us apply Corollary~\ref{cor-framed} to $(P,\beta)$, with the function $f(w'')=f_1(w'')+2$;
we obtain a $p$-framed proper path decomposition $(P',\beta')$ of $G$ of interior width at most $w''$ and order
at least $f_1(w'')+2$ for some $w''\le w_2$ and $p\le w$.  Let $\sigma$ be the signature of $(P',\beta')$ and note that
$|\sigma|=p\le w$.  Let $\varphi_0$ be a frame of $(P',\beta')$ and let $\Gamma$ consist of $G$-normal
curves for this (super)frame as in the definition of the uniformity with respect to the relevant properties.

Let $P'=s'x_1x_2\ldots x_nt'$, where $n\ge f_1(w'')$, and let $s$ be the string $$f_{\sigma,\Gamma}(\beta[x_1])\ldots f_{\sigma,\Gamma}(\beta[x_n]).$$
By Corollary~\ref{cor-manyidem} for $(A_\sigma,\circ)$ and the function $f(k)=h(w''k)$, there exist integers
$k\le k_0\le k(w'')$ and $m\ge h(w''k)$ such that $s$ is the concatenation of strings $s_0$, $s_1$, \ldots, $s_m$, $s_{m+1}$,
where $s_1$, \ldots, $s_m$ have length between $1$ and $k$ and the product of each of them is the same idempotent element $a_0\in A_{\sigma,\Gamma}$.
Let $P'_0$, $P_1$, $P_2$, \ldots, $P_m$, $P'_{m+1}$ be the subpaths of $P'-\{s',t'\}$ corresponding to these strings,
let $P_0$ be the concatenation of $s'$ and $P'_0$, and let $P_{m+1}$ be the concatenation of $P'_{m+1}$ and $t$.

Let $(P'',\beta'')$ be the coarsening of $(P',\beta')$ according to the subpaths $P_0$, \ldots, $P_{m+1}$, and let $s''$ and $t''$ be the ends of $P''$.
We have $f_{\sigma,\Gamma}(\beta[x])=a_0$ for every $x\in V(P'')\setminus\{s'',t''\}$, and thus the $p$-framed path decomposition $(P'',\beta'')$
is uniform with respect to the relevant properties for $Y$ and the superframe $\varphi_0$ of $(P'',\beta'')$.
Let $w'=w''k\le w''k(w'')\le w_3$.  Since the interior width of $(P',\beta')$ is at most $w''$, the interior width of $(P'',\beta'')$
is at most $w''k=w'$.  Moreover, $(P'',\beta'')$ has order at least $m\ge h(w''k)=h(w')$.
\end{proof}

In other words, we can find a long path decomposition in which all relevant properties of the drawing
that hold in one bag repeat in all the bags.  So, for example, if \ref{prop:sepcyc} holds for one bag (for some curve $\gamma$),
then it holds in every bag, and we conclude that the drawing contains many cycles separating the ends of $\gamma$.
If we can additionally argue that many of these cycles are (nearly) disjoint, we obtain a large nest,
which is quite useful in view of Theorem~\ref{thm-nestdepth}.  The precise statement arising from this idea
is stated and proven in the following section.

\section{Bands and fans of plane tiles}
\label{sec:struc-tiles}

The proof of our structural characterization of crossing-critical graphs can
be roughly divided into two main parts.
In the first one, presented in this section, we establish the existence
of a plane band or fan in each crossing-critical graph,
where the band or fan is subdivided into short segments (called \emph{tiles})
and the band or fan is arbitrarily long compared to the size of the tiles.
This key result is summarized below in Corollary~\ref{cor-tiles}.

In particular, we can require the band or fan to be so long that the pigeonhole principle
implies that some of the tiles must repeat many times, even including their
bounded-length neighborhood.
In the the second part, presented in Section~\ref{sec:shorten-band}, we analyze
this situation more closely, defining reduction and expansion operations and showing
that when applied to such a ``prolific'' tile, they both preserve $c\,$-crossing-criticality.

In the traditional ``bottom-up'' approach (e.g., \cite{bokal2016characterizing}),
the allowed types of the tiles are fixed in advance and one then needs
to prove that (at least the central part of) every sufficiently long band in the considered $c\,$-crossing-critical
graph can be completely partitioned into such tiles.  As we cannot define allowed tile types explicitly,
we define tiles in a less precise, ``top-down'' fashion, where we divide the band (or fan)
into short segments somewhat arbitrarily. The interesting tile types are then only identified
based on their prolificity.  In particular, tiles in our sense might actually correspond to combinations of
several allowed tiles (or their parts) in the traditional sense.

\begin{figure}
\begin{center}
\begin{tikzpicture}[scale=0.75]
\tikzstyle{every node}=[draw, shape=circle, minimum size=2pt,inner sep=1pt, fill=black]
\tikzstyle{every path}=[fill=white, thin]
\draw[fill=black!20] (0,0) -- (0,2.98) -- (10.45,2.98) -- (10.45,0.02) -- (0,0.02) ;
\draw[fill=black!10] (2.1,0.02) -- (4.2,0.02) -- (4.2,2.98) -- (2.1,2.98) -- (2.1,0) ;
\draw[fill=black!10] (6.3,0.02) -- (8.4,0.02) -- (8.4,2.98) -- (6.3,2.98) -- (6.3,0) ;
\draw (0,-0.02) arc (100:80:30) ;
\draw (0,3.02) arc (-100:-80:30) ;
\node[label=below:$v_1$,label=above:$P_1\quad~$] (v1) at (0,0) {};
\node[label=below:$v_2$,label=above:$P_2\quad~~$] (v2) at (2.1,0.25) {};
\node[label=below:$v_3$,label=above:$P_3\quad~$] (v3) at (4.2,0.4) {};
\node[label=below:$v_4$,label=above:$P_4\quad~~$] (v4) at (6.3,0.4) {};
\node[label=below:$v_5$,label=above:$P_5\quad~$] (v5) at (8.4,0.25) {};
\node[label=below:$v_6$,label=above:$P_6\quad~$] (v6) at (10.45,0) {};
\node[label=above:$u_6$] (u6) at (0,3) {};
\node[label=above:$u_5$] (u5) at (2.1,2.75) {};
\node[label=above:$u_4$] (u4) at (4.2,2.6) {};
\node[label=above:$u_3$] (u3) at (6.3,2.6) {};
\node[label=above:$u_2$] (u2) at (8.4,2.75) {};
\node[label=above:$u_1$] (u1) at (10.45,3) {};
\node[draw=none, fill=none] at (5.3,-0.9) {$F_1$} ;
\node[draw=none, fill=none] at (5.3,3.9) {$F_2$} ;
\tikzstyle{every path}=[fill=none, dashed, color=black!50]
\draw (v1) to [out=-20,in=-150] (v2) ;
\draw (v2) to [out=-20,in=-150] (v3) ;
\draw (v3) to [out=-30,in=-150] (v4) ;
\draw (v4) to [out=-30,in=-160] (v5) ;
\draw (v5) to [out=-30,in=-160] (v6) ;
\draw (u6) to [out=20,in=150] (u5) ;
\draw (u5) to [out=20,in=150] (u4) ;
\draw (u4) to [out=30,in=150] (u3) ;
\draw (u3) to [out=30,in=160] (u2) ;
\draw (u2) to [out=30,in=160] (u1) ;
\tikzstyle{every path}=[color=black, very thick]
\draw (v1) -- (u6) ;
\draw (v2) -- (u5) ;
\draw (v3) -- (u4) ;
\draw (v4) -- (u3) ;
\draw (v5) -- (u2) ;
\draw (v6) -- (u1) ;
\tikzstyle{every node}=[draw, shape=circle, minimum size=2pt,inner sep=1pt, fill=white]
\tikzstyle{every path}=[very thin]
\draw (-1,-0.2) -- (v1) --  (-0.7,0.4) ;
\draw (-1,3.2) -- (u6) ;
\node (x1) at (0,1.7) {};
\node (y1) at (1,0.13) {};
\node (z1) at (1.1,2.86) {};
\node (x2) at (2.1,1.9) {};
\draw (-0.8,1.5) -- (x1) -- (-0.7,2) ;
\draw (y1) -- (x1) -- (z1) -- (x2) -- (x1) ;
\node (y2) at (2.1,0.9) {};
\node (z2) at (3.1,0.37) {};
\node (t2) at (3.1,1.5) {};
\draw (z2) -- (t2) -- (x2) -- (u4) ;
\draw (y2) -- (t2) -- (u4) ;
\node (x3) at (5,1.8) {};
\node (y3) at (5.3,0.42) {};
\node (z3) at (5.3,2.58) {};
\node (x4) at (6.3,1.8) {};
\draw (y3) -- (x3) -- (z3) ;
\draw (v3) -- (x3) -- (x4) ;
\node (y4) at (6.3,0.9) {};
\node (z4) at (7.3,0.37) {};
\node (t4) at (7.3,1.5) {};
\draw (z4) -- (t4) -- (x4) -- (u2) ;
\draw (y4) -- (t4) -- (u2) ;
\node (x6) at (10.45,1.9) {};
\draw (v5) -- (x6) -- (u2) ;
\draw (11.4,-0.2) -- (v6) --  (11.2,0.4) ;
\draw (11.4,1.2) -- (x6) --  (11.2,1.8) ;
\draw (11.4,3.2) -- (u1) ;
\end{tikzpicture}
\vspace*{-3ex}
\end{center}
\caption{An example of paths $P_1,\dots,P_6$ (bold lines) forming a band of
	length $6$ with bottom face $F_1$ and top face $F_2$, see Definition~\ref{def:band-fan}.
	The five tiles of this band, as in Definition~\ref{def:tiles-support},
	are shaded in gray and the dashed arcs represent $\alpha_i$ and
	$\alpha_i'$ from that definition.}\label{fig-band-and-tiles}
\end{figure}

Let us now give the definitions needed to state our results more precisely.
See Figure~\ref{fig-band-and-tiles} for an illustration of the definitions of
a band and its tiles.  The fan case is obtained by contracting the path
between $u_1$ and $u_6$ in the boundary of $F_2$ (with $u_2$, \ldots, $u_5$
in its interior) to a single vertex $u$.  We give the definitions for plane graphs.
In general, we are going to apply them to the planarizations of optimal drawings of
crossing-critical graphs; however, the main result of this section, Theorem~\ref{thm-tiles}
applies to plane graphs in general and may be of independent interest.

\begin{definition}[band and fan] 
\label{def:band-fan}
Let $G$ be a plane graph. 
Let $F_1$ and $F_2$ be distinct faces of $G$ and for an integer $m\ge 3$, let $v_1, v_2, \ldots, v_m$,
and $u_1, u_2, \ldots, u_m$ be some of the vertices incident with
$F_1$ and $F_2$, respectively, not necessarily consecutive but listed in the clockwise cyclic order along the faces.
If $P_1$, \ldots, $P_m$ are pairwise vertex-disjoint paths in $G$ such that for $i\in \{1,\ldots,m\}$,
the path $P_i$ has ends $v_i$ and $u_{m+1-i}$, then we say that
$(P_1,\ldots,P_m)$ forms a \emph{band of length $m$} with \emph{bottom face} $F_1$ and \emph{top face} $F_2$.  
Note that for some $i\in\{1,\ldots,m\}$ the path $P_i$ may consist of only a single vertex $v_i=u_{m+1-i}$.

Let $F_1$ and $v_1, v_2, \ldots, v_m$ be as above.
If $u$ is a vertex of $G$ different from $v_1$, \ldots, $v_m$ and $P_1$, \ldots, $P_m$ are
paths in $G$ such that for $i\in \{1,\ldots,m\}$, the path $P_i$ has ends $v_i$ and $u$,
and the paths are pairwise vertex-disjoint except for their common end~$u$, then
we say that $(P_1,\ldots,P_m)$ forms a \emph{fan of length~$m$} with \emph{bottom face} $F_1$ and \emph{top vertex} $u$. 
The fan is \emph{proper} if $u$ is not incident with $F_1$.
\end{definition}

Let us remark that non-proper fans are the same as degenerate $F$-nests, and thus the
long ones are excluded by Lemma~\ref{lemma-fnest}; hence, we can in general consider only proper fans.
By a \emph{sash}, we mean a proper fan or a band.

Let us remark that one could perhaps initially hope to modify the argument from the proof of Lemma~\ref{lemma-fnest}
to also exclude sashes.  However, the key difference here is that deleting the vertices of a path of a sash does not
(necessarily) disconnect the graph.  And indeed, there actually are critical graphs containing long bands and proper fans~\cite{kochol1987construction,dmcross}.

\begin{definition}[tiles and support]
\label{def:tiles-support}
Let $\PP=(P_1,\ldots,P_m)$ be a sash of length $m\ge 3$ with bottom face $F_1$ in a 2-connected plane graph $G$, and let $v_1$, \ldots, $v_m$ be the ends
of its paths incident with $F_1$.  For $i\in\{1,\ldots, m-1\}$, let $\alpha_i$ be an arc between $v_i$ and $v_{i+1}$ drawn inside $F_1$.
Moreover, in the case that $\PP$ is a band with top face $F_2$, let $u_m$, \ldots, $u_1$ denote the ends of $P_1$, \ldots, $P_m$
incident with $F_2$, and for $\in \{1,\ldots,m-1\}$, let $\alpha'_i$ be an arc drawn between $u_{m+1-i}$ and $u_{m-i}$
in $F_2$; $\alpha'_i$ is null when $\PP$ is a fan.  Furthermore, choose the arcs to be internally disjoint.
Let $\theta_i$ be the closed curve consisting of $P_i$, $\alpha_i$, $P_{i+1}$, and $\alpha'_i$.
Let $\lambda_i$ be the connected part of the plane minus $\theta_i$ that
contains none of the paths $P_j$ ($1\leq j\leq m$) in its interior.

The subgraphs $T_1$, \ldots, $T_{m-1}$ of $G$ drawn in the closures of $\lambda_1$, \ldots, $\lambda_{m-1}$
are called \emph{tiles} of the sash (the tile $T_i$
includes $P_i\cup P_{i+1}$ by this definition).
Let $F'_i$ be the unique face of $T_i$ not contained in $\lambda_i$;
we say that the closed walk bounding $F'_i$ is the \emph{border} of the tile, and denote it by $C(T)$.
The border consists of the paths $P_i$ and $P_{i+1}$ (the \emph{left border} and the \emph{right border} of the tile),
a path $Q_i$ in the boundary of $F_1$ (the \emph{bottom border} of the tile), and a path $Q'_i$ in the boundary of $F_2$ in the case that $\PP$ is a band
(the \emph{top border} of the tile, equal to the top vertex when $\PP$ is a fan).  
The union of the tiles is the \emph{support} of the sash.
A \emph{subsash} of the sash $\PP$ consists of a contiguous 
subinterval $(P_i, P_{i+1}, \ldots, P_j)$ of $\PP$.  We say that the subsash is \emph{delimited} by the paths $P_i$ and $P_j$.

A tile $T_i$ of $\PP$ is
\begin{itemize}
\item a \emph{pearl} if $|V(P_i)|=|V(P_{i+1})|=1$, and
\item \emph{shelled} if $|V(P_i)|,|V(P_{i+1})|\ge 2$ and the border of $T_i$ is a cycle
(i.e., either $\PP$ is a proper fan, or $\PP$ is a sash and the paths $Q_i$ and $Q'_i$ do not share a vertex incident with both $F_1$ and $F_2$).
\end{itemize}
We say that the sash is a \emph{necklace} if all its tiles are pearls,
and that it is \emph{shelled} if all its tiles are shelled.
\end{definition}

The cornerstone claim of this section is a structural result
on large plane graphs~$G$ of bounded path-width, showing that each such graph contains either a deep nest or a long sash.
Together with the structural properties of crossing-critical graphs that we have discussed in Section~\ref{sec:prelim},
this theorem implies the existence of long sashes in the planarizations of sufficiently large crossing-critical graphs.

One might think the existence of long sashes is nearly obvious; cannot we just take the bags of the path decomposition and turn them into
the tiles of the sash?  The issue with this simple idea is that the path decomposition does not give us much
control over the drawing of $G$; in particular, the vertices and edges of a single bag may be geometrically far apart in the plane drawing of~$G$.
As an example, consider the width two path decomposition of a cycle where one of the vertices of the cycle appears in all the bags.

To deal with this issue, we use Theorem~\ref{thm-eqdec} to obtain a framed path decomposition $(P,\beta)$ and focus on
a connected component $K_0$ of $\beta[P-\{s,t\}]-C$, where $s$ and $t$ are the ends of $P$ and $C$ is the core of the path decomposition.
This ensures that the parts of the bags in $K_0$ are drawn along the paths of the frame contained in $K_0$.
Moreover, using the uniformity with respect to the relevant properties, we conclude that each of the bags induces a connected
subgraph of $K_0$, ensuring that there exist many connections between the frame paths.
These connections then can be used to either divide $K_0$ (and its neighborhood in $C$) into tiles or to form a deep nest.

\begin{theorem}\label{thm-tiles}
Let $w$, $m$, and $k_0$ be non-negative integers, and let $g:\mathbb{N}\to\mathbb{N}$ be an arbitrary non-decreasing function.
There exist integers $w_0$ and $n_0$ such that the following claim holds.
Let $G$ be a $2$-connected plane graph and let $Y$ be a set of at most $k_0$ vertices of $G$ of degree at most $4$.
If $G$ has path-width at most $w$ and $|V(G)|\ge n_0$, then at least one of the following statements holds:
\begin{itemize}
\item[(a)] $G$ contains a $0$-nest, a $1$-nest, a $2$-nest, or a proper or degenerate $F$-nest for some face $F$ of $G$, of depth $m$, and with
all its cycles or paths disjoint from $Y$, or
\item[(b)] for some $w_1\le w_0$, $G$ contains a sash of length at least $g(w_1)$ and with support disjoint from $Y$,
such that each of its tiles has size at most $w_1$.  Moreover, the sash is either a necklace or shelled.
\end{itemize}
\end{theorem}
\begin{proof}
Let $h'(w')=\max(5,|Y|+1,4m,g(3w')+2)$, let $h(w')=2h'(w')+2$ and let $w_3$ and $n_3$ be the corresponding values from Theorem~\ref{thm-eqdec}.
Let $w_0=3w_3$ and $n_0=(w+1)n_3$.

Since $G$ has path-width at most $w$ and $|V(G)|\ge n_0$, $G$ has a proper path decomposition $(P_0,\beta_0)$
of (interior) width at most $w$ and order at least $n_0/(w+1)=n_3$.
By Theorem~\ref{thm-eqdec}, there exist integers $w'\le w_3$ and $p\le w$ such that
$G$ has a $p$-framed proper path decomposition $(P,\beta)$ of
interior width at most $w'$ and order at least $h(w')$, with a frame $\varphi$ and a superframe $\varphi_0$,
such that $(P,\beta)$ is uniform with respect to the relevant properties for $Y$ and $\varphi_0$.
Let $C$ be the core of the path decomposition $(P,\beta)$ and let $s$ and $t$ be the ends of $P$.
Let $P'=P-\{s,t\}$ and let $I$ be the independent set in $P'$ consisting of every other vertex of $P'$; the size of $I$ is at least $(h(w')-2)/2\ge h'(w')$.
Note that for distinct $x_1,x_2\in I$, we have $V(\beta[x_1] \cap \beta[x_2])=C$.  For a vertex $v\in L(\beta[P'])$ and every node $x\in V(P')$,
let $v_x$ be the vertex in the tuple $L(x)$ placed at the same position as $v$ in the tuple $L(\beta[P'])$;
and for $v\in R(\beta[P'])$, let us define $v_x\in R(x)$ analogously.  In either case, let $Q_v$ be the path of the frame $\varphi$
containing $v$.  For each path $Q$ of the frame $\varphi$, let $\gamma_Q$ be the curve from the definition of the uniformity
with respect to the relevant properties such that $\gamma_Q$ starts and ends in the faces incident with the ends of the path of $\varphi_0$
containing $Q$.

Since $|I|\ge h'(w')>|Y|$, there exists a node $x_0\in I$ such that $(V(\beta[x_0])\cap Y)\setminus C=\emptyset$.
By the uniformity with respect to the relevant property \ref{prop:contY}, we conclude that $V(\beta[P'])\cap Y\subseteq C$.

\begin{claim}\label{cl-pathsfromC}
If a vertex $u\in C$ is joined to another vertex $v\in L(\beta[P'])\cup R(\beta[P'])$ by a path $M$ in $\beta[P']$ with no internal vertices in $C$,
then $u\not\in Y$ and either the conclusion (a) holds, or $v\not\in C$.
\end{claim}
\begin{subproof}
By Observation~\ref{obs-bridge-to-comp} and the uniformity with respect to the relevant property \ref{prop:bridge},
for every node $x\in I$, the vertices $u$ and $v_x$ are connected by a path $M_x$ in $\beta[x]$ with no internal vertices in $C$.
The $|I|$ paths $\{M_x:x\in I\}$ are pairwise vertex-disjoint except for $u$ and possibly for $v$ if $v\in C$.
Since $|I|\ge 5$ and the vertices of $Y$ have degree four, we conclude that $u\not\in Y$.
If $v\in C$, then we also have $v\not\in Y$, and since $V(\beta[P'])\cap Y\subseteq C$,
the subgraph $\bigcup_{x\in I} M_x$ contains a $2$-nest of depth $\lfloor |I|/2\rfloor\ge m$ disjoint from $Y$,
and the conclusion (a) holds.
\end{subproof}
Let $K_0$ be a connected component of the graph $\beta[P']-C$. This graph is non-null, since
the path decomposition $(P,\beta)$ is proper.  Let $C_0$ be the set of vertices of $C$ adjacent in $\beta[P']$
to a vertex of $K_0$, and let $K$ be the subgraph of $\beta[P']$ induced by $V(K_0)\cup C_0$.

\begin{claim}\label{cl-Knice}
If the conclusion (a) does not hold, then $|C_0|\le 1$, at least one path of the frame~$\varphi$ is contained in $K_0$,
and $V(K)\cap Y=\emptyset$; and moreover, the graph $\beta(x)\cap K$ is connected for every node $x\in V(P')$.
\end{claim}
\begin{subproof}
If $C_0$ contained at least two distinct vertices, then they would be
joined by a path in $K$ with all internal vertices in $K_0$, and by Claim~\ref{cl-pathsfromC}, the conclusion (a)
would hold.  Therefore, we have $|C_0|\le 1$.  Since $G$ is $2$-edge-connected, we have $|(L(\beta[P'])\cup R(\beta[P']))\cap V(K)|\ge 2$,
and thus there exists a vertex $v\in (L(\beta[P'])\cup R(\beta[P']))\cap V(K_0)$.  Since $K_0$ is a component of
$\beta[P']-C$, the path $Q_v\in \varphi$ is contained in $K_0$. If $|C_0|=1$,
then the unique vertex $u\in C$ is joined to $v$ by a path in $K$ with all internal vertices in $K_0$,
and thus by Claim~\ref{cl-pathsfromC}, we have $u\not\in Y$.  Since $V(\beta[P'])\cap Y\subseteq C$, we conclude that $V(K)\cap Y=\emptyset$.

Consider now any node $x\in V(P')$ and any vertices $y_1,y_2\in V(\beta(x)\cap K)$.  For $i\in \{1,2\}$,
since $K$ is connected, it contains a path from $y_i$ to $v$.  Let $A_i\subseteq \beta(x)\cap K$ be the shortest initial segment of this
path from $y_i$ to a vertex $z_i\in (L(x)\cup R(x))\cap V(K)$, which exists since $(P,\beta)$ is a path decomposition
and $K_0$ is a component of $\beta[P']-C$.  Let $w^i\in L(\beta[P'])\cup R(\beta[P'])$ be the vertex such that $w^i_x=z_i$;
since the path of the frame $\varphi$ containing $z_i$ is contained in $K$, we have $w^i\in K$.
The vertices $w^1$ and $w^2$ are joined in $K$ by a path with internal vertices in $K_0$.
By Observation~\ref{obs-bridge-to-comp} and the uniformity with respect to the relevant property \ref{prop:bridge},
it follows that $z_1$ and $z_2$ are joined by a path $A$ in $\beta(x)$ with no internal vertices in $C$.
Since $|C_0|\le 1$, either $z_1=z_2$ or at least one of $z_1$ and $z_2$ belongs to $K_0$; and since $K_0$
is a component of $\beta[P']-C$, we conclude that $A$ is actually a path in $\beta(x)\cap K$.  Thus, the vertices $y_1$ and $y_2$
are joined by a walk consisting of $A_1$, $A$, and $A_2$ in $\beta(x)\cap K$.  It follows that the graph $\beta(x)\cap K$ is connected.
\end{subproof}

Let $s'$ and $t'$ be the ends of $P'$.

\begin{claim}\label{cl-commonface}
If the conclusion (a) does not hold, then $K$ has a face incident both with a vertex
of $\beta(s')\cap K_0$ and with a vertex of $\beta(t')\cap K_0$.
\end{claim}
\begin{subproof}
By Claim~\ref{cl-Knice}, there exists a path $Q$ of the frame $\varphi$ contained in $K_0$;
let $u\in L(s')\cap V(K_0)$ and $v\in R(t')\cap V(K_0)$ be its ends.
We claim that $u$ and $v$ are incident with a common face of $K$.

Otherwise, $K$ contains a cycle $S$ separating $u$ from $v$; let $\Lambda_u\ni u$ and $\Lambda_v\ni v$
be the open regions to which $S$ separates the plane.  Let $Q_0$ be the superpath
of $Q$ contained in the superframe $\varphi_0$, and let $u_0$ and $v_0$ be its ends.
The segment of $Q_0$ between $u_0$ and $u$ is disjoint from $K$ except for its end $u$,
and thus $u_0\in \Lambda_u$, and similarly $v_0\in \Lambda_v$.

It follows that the simple curve $\gamma_Q$ has one end in $\Lambda_u$ and the other end in $\Lambda_v$,
and thus $S$ separates the ends of $\gamma_Q$.
By the uniformity with respect to the relevant property \ref{prop:sepcyc},
for every $x\in I$, the plane graph $\beta[x]$ contains a cycle $S_x$ separating the ends of $\gamma_Q$.
Clearly $S_x$ intersects $Q_0$, and thus also $Q$. It follows that $S_x$ contains a vertex of $K_0$, and since
$K_0$ is a component of $\beta[P']-C$ and $K_0$ has at most one neighbor in $C$,
the cycle $S_x$ is contained in $K$.

Since $|I|\ge 4m$, there exists an end $q$ of $\gamma_Q$ and a set $I'\subseteq I$ of size $2m$
such that for each $x\in I'$, the point $q$ is contained in the open disk bounded by $S_x$.
The cycles $S_x$ for $x\in I'$ intersect at most in the single vertex of $C_0$ (if any),
and thus $m$ of them form either a $0$-nest or a $1$-nest of depth $m$ disjoint from $Y$.
Therefore, (a) holds.
\end{subproof}

By the previous claim, we can assume that $K$ has a face $F_0$ incident with both a vertex
of $\beta(s')\cap K_0$ and a vertex of $\beta(t')\cap K_0$.
Let $\varphi'$ be the set of the paths of the frame $\varphi$ that intersect (and thus are
contained in) $K$, and let $K'$ be the graph $(K\cap (\beta(s')\cup\beta(t')))\cup \bigcup \varphi'$.
By Claim~\ref{cl-Knice}, we can assume that the graphs $\beta(s')\cap K$ and $\beta(t')\cap K$ are connected,
and thus the graph $K'$ is connected.  Let $F'_0$ be the face of $K'$ containing $F_0$.

For each path $Q$ of $\varphi'$, let $Q'$ be the subpath of $Q$ between $R(s')$ and $L(t')$.
Since the graphs $\beta(s')\cap K_0$ and $\beta(t')\cap K_0$ are vertex-disjoint,
observe that there exist paths $Q_1$ and $Q_2$ of the frame $\varphi$ such that
the boundary of $F'_0$ consists of $Q'_1$, a walk in $\beta(s')\cap K$, $Q'_2$, and a walk in $\beta(t')\cap K$.
Let us remark that one of $Q'_1$ and $Q'_2$ can consist of the single vertex of $C_0$;
and that $Q'_1=Q'_2$ when $\varphi'$ consists of only one path.
Let $s''$ and $t''$ be the neighbors of $s'$ and $t'$ in $P'$, respectively.
A \emph{slice} is the subpath $\beta(x)\cap Q'_1$ or $\beta(x)\cap Q'_2$ for a node $x\in V(P')\setminus \{s',s'',t',t''\}$;
note that the ends of the slice are in $L(x)$ and $R(x)$.  The \emph{left arc} is the part of the boundary of
$F'_0$ between the vertices of $R(s'')$ in $Q'_1$ and $Q'_2$ which intersects $\beta(s')\cap K_0$,
and the \emph{right arc} is the part between the vertices of $L(t'')$ which intersects $\beta(t')\cap K_0$.
Thus, each edge incident with $F'_0$ is contained in exactly one slice or arc.  The union of the left arc and the right arc
is the \emph{brace}.

\begin{claim}\label{cl-incid}
If the conclusion (a) does not hold, then every path $B\subseteq G$ with end in the boundary of $F'_0$ and otherwise drawn inside $F'_0$
has both ends contained either in a slice or in the brace.  Moreover, for $i\in \{1,2\}$, $G$ has a face $F_i$ such that
for every $x\in V(P')\setminus \{s',s'',t',t''\}$, the vertices of $Q'_i\cap (L(x)\cup R(x))$ are incident with $F_i$.
\end{claim}
\begin{subproof}
We can assume that at least one end $u$ of $B$ is not contained in the brace, and in particular it does not belong to $R(s)\cup L(t)$.
Hence, we have $u\in V(K_0)\setminus (R(s)\cup L(t))$.  Since $K_0$ is a component of $\beta[P']-C$ and all its neighbors in $C$ are contained in the set $C_0\subset V(K)$ of size at most one,
it follows that $B$ is a path in $K$.  Let $X=\bigcup_{y\in V(P')} (L(y)\cup R(y))$.  The vertices of $X\cap V(K)$ are on the paths of $\varphi'$,
and thus they are not drawn in $F'_0$.  Hence, no internal vertex of $B$ belongs to $X$.  Since $(P,\beta)$ is a path decomposition,
we conclude that there exists a node $x\in V(P')\setminus \{s',t'\}$ such that $B\subseteq \beta(x)$.  If $x\in\{s'',t''\}$,
then both ends of $x$ are in the brace.  If $x\not\in\{s'',t''\}$, then $B$ cannot have one end in $Q'_1$ and the other end in $Q'_2$,
as otherwise $K$ could not have the face $F_0\subseteq F'_0$ incident with both a vertex
of $\beta(s')\cap K_0$ and a vertex of $\beta(t')\cap K_0$.  Hence, both ends of $B$ are contained in the slice $\beta(x)\cap Q'_1$ or $\beta(x)\cap Q'_2$.
Therefore, the first part of the claim holds.

This clearly implies that the ends of all slices of $Q_1$ are incident with the same face $F_1\subseteq F'_0$ of $G$,
and the ends of all slices of $Q_2$ are incident with the same face $F_2\subseteq F'_0$ of $G$.
Hence, the second part holds as well.
\end{subproof}

Let us remark that $F_1=F_2$ is possible, in the case that $G$ does not contain any path from the left arc to the right arc 
drawn in $F'_0$.  Consider any node $x\in V(P')\setminus \{s',s'',t',t''\}$.  Since the graph $\beta(x)\cap K$ is connected by Claim~\ref{cl-Knice},
it contains a path between the slices $\beta(x)\cap Q'_1$ and $\beta(x)\cap Q'_2$ disjoint from $Q'_1\cup Q'_2$ except for its ends;
by Claim~\ref{cl-incid}, the drawing of this path is disjoint from $F'_0$.  Since the ends of the slices are incident with $F_1$ and $F_2$,
respectively, we conclude that $\beta(x)\cap K$ contains a path $P_x$ with one end contained in $Q'_1$ and incident with $F_1$,
the other end contained in $Q'_2$ and incident with $F_2$, and otherwise disjoint from the boundaries of $F_1$ and $F_2$.

Let $I'=I\setminus \{s',s'',t',t''\}$ and note that the independent set $I'$ has size at least $|I|-2\ge \max(m,g(3w'))$.
For distinct nodes $x,x'\in I'$, we clearly have $V(P_x)\cap V(P_y)=C_0$.
If $C_0\neq\emptyset$, then by symmetry between $Q'_1$ and $Q'_2$, we can assume that the path $Q'_1$ does not consist of the single vertex
$u\in C_0$.  Let $P'_x$ denote the subpath of $P_x$ between it end in $Q'_1$ and $u$.
If $u$ is incident with $F_1=F_2$, then $(P'_x:x\in I')$ is a degenerate $F_1$-nest of depth at least $|I'|\ge m$, and the conclusion (a)
holds.  If $u$ is not incident with $F_1$, then $(P'_x:x\in I')$ is a shelled fan of length at least $|I'|\ge g(3w')$ with bottom face $F_1$ and top vertex $u$.
Since $I$ consists of every other vertex of $P$, each tile of the fan is contained in the union of three consecutive
bags of $(P,\beta)$. Since the interior width of the path decomposition $(P,\beta)$ is at most $w'$, each tile has size at most $3w'$.
Moreover, since $V(K)\cap Y=\emptyset$, the support of the fan is disjoint from $Y$.
Hence, the outcome (b) holds, with $w_1=3w'\le 3w_3=w_0$.

Similarly, if $C_0=\emptyset$, then $(P_x:x\in I')$ is either a proper $F_1$-nest of depth at least $m$ (if $F_1=F_2$)
and (a) holds, or a band of length at least $g(3w')$ with bottom face $F_1$ and top face $F_2$, tiles of size at most $3w'$, and support disjoint from $Y$ (if $F_1\neq F_2$).
In the latter case, if $\varphi'$ consists of a single path, then the band is a necklace.  Otherwise, it is shelled, since
the vertex-disjoint paths $Q'_1$ and $Q'_2$ separate the top and the bottom border of each tile.
\end{proof}

Since planarizations of optimal drawings of crossing-critical graphs do not contain deep nests,
only the long sash outcome is possible.  Thus, we obtain the following corollary, 
proving the core part of \ref{it:twokinds}.

\begin{corollary}\label{cor-tiles}
Let $c$ be a positive integer, and let $g:\mathbb{N}\to\mathbb{N}$ be an arbitrary non-decreasing function.
There exist integers $w_0$ and $n_0$ such that the following holds.
Let $G$ be a $2$-connected $c\,$-crossing-critical graph, and let $G'$ be the 
planarization of a drawing of $G$ with the smallest number of crossings.  
Let\/ $Y$ denote the set of crossing vertices of $G'$.
If $|V(G)|\ge n_0$, then $G'$ contains a sash $\PP$
such that for some $w_1\le w_0$, all the tiles of $\PP$ have size at most $w_1$ and are disjoint from $Y$,
and the length of $\PP$ is at least $g(w_1)$.  Moreover, the sash $\PP$ is either a necklace or shelled.
\end{corollary}
\begin{proof}
Let $k_0=\fnlab{thm-crit-cno}(c)$, $w=\fnlab{thm:bounded-pw}(c)+k_0$ 
and $m=\max(\fnlab{thm-nestdepth}(c),\fnlab{lemma-fnest}(c))$.  Let $w_0$ and $n_0$ be the corresponding values from
Theorem~\ref{thm-tiles}.

By Theorem~\ref{thm-crit-cno}, each $c\,$-crossing-critical graph 
has a drawing with at most $k_0$ crossings, and thus $|Y|\le k_0$.
By Theorem~\ref{thm:bounded-pw}, the graph $G$ has
path-width at most $w-k_0$, and thus its planarization $G'$ has path-width at most $w$
(we can simply add the vertices of $Y$ to all the bags of the path decomposition of $G$ to turn it into
a path decomposition of $G'$).  By Observation~\ref{obs-planarization-2connected}, the graph $G'$ is $2$-connected.

By Theorem~\ref{thm-nestdepth} and Lemma~\ref{lemma-fnest}, the graph $G'$ does not contain a
$0$-nest, a $1$-nest, a $2$-nest, or proper or degenerate $F$-nest (for any face $F$ of $G'$) of depth $m$
with cycles or paths disjoint from $Y$, and thus it does not satisfy (a) from Theorem~\ref{thm-tiles}.
Therefore, $G'$ satisfies (b) from Theorem~\ref{thm-tiles}, which matches the outcome of Corollary~\ref{cor-tiles}.
\end{proof}

\section{Removing and inserting tiles}
\label{sec:shorten-band}

In the second part of the paper, we study the arrangement of bounded tiles in a long enough plane sash.
We focus on finding repeated subsequences which then can be shortened.
Importantly, this shortening preserves $c\,$-cros\-sing-criticality.
In the opposite direction we then manage to define the converse operation of
``expansion'' of a plane band which also preserves $c\,$-cros\-sing-criticality.
These findings will imply the final outcome---a construction of 
all $c\,$-crossing-critical graphs from an implicit list of basic graphs of bounded size.

The precise statements can be found in Corollaries~\ref{cor-reducrit} and \ref{cor-expcrit} and in Theorem~\ref{thm:mainexpansion};
however, we need several auxiliary results and definitions first.

\subsection{Surgery across an edge cut}

For a graph $H$, the subgraphs $Q_1,Q_2\subset H$ are \emph{incomparable} if $V(Q_1)\setminus V(Q_2)\neq \emptyset\neq V(Q_2)\setminus V(Q_1)$.
For such subgraphs, the \emph{$(Q_1,Q_2)$-breadth} of $H$ is the maximum number $p$ of pairwise
edge-disjoint paths between $V(Q_1)\setminus V(Q_2)$ and $V(Q_2)\setminus V(Q_1)$ in $H-V(Q_1\cap Q_2)$.
For a sash $\PP=(P_1,\ldots,P_m)$ with tiles $T_1$, \ldots, $T_{m-1}$ in a plane graph, the \emph{breadth} of
the sash $\PP$ is defined as the $(P_1,P_m)$-breadth of the support of $\PP$, and for $i\in \{1,\ldots,m-1\}$, the \emph{breadth} of the tile $T_i$
is defined to be its $(P_i,P_{i+1})$-breadth.

Suppose that a graph $H$ is connected and drawn in the plane with crossings.
A directed path $P$ in $H$ is \emph{left-exposed} (resp. \emph{right-exposed}) if
\begin{itemize}
\item there are no crossings on the edges of $P$, and 
\item $H$ has a face $f$ such that $P$ is a part of the boundary walk of $f$ in the counterclockwise (resp. clockwise) order around $f$.
\end{itemize}
We say that such a face $f$ is a \emph{left window} (resp. \emph{right window}) of $P$.
Note that if $P$ consists of a single vertex, then it is both left- and right-exposed and all incident faces are both left and right windows,
and otherwise the left or right window is uniquely determined.
For left-exposed (or right-exposed) paths $P_1$ and $P_2$, a \emph{left (or right) window connector} is a simple $H$-normal curve disjoint from $P_1\cup P_2$
starting in a left (right) window of $P_1$ and ending in a left (right) window of $P_2$.

Suppose that $G$ is a graph and $G=G_1\cup H$ for subgraphs $G_1$ and $H$, and let $G'$ be another graph such that $G'=G_1\cup H'$ and $G_1\cap H=G_1\cap H'$.
Then we say that $G'$ is obtained from $G$ by \emph{replacing $H$ by $H'$}.  If $G_1\cap H$ is the union of incomparable paths $P_1$ and $P_2$ intersecting at most in a common
end of $P_1$ and $P_2$, then we say that $G_1$ and $H$ are \emph{$(P_1,P_2)$-fragments} of $G$.

We are now ready to state a key lemma, which we use to argue about the crossing number of graphs obtained by replacing parts of the sashes.

\begin{lemma}\label{lemma-surgcut}
Let $G$ be a graph and let $G_1$ and $H$ be $(P_1,P_2)$-fragments of $G$ for incomparable paths $P_1$ and $P_2$ intersecting at most in a common end,
let $G'$ be obtained from $G$ by replacing $H$ by another graph $H'$, and let us fix an orientation of the paths $P_1$ and $P_2$ and
a drawing $\Gamma'$ of $G'$ in the plane with $c'$ crossings.  Suppose that
\begin{enumerate}[(i)]
\item\label{it-sur-nocr} there are no crossings on the edges of $P_1$ and $P_2$ in the drawing $\Gamma'$,
\item\label{it-sur-rexp} $P_1$ and $P_2$ are right-exposed in the drawing of $G_1$ induced by $\Gamma'$,
\item\label{it-sur-wind} for $i\in\{1,2\}$, in the drawing of $H'$ induced by $\Gamma'$, each edge of $E(H')\setminus E(P_i)$ incident with a vertex of $V(P_i)\setminus V(P_{3-i})$
starts in a right window of $P_i$ in $G_1$, and
\item\label{it-sur-plan} $H$ has a plane drawing $\Gamma_H$ (without crossings) such that $P_1$ and $P_2$ are left-exposed
and a face $F$ of $H$ is a left window of both $P_1$ and $P_2$.
\end{enumerate}
Let $q$ be a non-negative integer.
If the $(P_1,P_2)$-breadth of $H'$ is at least $q$ and
either $q>c'$ or the $(P_1,P_2)$-breadth of $H$ is at most $q$, then there exists a drawing $\Gamma$ of $G$ with at most $c'$ crossings.
Moreover, the drawing of $G_1$ induced by $\Gamma$ is the same as the one induced by $\Gamma'$,
and for every $e\in E(G_1)$, if there is no crossing on $e$ in $\Gamma'$, then there also is no crossing on $e$ in $\Gamma$.
\end{lemma}
\begin{proof}
For any path $Q$ in $H'-V(P_1\cap P_2)$ from $V(P_1)\setminus V(P_2)$ to $V(P_2)\setminus V(P_1)$,
let $\gamma_Q$ be a simple curve tracing $Q$ in the drawing $\Gamma'$, but starting just after the first vertex of $Q$ and ending just before the last vertex of $Q$.
Observe that by \ref{it-sur-nocr}, \ref{it-sur-rexp}, and \ref{it-sur-wind}, $\gamma_Q$ is a right window connector for $P_1$ and $P_2$ in the drawing $\Gamma_1$.

If $q>0$, then since the $(P_1,P_2)$-breadth of $H'$ is at least $q$, there exists a path
in $H'-V(P_1\cap P_2)$ from $V(P_1)\setminus V(P_2)$ to $V(P_2)\setminus V(P_1)$.  Choose such a path
$Q_0$ with the smallest number $a$ of crossings with the drawing $\Gamma_1$, and let $\gamma=\gamma_{Q_0}$.
If $q=0$, then let $\gamma$ be an arbitrary right window connector for $P_1$ and $P_2$ in the drawing $\Gamma_1$
and let $a$ be the number of its intersections with $\Gamma_1$.
In both cases, let $F_1$ and $F_2$ be the right windows of $P_1$ and $P_2$, respectively,
containing the ends of $\gamma$.

Let $c_1$ be the number of crossings in the drawing $\Gamma_1$ of $G_1$ induced by $\Gamma'$.
If $F_1=F_2$, then $G$ has a drawing with $c_1\le c'$ crossings, obtained from $\Gamma_1$
by gluing the plane drawing $\Gamma_H$ of $H$ from \ref{it-sur-plan} between $P_1$ and $P_2$ in the face $F_1=F_2$.  Hence, suppose
that $F_1\neq F_2$, and thus $a\ge 1$.

Since the $(P_1,P_2)$-breadth of $H'$ is at least $q$, Menger's theorem implies that $H'-V(P_1\cap P_2)$ contains $q$
pairwise edge-disjoint paths from $V(P_1)\setminus V(P_2)$ to $V(P_2)\setminus V(P_1)$.
By the choice of $Q_0$, each such path crosses $G_1$ at least $a$ times in the drawing $\Gamma'$.
It follows that
$$c'\ge c_1+aq\ge q.$$
In particular $q\le c'$, and by the assumptions, the $(P_1,P_2)$-breadth of $H$ is at most $q$.
Therefore, the smallest edge-cut $C$ in $H-V(P_1\cap P_2)$ separating $V(P_1)\setminus V(P_2)$ from
$V(P_2)\setminus V(P_1)$ has size at most $q$.  Let $H-C=H_1\cup H_2$, where $P_1\subseteq H_1$, $P_2\subseteq H_2$, and $H_1\cap H_2=P_1\cap P_2$.
Observe that by the minimality of the edge-cut~$C$, all edges of $C$ are in the drawing of $H-C$
induced by $\Gamma_H$ drawn inside the same face of $H-C$, specifically the one containing $F$.

For $i\in \{1,2\}$, let us now draw $H_i$ in the face $F_i$ of the drawing $\Gamma_1$ of $G_1$, using the plane drawing of $H_i$
induced by $\Gamma_H$.  Then, let us draw the edges of $C$ along the curve $\gamma$, so that each of them crosses the drawing of $G_1$
exactly $a$ times.  The resulting drawing $\Gamma$ of the graph $G$ has
$$c_1+a|C|\le c_1+aq\le c'$$
crossings.  Moreover, the drawing of $G_1$ induced by $\Gamma$ is $\Gamma_1$.
Finally, if there is a crossing on an edge $e$ in $\Gamma$ but not in $\Gamma_1$, then $e$ is intersected by $\gamma$, and thus there is a crossing
of an edge of the path $Q_0$ with $e$ in the drawing $\Gamma'$.
\end{proof}

The conditions \ref{it-sur-nocr}, \ref{it-sur-rexp}, and \ref{it-sur-wind} of Lemma~\ref{lemma-surgcut} clearly hold when $P_1$ and $P_2$ 
are single-vertex paths, no matter what the drawing of $G'$ is.  Hence, we obtain the following corollary, which is useful in the necklace case.

\begin{corollary}\label{cor-surgcut}
Let $G$ be a graph and let $G_1$ and $H$ be $(v_1,v_2)$-fragments of $G$ for distinct vertices $v_1$ and $v_2$ of $G$.
Let $G'$ be a graph obtained from $G$ by replacing $H$ by another graph $H'$.
Suppose that $H$ has a plane drawing such that $v_1$ and $v_2$ are incident with the same face of $H$,
and that $G'$ has a drawing with $c'$ crossings.
Let $q$ be a positive integer.  If the $(v_1,v_2)$-breadth of $H'$ is at least $q$ and
either $q>c'$ or the $(v_1,v_2)$-breadth of $H$ is at most $q$, then $\crg(G)\le c'$.
\end{corollary}

\subsection{Reducing and expanding a necklace}

An essentially standard argument can be used to reduce and expand a necklace using the
operations described in the following definition.
Let $H$ be a graph with distinct vertices $u$ and $v$ and let $p$ be the $(u,v)$-breadth of $H$.  We say that $H$ is \emph{$(u,v)$-breadth-$p$-uniform}
if each edge of $H$ is contained in an edge-cut of size $p$ separating $u$ from $v$.
Equivalently, $H$ is the union of $p$ pairwise edge-disjoint paths from $u$ to $v$,
but no proper subgraph of $H$ has this property.  

\begin{definition}[reduction and expansion, the necklace case]\label{def:red-exp-necklace}
Let $G$ be a $2$-connected graph drawn in the plane with crossings, let $G'$ be the planarization of $G$
and let $\PP=(v_1,\ldots, v_m)$ be a necklace in $G'$
whose support does not contain any crossing vertices, with tiles $T_1$, \ldots, $T_{m-1}$ in order.
For any $i\in\{1,\ldots,m-1\}$, the graph obtained from $G-(V(T_i)\setminus \{v_i,v_{i+1}\})$ by identifying the vertices $v_i$ and $v_{i+1}$
is a \emph{$\PP$-reduction} of $G$.  Conversely, let $p$ be the breadth of $\PP$,
and suppose that for some $i\in \{2,\ldots,m-1\}$, a graph $M$ is obtained from $G$ by
\begin{itemize}
\item splitting the vertex $v_i$ into two vertices $v'_i$ and $v''_i$, with all edges of $T_{i-1}$ incident with $v_i$ redirected to $v'_i$
and all edges of $T_i$ incident with $v_i$ redirected to $v''_i$, and 
\item adding a plane $(v'_i,v''_i)$-breadth-$p$-uniform graph $T$ intersecting the rest of $G$ exactly in $v'_i$ and $v''_i$.
\end{itemize}
Then $M$ is a \emph{$\PP$-expansion of $G$}.
\end{definition}

Let us now argue that these operations preserve the crossing-criticality of $G$.

\begin{lemma}\label{lemma-neck}
Let $c$ be a positive integer.  Let $G$ be a $2$-connected $c\,$-crossing-critical graph, 
let $G'$ be the planarization of an optimal drawing $\Gamma$ of $G$,
let $\PP=(v_1,\ldots, v_m)$ be a necklace in $G'$
whose support $S$ does not contain any crossing vertices, with tiles $T_1$, \ldots, $T_{m-1}$ in order,
and let $p$ be the breadth~of~$\PP$.
Then $p\le c$ and for every $i\in\{1,\ldots,m-1\}$, the tile $T_i$ is $(v_i,v_{i+1})$-breadth-$p$-uniform.
Moreover, all $\PP$-reductions and $\PP$-expansions of $G$ are 2-connected $c\,$-crossing-critical graphs of crossing number $\crg(G)$.
\end{lemma}
\begin{proof}
Consider any $i\in \{1,\ldots,m-1\}$, and for the first claim, suppose for a contradiction that $p>c$ or $T_i$ is not $(v_i,v_{i+1})$-breadth-$p$-uniform,
and thus there exists an edge $e\in E(T_i)$ not contained in any edge cut of size at most $q=\min(p,c)$ separating
$v_i$ from $v_{i+1}$.  Since the $(v_1,v_m)$-breadth of the support $S$ of $\PP$ is $p\ge q$, it follows that the $(v_1,v_m)$-breadth of $S-e$ is at least $q$.

Let $G_1$ be the subgraph obtained from $G$ by deleting the subgraph $S$ except for the vertices $v_1$ and $v_m$;
thus, $G_1$ and $S$ are $(v_1,v_m)$-fragments of $G$ and $G-e$ is obtained from $G$ by replacing $S$ by $S-e$.
By the criticality of $G$, the graph $G-e$ can be drawn in the plane with $c'<c$ crossings.
Moreover, either $q=c>c'$, or $q=p$ is the $(v_1,v_m)$-breadth of $S$.  By Corollary~\ref{cor-surgcut},
we conclude that $\crg(G)\le c'<c$, which is a contradiction.

Next, let us argue that $\PP$-reductions and $\PP$-expansions of $G$ are 2-connected $c\,$-crossing-critical graphs of crossing number $\crg(G)$.
If $p=1$, then the breadth-uniformity of the tiles of $\PP$ implies that $S$ is a path from $v_1$ to $v_m$,
and the claim is obviously true, since the $\PP$-reductions and $\PP$-expansions only replace $S$ by paths of different length.
Hence, suppose that $p\ge 2$.

Consider now a $\PP$-reduction $G_2$ of $G$, obtained by contracting the tile $T_i$ for some $i\in\{1,\ldots,m-1\}$,
and let $S_2$ be the corresponding subgraph obtained from $S$ by contracting the same tile.
Note that $G_2$ is $2$-connected.
Since all tiles of $\PP$ are breadth-$p$-uniform and $m\ge 3$ by the definition of a band,
it follows that $S_2$ is $(v_1,v_m)$-breadth-$p$-uniform. Since $G_2$ is obtained from $G$ by replacing $S$ by $S_2$ and vice versa,
Corollary~\ref{cor-surgcut} implies that $\crg(G_2)=\crg(G)$.

Consider any edge $e$ of $G_2$.  Note that $G-e$ is obtained from $G_2-e$ by replacing $S_2-e$ by $S-e$, and that
$S-e$ and $S_2-e$ have the same $(v_1,v_m)$-breadth ($p-1$ if $e\in E(S)$ and $p$ if $e\in E(G_1)$).
By Corollary~\ref{cor-surgcut} with $G_2-e$ playing the role of $G$ and $G-e$ playing the role of $G'$,
we conclude that $\crg(G_2-e)\le \crg(G-e)<c$.  Since this holds for
every edge $e\in E(G_2)$, it follows that $G_2$ is also $c$-crossing-critical.

Finally, consider a $\PP$-expansion $G_3$ of $G$, obtained by splitting the vertex $v_i$ to $v'_i$ and $v''_i$
for some $i\in\{2,\ldots,m-1\}$ and adding a plane $(v'_i,v''_i)$-breadth-$p$-uniform graph $T$.
Let $S_3$ be the graph obtained from $S$ by performing the same transformation, and note
that $S_3$ is $(v_1,v_m)$-breadth-$p$-uniform.
Observe that $G_3$ is $2$-connected, and that $G_3$ is obtained from $G$ by replacing $S$ by $S_3$ and vice versa.
Hence, Corollary~\ref{cor-surgcut} implies that $\crg(G_3)=\crg(G)$.

Consider any edge $e$ of $G_3$.  If $e\not\in E(T)$, then let $e'=e$, and otherwise let $e'$ be an arbitrary edge of $S$.
Note that $G-e'$ is obtained from $G_3-e$ by replacing $S_3-e$ by $S-e'$, and that
$S-e'$ and $S_3-e$ have the same $(v_1,v_m)$-breadth ($p-1$ if $e\in E(S_3)$ and $p$ if $e\in E(G_1)$).
By Corollary~\ref{cor-surgcut} with $G_3-e$ playing the role of $G$ and $G-e'$ playing the role of $G'$,
we conclude that $\crg(G_3-e)\le \crg(G-e')<c$.  Since this holds for
every edge $e\in E(G_3)$, it follows that $G_3$ is also $c$-crossing-critical.
\end{proof}

Ideally, we would like to follow the proof scheme of Lemma~\ref{lemma-neck}
also in the case of shelled sashes.
Unfortunately, this case is more involved, and requires additional preparatory work.
Compared to the easier case of a necklace, the important difference
in the case of a shelled band comes from the fact that the band may be drawn
not only in the ``straight way'' but also in the ``twisted way'' (recall Figure~\ref{fig:schemes}).
An indication that this is troublesome comes from the result of Hlin\v en\'y and Der\v{n}\'ar~\cite{DBLP:conf/compgeom/HlinenyD16},
who showed that determining the crossing number of a twisted planar tile is NP-complete (and in particular, this ``twisted'' crossing number is not
determined by the breadth or any other simple parameter).
Consequently, the analysis of shelled bands is significantly
more complicated than the relatively straightforward proof of
Lemma~\ref{lemma-neck}.  The same remark applies for the shelled fans.

Before we dive into technical details, let us present an informal outline of our
approach:
\begin{enumerate}
\item 
Having a very long shelled sash $\PP$ in our graph $G$, it is easy to see
that the isomorphism types of bounded-size tiles in $\PP$ must repeat.
Moreover, even bounded-length subsashes must have isomorphic repetitions.
The first idea is to shorten the sash between such repeated isomorphic
subsashes $\PP_1$ and $\PP_2$ by identifying the repeated pieces and discarding what was
between (see Definition~\ref{def:reduction} for more details).
If the repeated subsash is long enough, we can use
some rather easy connectivity properties of $\PP$ to show that this yields a
smaller graph $G_2$ of crossing number $\crg(G)$.
\item 
However, it is not clear that this reduced graph $G_2$ is $c\,$-crossing-critical.
Analogously to Lemma~\ref{lemma-neck}, for any edge $e\in E(G_2)$,
we would like to transform a drawing of $G-e$ with less than $c$ crossings to
a drawing of $G_2-e$ with less than $c$ crossings.  However, if the drawing
of $G-e$ uses some unique properties of the part $\PP_{12}$ of the sash between $\PP_1$ and
$\PP_2$, we have no way how to mimic this in the drawing of $G_2-e$.
This is especially troublesome if this part of $G-e$ is drawn in a twisted way,
since there is no simple description of what these ``unique properties'' might be
by the aforementioned NP-completeness result~\cite{DBLP:conf/compgeom/HlinenyD16}.

We overcome this difficulty by performing the described reduction only inside
``typical'' longer pieces which repeat elsewhere in the sash (see Definition~\ref{def:typical}).
Hence, in $G_1-e$ we have
many copies of $\PP_{12}$, and by an appropriate surgery, we can use one of them to
mimic the drawing of $\PP_{12}$ in $G-e$.
\item A further advantage of reducing within parts that repeat elsewhere is
that we can more explicitly describe the converse expansion operation,
as duplicating subsashes which already exist elsewhere in the (reduced) sash.
Moreover, a similar surgery argument can be used to show that
such expansions preserve the crossing number as well as the $c$-crossing-criticality.
\end{enumerate}

\subsection{Preparation for a surgery}

Before we can proceed with the execution of the outlined plan, we need two auxiliary results
that will enable us to apply Lemma~\ref{lemma-surgcut}.  The first one gives us control
over the breadth of the tiles.  Because of the possibility of a twisted drawing, we cannot argue that the shelled sashes in
crossing-critical graphs are breadth-uniform.  Instead, we show that we can always find a long sash satisfying the following
weaker edge-connectivity property.

Let $\PP$ be a sash in a plane graph and let $p$ be the
breadth of $\PP$.  We say that $\PP$ is \emph{weakly breadth-uniform} if every tile $T$ of $\PP$
has breadth exactly $p$; that is, denoting by $Q_1$ and $Q_2$ the paths of $\PP$ forming the left and the right border of $T$,
the graph $T-V(Q_1\cap Q_2)$ contains an edge cut of size $p$ separating $Q_1-V(Q_2)$ from $Q_2-V(Q_1)$.

\begin{lemma}\label{lemma-elink}
Let $w$ and $c$ be positive integers and $f:\mathbb{N}^2\to\mathbb{N}$ be an arbitrary function.
There exist integers $\ell$ and $w_0$ such that the following claim holds.
Let $G$ be a $2$-connected plane graph in which each edge has multiplicity at most $c$, and let $\PP=(P_1, \ldots, P_m)$ be a sash in $G$
with all tiles of size at most $w$.
If $m\ge \ell$, then for some $w'\le w_0$ and $k\le 3cw$, the graph $G$ contains a weakly breadth-uniform sash $\PP'$ of
length $f(k,w')$, breadth $k$, and with tiles of size at most $w'$, whose support is contained in the support of $\PP$.
Moreover, if $\PP$ is shelled, then $\PP'$ is shelled as well.
\end{lemma}
\begin{proof}
Let $k_0=3cw$.  Let $w_{k_0}=w$ and $n_{k_0}=f(k_0,w)$, and for $i=k_0-1,\ldots,0$, let $w_i=(n_{i+1}-1)w$ and $n_i=\max(n_{i+1},(f(i,w_i)-1)(n_{i+1}-1))$.
Finally, let $\ell=n_0$.
Observe that each tile of $\PP$ has at most $3cw$ edges, since a simple planar graph with at most $w$ vertices has at most $3w$ edges,
and every edge of $G$ has multiplicity at most $c$.  Consequently, each tile of $\PP$ has breadth at most $k_0$.

For $1\le i<j\le m$, let $\PP_{i,j}$ denote the subsash of $\PP$ delimited by $P_i$
and $P_j$ and let $p_{i,j}\le k_0$ be the breadth of $\PP_{i,j}$,
Note that if $i\le i' < j'\le j$, then $p_{i',j'}\ge p_{i,j}$.
Let $k$ be the largest integer such that $p_{i,j}=k$ for some indices $1\le i<j\le m$ satisfying $j\ge i+n_k-1$;
such an integer $k$ exists, since $m\ge \ell\ge n_{p_{1,m}}$.
If $k=k_0$, then the conclusion of the lemma holds for $\PP_{i,j}$ with $w'=w$; indeed, each tile of $\PP$ has breadth at most $k_0$,
and since $p_{i,j}=k_0$, each tile of $\PP_{i,j}$ must have breadth exactly $k_0$.

Suppose now that $k<k_0$.  Let $i_1=i$, and for $t=2, \ldots, f(k,w_k)$,
let $i_t$ be the minimum index greater than $i_{t-1}$ such that $p_{i_{t-1},i_t}=k$.  Note that $i_t\le i_{t-1}+n_{k+1}-1$ by the maximality
of $k$ and the fact that for $i\le i' < j'\le j$ we have $p_{i',j'}\ge p_{i,j}=k$. In particular,
$$i_{f(k,w_k)}\le i+(f(k,w_k)-1)(n_{k+1}-1)\le i+n_k\le j.$$
Consider the sash $\PP'=(P_{i_1}, P_{i_2}, \ldots, P_{i_{f(k,w_k)}})$.  The tiles of this sash have size
at most $(n_{k+1}-1)w=w_k$.  If $\PP$ is shelled, then the sash $\PP'$ is clearly shelled as well.  The choice of the paths of $\PP'$ implies that every tile of $\PP'$ has breadth exactly $k$,
and in particular the sash $\PP'$ has breadth at most $k$.  On the other hand, the breadth of $\PP'$ is
$p_{i_1, i_f(k,w_k)}\ge p_{i,j}=k$; hence, $\PP'$ is weakly breadth-uniform and its breadth is $k\le k_0=3cw$.
Therefore, $k$, $w'=w_k$, and $\PP'$ satisfy the conclusion of the lemma.
\end{proof}

Let $\PP=(P_1, \ldots, P_m)$ be a shelled sash in a $2$-connected plane graph $G$, with tiles $T_1$, \ldots, $T_{m-1}$.
A tile $T_i$ of $\PP$ is \emph{internal} if $3\le i\le m-3$; a subsash of $\PP$ is \emph{internal} if all its tiles are internal.
Suppose that $G$ is the planarization of a graph $G_0$ drawn in the plane with crossings, such that the
crossing vertices are not contained in the support of $\PP$.  
For an internal tile $T$ of $\PP$ and a subgraph $G_1$ of $G_0$ containing $T$, the \emph{$T$-bridge-skeleton}
of $G_1$ is the graph consisting of
\begin{itemize}
\item the border cycle $C(T)$ of $T$ and
\item for each $T$-bridge $B$ of $G_1$, a single vertex adjacent to the attachments of the bridge.
\end{itemize}
We say that the $T$-bridge-skeleton is \emph{planarly realizable} if it has a plane drawing where $C(T)$ bounds a face.

In the following arguments, we are going
to consider graphs derived from $G_0$ by altering the (support of) the sash $\PP$ by the reduction or expansion
operation, and possibly deleting an edge when proving crossing-criticality.  The reduction and expansion preserve
the overall structure of $G_0$ (2-connectivity, the existence of a long sash), but the edge deletion requires
a bit of care.  Let us note the following property of internal tiles after possibly deleting an edge.
\begin{observation}\label{obs-exttile-bridges}
Let $G_0$ be a $2$-connected graph drawn in the plane and let $G$ be the planarization of $G_0$.
Let $\PP$ be a shelled sash in $G$ whose support $S$ does not contain any crossing vertices,
and let $T$ be an internal tile of $\PP$. Let $G_1$ be either $G_0$ or a graph obtained from $G_0$ by deleting an edge $e$
not belonging to $T$.  Then the $T$-bridge-skeleton of $G_1$ is planarly realizable.
\end{observation}
\begin{proof}
Let $U$ be the top vertex of $\PP$ if $\PP$ is a fan and $U=\emptyset$ if $\PP$ is a band.
Let $L$ be either the first or the second path of $\PP$, chosen so that $e\not\in E(L)$.
Let $R$ be either the last or the next to last path of $\PP$, chosen so that $e\not\in E(R)$.
Let $S'$ be the plane graph obtained from $S$ by adding an edge $e'$ drawn in the bottom face of $\PP$ and joining the ends of $L$ and $R$,
and by deleting the vertices and edges of $T$ not contained in $C(T)$.  

Note that $G_1$ is connected, since $G_0$ is $2$-connected, and thus every $T$-bridge of $G_1$ has at least one attachment.
Consider any $T$-bridge $B$ of $G_1$ which is not contained in $S$ and which has an attachment $v\not\in U$.  Since $B\not\subseteq S$,
there exists a path in $B-U$ from $v$ to a vertex or edge not contained in $S$; this path necessarily intersects $L$ or $R$, and since $e\not\in E(L\cup R)$,
$L\subseteq B$ or $R\subseteq B$.

Thus, letting $B_L$ be the $T$-bridge containing $L$ and $B$ be the $T$-bridge containing $R$ (where $B_L=B_R$ is possible),
every $T$-bridge of $G_1$ other than $B_L$ and $B_R$ is either contained in $S$ or has only one attachment, contained in $U$.
Thus, the $T$-bridge-skeleton of $G_1$ is a minor of the plane graph obtained from $S'$ by possibly adding pendant vertices at the vertex of $U$,
and it follows that it is planarly realizable.
\end{proof}

Let $G_0$ be a $2$-connected graph, let $\Gamma_0$ be a drawing of $G_0$ in the plane, and let $G$ be the planarization of $\Gamma_0$.
Let $\PP$ be a shelled sash in $G$ whose support does not contain any crossing vertices and let $T$ be an internal tile of $\PP$.
Let $G_1$ be a subgraph of $G_0$ containing $T$.  We say that a drawing of $G_1$ in the plane
is \emph{$T$-flat} (relative to the fixed drawing $\Gamma_0$ of $G_0$, which will always be clear from the context)
if there are no crossings on the edges of $C(T)$, $T$ is drawn in the closed disk bounded by $C(T)$ with the drawing
homeomorphic to the one induced by $\Gamma_0$, and the rest of $G_1$ is drawn in the unbounded face of $C(T)$.
We say that two tiles of $\PP$ are \emph{non-consecutive} if they are edge-disjoint (i.e., they are disjoint when $\PP$ is a band or intersect only in the top vertex when $\PP$ is a fan).
For a set $\TT$ of non-consecutive internal tiles of $\PP$, we say that a drawing is \emph{$\TT$-flat} if it is $T$-flat for all $T\in \TT$.
The following lemma enables us to focus on $T$-flat drawings.

\begin{lemma}\label{lemma-nicedex}
Let $G_0$ be a $2$-connected graph, let $\Gamma_0$ be a drawing of $G_0$ in the plane, and let $G$ be the planarization of $\Gamma_0$.
Let $\PP$ be a shelled sash in $G$ whose support $S$ does not contain any crossing vertices,
and let $T$ be an internal tile of $\PP$. Let $G_1$ be either $G_0$ or a graph obtained from $G_0$ by deleting an edge
not belonging to $T$, and let $\Gamma_1$ be a drawing of $G_1$ in the plane with $k$ crossings.
If no edge of $C(T)$ is crossed in the drawing $\Gamma_1$, then $G_1$ has a $T$-flat drawing $\Gamma'_1$ with at most $k$ crossings.
Moreover, the drawing $\Gamma'_1$ can be chosen so that
\begin{itemize}
\item for every edge $e\in E(G_1)$, if there is no crossing on $e$ in the drawing $\Gamma_1$, then there also is no crossing on $e$ in the drawing $\Gamma'_1$, and
\item for every internal tile $T'$ of $\PP$ non-consecutive with $T$ and contained in $G_1$, if the drawing $\Gamma_1$ is $T'$-flat,
then the drawing $\Gamma'_1$ is also $T'$-flat.
\end{itemize}
\end{lemma}
\begin{proof}
Let $R$ be the $T$-bridge-skeleton of $G_1$ and fix a plane drawing of $R$ such that the open disk $\Lambda$ bounded by $C(T)$ is a face; such a drawing exists by Observation~\ref{obs-exttile-bridges}.
For each $T$-bridge $B$ of $G_1$, let $v_B$ the vertex of $R$ representing $B$, and choose a closed disk $\Delta_B$ in the plane disjoint from $\Lambda$
which contains $v_B$ and the incident edges and is otherwise disjoint from the drawing of $R$.
Moreover, these closed disks are chosen so that for distinct $T$-bridges $B_1$ and $B_2$, the disks $\Delta_{B_1}$ and $\Delta_{B_2}$ intersect
exactly in the points representing the common attachments of $B_1$ and $B_2$.

Let $\Lambda_1$ and $\Lambda_2$ be the two regions to which the cycle $C(T)$ splits the plane in the drawing $\Gamma_1$.
Consider any $T$-bridge $B$ of $G_1$, and let $R_B$ be the plane subgraph of $R$ induced by $V(C(T))\cup v_B$.  Let $e_B$ be an edge of $C(T)$
incident with the outer face of $R_B$ in the drawing induced by $\Gamma_R$.  Let $\Gamma'_B$ be the drawing of $B\cup C(T)$ induced by $\Gamma_1$,
let $\Lambda_B\in\{\Lambda_1,\Lambda_2\}$ be the region containing the drawing of $B-V(C(T))$, and let $F_B$ be the face of $B\cup C(T)$ in the drawing $\Gamma'_B$
contained in $\Lambda_B$ and incident with the edge $e_B$.
\begin{itemize}
\item If $F_B$ is the outer face of $B\cup C(T)$, then let $\Gamma''_B=\Gamma'_B$.
\item If $F_B$ is contained in the open disk bounded by $C(T)$, then let $\Gamma''_B$ be a drawing obtained from $\Gamma'_B$
by performing the circular inversion around a point in $F_B$.
\item If $F_B$ is outside the open disk bounded by $C(T)$, but not the outer face of $B\cup C(T)$, then let $\Gamma''_B$ be obtained by performing the circular inversion
around a point in $F_B$, then flipping the drawing so that $C(T)$ stays in the same orientation around its interior as in $\Gamma'_B$.
\end{itemize}
This ensures that in the drawing $\Gamma''_B$, $B-V(C(T))$ is drawn outside of the open disk bounded by $C(T)$ and that $e_B$ is incident with the outer face
of $B\cup C(T)$ in this drawing.  Moreover, consider any internal tile $T'$ of $\PP$ non-consecutive with $T$ such that $T'\subset B$ and the drawing $\Gamma_1$ is $T'$-flat.
By the $T'$-flatness, the cycle $C(T)$ is drawn outside of the open disk bounded by $C(T')$ in $\Gamma_1$, and since $E(T\cap T')=\emptyset$,
the edge $e_B$ is drawn outside of the closure of the open disk bounded by $C(T')$.  Consequently, the face $F_B$ is outside of the open disk bounded by $C(T')$
in the drawing $\Gamma'_B$, and thus the transformations used to obtain $\Gamma''_B$ do not exchange the interior and exterior of $C(T')$.
Consequently, the drawing  $\Gamma''_B$ is also $T'$-flat.

Finally, let $\Gamma_B$ be a drawing of $B\cup C(T)$ obtained by deforming $\Gamma''_B$ continuously so that the cycle $C(T)$ is drawn in exactly the same way as in the plane drawing of $R$
and so that the drawing of $B$ induced by $\Gamma_B$ is contained in $\Delta_B$; this is possible, since $B-V(C(T))$ is drawn outside of the open disk bounded by $C(T)$ and $e_B$ is incident with the outer face
of $B\cup C(T)$ in the drawing $\Gamma''_B$.

Let $\Gamma'_1$ consist of the union of the drawings $\Gamma_B$ for the $T$-bridges $B$ of $G_1$ and of a drawing of $T$ in the closed disk bounded by $C(T)$
homeomorphic to the one induced by $\Gamma_0$.  Then $\Gamma'_1$ is clearly $T$-flat. Moreover, the analysis above implies that
it is $T'$-flat for every internal tile $T'$ of $\PP$ non-consecutive with $T$ and contained in $G_1$ such that the drawing $\Gamma_1$ is $T'$-flat.
Finally, note that the construction of $\Gamma'_1$ does not introduce any crossings not present in $\Gamma_1$ (though it could possibly
eliminate some crossings between different $T$-bridges), and thus $\Gamma'_1$ has at most $k$ crossings and all edges not crossed in $\Gamma_1$
are also not crossed in $\Gamma'_1$.
\end{proof}

The last part of the conclusion of Lemma~\ref{lemma-nicedex} allows us to apply the lemma iteratively for extended-non-consecutive tiles without spoiling
the flatness with respect to the previously processed ones.  Thus, we obtain the following conclusion.
\begin{corollary}\label{cor-nicedex}
Let $G_0$ be a $2$-connected graph drawn in the plane and let $G$ be the planarization of $G_0$.
Let $\PP$ be a shelled sash in $G$ whose support does not contain any crossing vertices,
and let $\TT$ be a set of non-consecutive internal tiles of $\PP$.
Let $G_1$ be either $G_0$ or a graph obtained from $G_0$ by deleting an edge not belonging to $\bigcup \TT$.
For every positive integer $k$, if $G_1$ has a drawing in the plane with at most $k$ crossings such that no edge of $\bigcup_{T\in \TT} C(T)$ is crossed,
then it also has a $\TT$-flat one.
\end{corollary}


\subsection{Reductions in shelled sashes}

Let us now define the reduction operation in shelled sashes.
\begin{definition}[reduction]\label{def-redu}
Let $G$ be a graph drawn in the plane with crossings, let $G'$ be the planarization of $G$,
and let $\PP$ be a shelled sash in $G'$ whose support does not contain any crossing vertices,
let $P_1$ and $P_2$ be distinct paths of $\PP$ of the same length, and let $S$ be the support of the subsash of
$\PP$ delimited by $P_1$ and $P_2$.  The \emph{$(\PP,P_1,P_2)$-reduction} $G_1$ of $G$ is the graph obtained from $G$ by
removing the vertices and edges of $S$ not contained in $P_1\cup P_2$, then identifying $P_1$ with $P_2$.
We turn the drawing of $G$ to a drawing of $G_1$ with the same number of crossings in the natural way,
stretching the drawing through the face created by the removal of $S$.  The \emph{$(P_1,P_2)$-reduction} of $\PP$
is the sash in the planarization of $G_1$ obtained from $\PP$ by removing the paths strictly between $P_1$ and $P_2$
and by identifying $P_1$ with $P_2$.
\end{definition}

Unlike the necklace case, it is not in general true that a $(\PP,P_1,P_2)$-reduction preserves
the crossing-criticality, or even the crossing number---the combination of the tile preceding $P_1$ and the one following $P_2$
may contain a feature (e.g., a small cut) that does not appear in $G$ and makes the reduction easier to draw.
To avoid this issue, we perform the reduction only when sufficiently long parts of the sash preceding $P_1$ and $P_2$
are identical, in the sense of the following definitions.

\begin{definition}[isomorphic tiles and subsashes, repetitions]\label{def:reduction}
Let $\PP$ be a shelled sash in a plane graph $G'$.
Tiles $T_1$ and $T_2$ of $\PP$ are \emph{isomorphic} if there exists a
homeomorphism of the plane mapping $T_1$ to $T_2$ and the left/top/right/bottom
border of $T_1$ to the left/top/right/bottom border of $T_2$.

Subsashes $\PP_1=(P_1, \ldots, P_m)$ and $\PP_2=(P'_1,\ldots,P'_m)$
of $\PP$ are \emph{isomorphic} if for $i\in \{1,\ldots,m-1\}$,
the tile in $\PP_1$ delimited by $P_i$ and $P_{i+1}$ is isomorphic to the
tile in $\PP_2$ delimited by $P'_i$ and $P'_{i+1}$.
Equivalently, there exists a homeomorphism mapping the
support $S_1$ of $\PP_1$ to the support $S_2$ of $\PP_2$ and the paths of $\PP_1$
to the paths of $\PP_2$ in order, such that the vertices of $S_1$ incident with the bottom face $F_1$ of the sash $\PP$
are mapped to the vertices of $S_2$ also incident with the face $F_1$.

If additionally $S_1$ and $S_2$ are disjoint (except for the top vertex when $\PP$ is a fan),
the subsashes $\PP_1$ and $\PP_2$ are internal, and $\PP_2$ appears after $\PP_1$ in $\PP$, then we say that $(\PP_1,\PP_2)$
is a \emph{repetition}.  The \emph{order} of this repetition is $m$, and the \emph{span} of the repetition is the subsash
of $\PP$ between the first path of $\PP_1$ and the last path of $\PP_2$.
More generally, if $(\PP_1,\PP_2)$, $(\PP_2,\PP_3)$, \ldots, and $(\PP_{a-1},\PP_a)$ are repetitions,
then we say that $(\PP_1,\ldots,\PP_a)$ is an \emph{$a$-repetition}, whose \emph{order} and \emph{span} is defined
to be the order and the span, respectively, of the repetition $(\PP_1,\PP_a)$.

If $G'$ is the planarization of a drawing of a graph $G$ in the plane such that
the support of the sash $\PP$ does not contain any crossing vertices,
then let \emph{$(\PP,\PP_1,\PP_2)$-reduction} of $G$ mean the $(\PP,P_m,P'_m)$-reduction of $G$,
and let \emph{$(\PP_1,\PP_2)$-reduction} of $\PP$ mean the $(P_m,P'_m)$-reduction of $\PP$.
\end{definition}

When the reduction is performed on a repetition of large enough order, it is easy to argue that the crossing number
cannot decrease.

\begin{lemma}\label{lemma-reducross}
Let $\Gamma$ be a drawing of a $2$-connected graph $G$ in the plane with the minimum number $c_0$ of crossings.
Let $G'$ be the planarization of $\Gamma$ and let $\PP$ be a shelled sash in $G'$ whose support does not
contain any crossing vertices.  Let $(\PP_1,\PP_2)$ be a repetition in $\PP$ of order $m$
and let $H$ with its drawing $\Gamma_H$ be the $(\PP,\PP_1,\PP_2)$-reduction of $G$.  If $m\ge 4c_0-2$, then $\crg(H)=c_0$.
\end{lemma}
\begin{proof}
The drawing $\Gamma_H$ has exactly $c_0$ crossings, and thus $\crg(H)\le c_0$.
Suppose for a contradiction that $H$ has a drawing $\Gamma'_H$ with less than $c_0$ crossings.

Since the subsash $\PP_1$ has length at least $4c_0-2$, it contains $2c_0-1$ pairwise non-consecutive (and thus edge-disjoint) tiles.
Each crossing of $\Gamma'_H$ belongs to at most two of these tiles, and thus there exists a tile
$T$ of $\PP_1$ such that the edges of $T$ are not crossed in $\Gamma'_H$.  By Lemma~\ref{lemma-nicedex},
we can assume that the drawing $\Gamma'_H$ is $T$-flat.
Let $\PP_T$ denote the subsash of $\PP$ between (and including) $T$ and the isomorphic copy of $T$ in $\PP_2$.
We can transform $\Gamma'_H$ to a drawing of $G$ by replacing the drawing of the tile $T$ by a drawing of the support of $\PP_T$
homeomorphic to the one induced by $\Gamma$.

This creates no new crossings, giving a drawing of $G$ with less than $c_0$ crossings, which is a contradiction.
\end{proof}

It is somewhat harder to show that the reduction $G_1$ is $c$-crossing-critical, and indeed, just the presence of a repetition does not seem to be sufficient.
Of course, the basic idea would be to show that for each edge $e\in E(G_1)$, a drawing of $G-e$ with less than $c$ crossings can be turned
into a drawing of $G_1-e$ with less than $c$ crossings.  The main issue is that if the sash is drawn in $G-e$ in a ``twisted'' way with the twist
occurring in the subsash between $\PP_1$ and $\PP_2$, it is not clear how to replicate the twist in $G_1-e$.  To deal with this issue,
we are going to need an additional assumption that this whole subsash is repeated elsewhere as well.
\begin{definition}[nested repetitions]\label{def:typical}
Let $\PP$ be a shelled sash in a plane graph~$G'$.
Let $(\PP_{1,1},\PP_{1,2},\ldots,\PP_{1,a})$ be an $a$-repetition in $\PP$ of order $m$ and let $\PP_1$ be the span of this repetition.
Suppose that $(\PP_1,\ldots,\PP_b)$ is a $b$-repetition in $\PP$; thus, for $i\in\{2,\ldots,b\}$, the subsashes
$\PP_{i,1},\ldots,\PP_{i,a}$ of $\PP_i$ corresponding to $\PP_{1,1},\PP_{1,2},\ldots,\PP_{1,a}$ form an $a$-repetition.
In this situation, we say that the $b$-tuple $((\PP_{1,1},\ldots,\PP_{1,a}), \ldots, (\PP_{b,1},\ldots,\PP_{b,a})$ is a \emph{nested $(b,a)$-repetition} in $\PP$
of \emph{order} $m$.  The \emph{span} of this nested repetition is defined to be the span of the $b$-repetition $(\PP_1,\ldots,\PP_b)$.
\end{definition}

Let us give a formal definition of what we mean by a ``twist''.
Let $\Gamma_1$ and $\Gamma_2$ be drawings of two (not necessarily distinct) graphs $G_1$ and $G_2$ in the plane, and let $C_1$ and $C_2$ be cycles
appearing both in $G_1$ and $G_2$.  Suppose that $C_1$ and $C_2$ intersect in at most one vertex
(the reader should imagine $C_1$ and $C_2$ are border cycles of non-consecutive tiles of sashes) and their edges are uncrossed both in $\Gamma_1$ and $\Gamma_2$.
For $i\in \{1,2\}$, consider $C_1\cup C_2$ as the plane graph with the drawing induced by $\Gamma_i$, and let $F_i$ be its face incident with both $C_1$ and $C_2$.  Orient the cycles $C_1$ and $C_2$
so that the face $F_1$ is to the right of both $C_1$ and $C_2$ (e.g., if the open disks bounded by $C_1$ and $C_2$ in the drawing $\Gamma_1$
are disjoint, then both $C_1$ and $C_2$ are oriented counterclockwise).  With this fixed orientation, if the face $F_2$ is to the left
of one of the cycles $C_1$ and $C_2$ and to the right of the other one, we say that the drawing $\Gamma_2$ is
\emph{$(C_1,C_2)$-twisted} with respect to $\Gamma_1$; otherwise, it is \emph{$(C_1,C_2)$-straight} with respect to $\Gamma_1$.

Let us now introduce a basic operation on sashes.  For $i\in \{1,2\}$, let $G_i$ be a graph drawn in the plane with crossings,
and let $\PP_i$ be a shelled sash in the planarization of $G_i$ whose support does not contain any crossing vertices. 
Let $G'_2$ be a subgraph of $G_2$.
\begin{itemize}
\item For $i\in \{1,2\}$, let $T_{i,1}$ and $T_{i,2}$ be tiles of $\PP_i$ such that $T_{i,2}$ appears after $T_{i,1}$ in $\PP_i$
($T_{i,1}=T_{i,2}$ is possible); and moreover, the tiles $T_{1,1}$, $T_{1,2}$, $T_{2,1}$, and $T_{2,2}$ are isomorphic. 
Let $S_i$ be the support of the subsash of $\PP_i$ between the left border $Q_{i,1}$ of $T_{i,1}$ and the right border $Q_{i,2}$ of $T_{i,2}$,
and suppose that $S_2\subseteq G'_2$.  Let $G_3$ be the graph obtained from $G'_2$ by removing the vertices and edges of $S_2$ not contained in $Q_{2,1}\cup Q_{2,2}$,
adding a copy $S'_1$ of $S_1$, and identifying the copy of the path $Q_{1,j}$ in $S'_1$ with the path $Q_{2,j}$ for $j\in \{1,2\}$.
We say that $G_3$ is obtained from $G'_2$ by the \emph{$((\PP_2,T_{2,1},T_{2,2})\to (\PP_1,T_{1,1},T_{1,2}))$-transplantation}.
\item In case that $T_{2,1}=T_{2,2}$, we say that $G_3$ is obtained from $G'_2$ by \emph{$(T_{2,1}\to (\PP_1,T_{1,1},T_{1,2}))$-transplantation}.
We refer to this special case as a \emph{single-tile transplantation}.
\end{itemize}
\begin{observation}\label{obs-prescrossing}
For $i\in \{1,2\}$, let $G_i$ be a graph, let $\Gamma_i$ be a drawing of $G_i$ in the plane with crossings,
and let $\PP_i$ be a shelled sash in the planarization of $\Gamma_i$ whose support does not contain any crossing vertices. 
For $i,j\in\{1,2\}$, let $T_{i,j}$ be a tile of $\PP_i$, where $T_{i,2}$ appears after $T_{i,1}$, and all four tiles are isomorphic.
For $i\in\{1,2\}$, let $S_i$ be the support of the subsash of $\PP_i$ between the left border $Q_{i,1}$ of $T_{i,1}$ and the right border $Q_{i,2}$ of $T_{i,2}$.
Let $G'_2$ be a subgraph of $G_2$ containing $S_2$ and let $G_3$ be the graph obtained from $G'_2$ by the $((\PP_2,T_{2,1},T_{2,2})\to (\PP_1,T_{1,1},T_{1,2}))$-transplantation.
Let $\TT$ be a set of pairwise non-consecutive tiles of $\PP_2$ contained in $G'_2$ but not in $S_2$ and non-consecutive to $T_{2,1}$ and $T_{2,2}$.
Let $\Gamma$ be a $(\TT\cup\{T_{2,1},T_{2,2}\})$-flat drawing of $G'_2$ in the plane with $c$ crossings.  Suppose that either
\begin{itemize}
\item $T_{2,1}=T_{2,2}$, or
\item the tiles $T_{2,1}$ and $T_{2,2}$ are non-consecutive, the drawing $\Gamma$ is $(C(T_{2,1}),C(T_{2,2}))$-straight with respect to $\Gamma_2$,
and $\PP_2$ is weakly breadth-uniform.
\end{itemize}
Then $G_3$ has a $\TT$-flat drawing in the plane with at most $c$ crossings.
\end{observation}
\begin{proof}
If $T_{2,1}=T_{2,2}$, then we can simply draw $S_1$ in the face of $G'_2$ created by the removal of the tile $T_{2,1}=T_{2,2}$, with the plane drawing induced by $\Gamma_1$.
Hence, suppose that $T_{2,1}\neq T_{2,2}$.

Let $p$ be the breadth of $\PP_2$; clearly, the $(Q_{2,1},Q_{2,2})$-breadth of $S_2$ is at least $p$.  Moreover, since $\PP_2$ is weakly breath-uniform
and the tile $T_{1,1}$ of $S_1$ is isomorphic to the tile $T_{2,2}$ of $S_2$, we conclude that the $(Q_{1,1},Q_{1,2})$-breadth of $S_1$ is at most~$p$.
Let $K$ be the subgraph of $G'_2$ obtained by deleting the vertices and edges of $S_2$ not contained in $Q_{2,1}\cup Q_{2,2}$,
and let $\Gamma_K$ be the drawing of $K$ induced by $\Gamma$.  Let us orient the path $Q_{2,1}$ away from the bottom face of the sash $\PP_2$ and the path $Q_{2,2}$ towards it.
Since the drawing $\Gamma$ is $(C(T_{2,1}),C(T_{2,2}))$-straight, we can assume (by flipping the drawing if necessary) that with this orientation, the paths $Q_{2,1}$ and $Q_{2,2}$ are 
right-exposed in the drawing $\Gamma_K$. Moreover, since the drawing $\Gamma$ is $\{T_{2,1},T_{2,2}\}$-flat, each edge of $S_2-E(Q_{2,1}\cup E(Q_{2,2}))$ incident with a vertex of $(Q_{2,1}\cup Q_{2,2})-V(Q_{2,1}\cap Q_{2,2})$
starts in the respective right window.  The claim then follows from Lemma~\ref{lemma-surgcut}.
\end{proof}
The fact that the drawing of $G_3$ is $\TT$-flat means that after a transplantation, we can further perform
an additional single-tile transplantation for each tile $T\in \TT$ (formally, for these additional transplantations,
the graph $G_2$ and the sash $\PP_2$ are replaced by the graph and sash arising from them by applying the preceding transplantations).

Having introduced these operations, let us make a brief side remark.  As we have seen in the introduction,
there are band-like constructions of crossing-critical graphs only depending on straight drawings: Deletion of an edge from the band
decreases its breadth, making it possible to draw another part of the graph across it in order to reduce the number of crossings.
However, for the fans, the possibility to twist the drawing is essential.
\begin{observation}\label{obs-fan-twisted}
Let $G$ be a $c$-crossing-critical graph and let $G'$ be the planarization of an optimal drawing $\Gamma_1$ of $G$ in the plane.
Let $\PP$ be a shelled weakly breadth-uniform fan in $G'$ whose support does not contain any crossing vertices.
Let $u$ be the top vertex of $\PP$, let $T_1$ and $T_2$ be non-consecutive isomorphic internal tiles appearing in $\PP$ in order,
let $P\in \PP$ be a path appearing in $\PP$ between the right border of $T_1$ and the left border of $T_2$,
and let $e$ be the edge of $P$ incident with $u$.  Let $\Gamma$ be a drawing of $G-e$ with
$k<c$ crossings.  If the drawing $\Gamma$ is $\{T_1,T_2\}$-flat, then it is $(C(T_1),C(T_2))$-twisted.
\end{observation}
\begin{proof}
Let $\Gamma_2$ be the drawing of $G-e$ induced by $\Gamma_1$.  Let $\PP_2=\PP\setminus\{P\}$ and note that $\PP_2$
is a shelled weakly breadth-uniform fan in the planarization of $\Gamma_2$; the weak breadth-uniformity follows from the assumption
that $e$ is incident with $u$.  The graph $G$ is obtained from $G-e$ by the $((\PP_2,T_1,T_2)\to (\PP,T_1,T_2))$-transplantation.
Thus, if the drawing $\Gamma$ were $(C(T_1),C(T_2))$-straight, Observation~\ref{obs-prescrossing} would imply that $\crg(G)\le k<c$,
contradicting the assumption that $G$ is $c$-crossing-critical.
\end{proof}

We are now ready to state the key lemma used to prove criticality of a reduction within a nested repetition of sufficiently large order
(the statement includes a variant without edge removal which will be useful later).

\begin{lemma}\label{lemma-redurem}
Let $G$ be a $2$-connected graph, let $\Gamma_G$ be a drawing of $G$ in the plane with crossings, let $G'$ be the planarization of $\Gamma_G$,
and let $\PP$ be a weakly breadth-uniform shelled sash in $G'$ whose support does not contain any crossing vertices.
Let $((\PP_{1,1},\PP_{1,2}), (\PP_{2,1},\PP_{2,2}))$ be a nested $(2,2)$-repetition in $\PP$ of order $m$ for an integer $m\ge 2$,
and let $S$ be the support of the span of this nested repetition.
Let us fix an integer $m'\in\{1,\ldots,m-1\}$ and for $i,j\in \{1,2\}$, let $T_{i,j}$ be the $m'$-th tile of $\PP_{i,j}$.
Let $H$ be the $(\PP,\PP_{2,1},\PP_{2,2})$-reduction of $G$.
Let $G_0$ and $H_0$ be either $G$ and $H$, or $G-e$ and $H-e$ for an edge $e\in E(G)\setminus E(S)$.
If $G_0$ has a $\{T_{1,1},T_{1,2},T_{2,1}, T_{2,2}\}$-flat drawing $\Gamma$ with $c_0$ crossings, then $\crg(H_0)\le c_0$.
\end{lemma}
\begin{proof}
Let $\Gamma_H$ be the drawing of the $(\PP,\PP_{2,1},\PP_{2,2})$-reduction $H$ naturally obtained from $\Gamma_G$, let $\PP_H$ be the $(\PP_{2,1},\PP_{2,2})$-reduction of $\PP$,
and let $T$ be the tile of $\PP_H$ corresponding to $T_{2,1}$ and $T_{2,2}$.

If the drawing $\Gamma$ is $(C(T_{i,1}),C(T_{2,2}))$-straight with respect to $\Gamma_G$ for some $i\in\{1,2\}$,
then note that $H_0$ is obtained from $G_0$ by the $((\PP,T_{i,1},T_{2,2})\to (\PP_H, T_{i,1}, T))$-transplantation,
and thus $\crg(H_0)\le c_0$ by Observation~\ref{obs-prescrossing}.

\begin{figure}
\begin{center}
\includegraphics[width=0.8\textwidth]{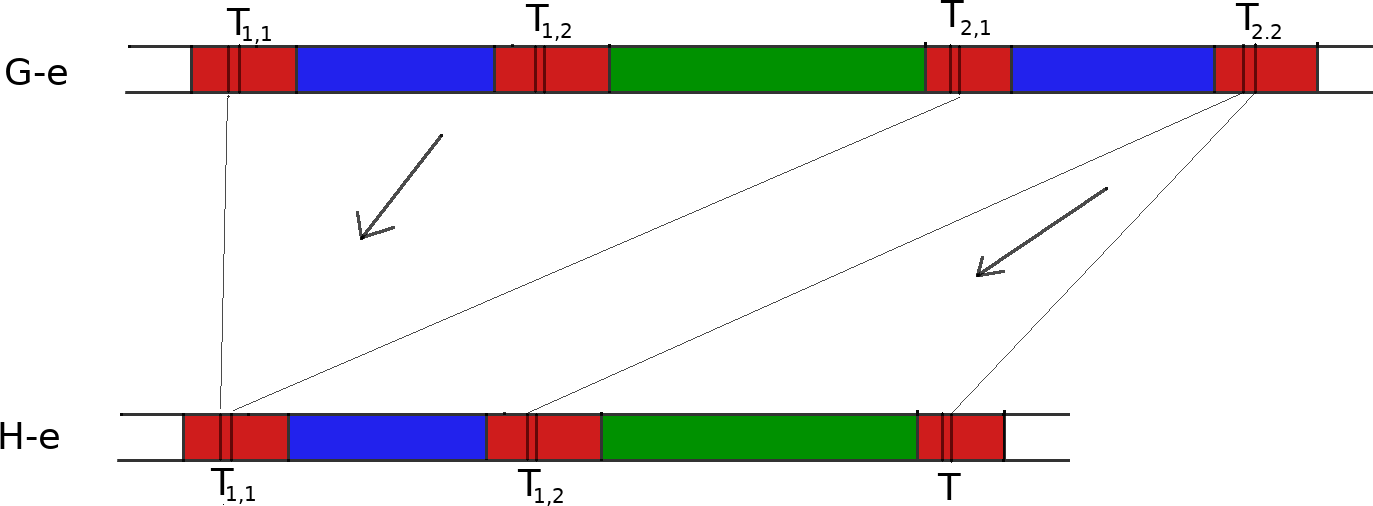}
\end{center}
\caption{An illustration of the double surgery from the proof of Lemma~\ref{lemma-redurem}.}\label{fig-crit-surgery}
\end{figure}

Hence, we can assume that $\Gamma$ is $(C(T_{1,1}),C(T_{2,2}))$- and $(C(T_{2,1}),C(T_{2,2}))$-twisted with respect to $\Gamma_G$.
Consequently, $\Gamma$ is $(C(T_{1,1}),C(T_{2,1}))$-straight with respect to $\Gamma$.
In this case, we perform two transplantations to turn $G_0$ into $H_0$, taking advantage of the nested repetition;
see Figure~\ref{fig-crit-surgery} for an illustration. First, we perform the $((\PP,T_{1,1},T_{2,1})\to (\PP_H, T_{1,1}, T_{1,1}))$-transplantation,
then the single-tile $(T_{2,2}\to (\PP_H, T_{1,2}, T))$-transplantation.  Again, $\crg(H_0)\le c_0$ follows by Observation~\ref{obs-prescrossing}.
\end{proof}

With Lemma~\ref{lemma-redurem}, we can easily conclude that a $(\PP,\PP_1,\PP_2)$-reduction preserves the crossing-criticality,
as long as the subsash between $\PP_1$ and $\PP_2$ (inclusive) has sufficiently many repetitions.

\begin{corollary}\label{cor-reducrit}
Let $G$ be a $2$-connected $c\,$-crossing-critical graph and let $\Gamma_G$ be a drawing of $G$ in the plane
with the minimum number of crossings.
Let $G'$ be the planarization of $G$, let $\PP$ be a weakly breadth-uniform shelled sash in $G'$ whose support
does not contain any crossing vertices, and let $k$ be the breadth of $\PP$.  For an integer $m$, let $((\PP_{1,1},\PP_{1,2}), (\PP_{2,1},\PP_{2,2}), (\PP_{3,1},\PP_{3,2}))$
be a nested $(3,2)$-repetition in $\PP$ of order $m$.
If $m\ge 8\fnlab{thm-crit-cno}(c)$, then the $(\PP,\PP_{2,1},\PP_{2,2})$-reduction $H$ of $G$ is $c$-crossing-critical.
\end{corollary}
\begin{proof}
Let $c_0$ be the number of crossings in the drawing $\Gamma_G$; by Theorem~\ref{thm-crit-cno}, we have $c_0\le \fnlab{thm-crit-cno}(c)$,
and thus $m\ge 4c_0-2$.  By Lemma~\ref{lemma-reducross}, we have $\crg(H)=c_0$.

Let $\PP_H$ be the $(\PP_{2,1},\PP_{2,2})$-reduction of the sash $\PP$, and let $\PP'_2$ be the subsash of $\PP_H$
corresponding to $\PP_{2,1}$ (and $\PP_{2,2}$).  Let us now consider any edge $e\in E(H)$.  
Observe that there exists a subsash $\QQ'_2$ of $\PP'_2$ of length $4c_0$ whose support
does not contain~$e$.  By symmetry, we can assume that if $e$ is contained in the support of $\PP_H$, then
it appears after the subsash $\QQ'_2$; otherwise, we can flip the drawings of $G$ and $H$ and consider
the reversed sashes.  For $i,j\in\{1,2\}$, let $\QQ_{i,j}$ be the subsash of $\PP_{i,j}$ corresponding to $\QQ'_2$.
Hence, $((\QQ_{1,1},\QQ_{2,1}), (\QQ_{2,1},\QQ_{2,2})$ is a nested $(2,2)$-repetition in $\PP$ of order $4c_0$;
let $S$ be the support of its span.  Observe that $H$ is also a $(\PP,\QQ_{2,1},\QQ_{2,2})$-reduction of $G$.
From this perspective $e$ corresponds to a unique edge of $G$, which we also denote by $e$, and this edge is
not contained in $S$.

To show that $H$ is crossing-critical, it suffices to argue that $\crg(H-e)\le \crg(G-e)$.
Let $\Gamma$ be a drawing of $G-e$ with $\crg(G-e)<c$ crossings.
Since each crossing belongs to at most two non-consecutive tiles,
observe that there exists $m'\in \{1,3,5,\ldots,4c-3\}$ such that for every $i,j\in\{1,2\}$, no edge of the $m'$-th tile $T_{i,j}$
of $\QQ_{i,j}$ is crossed in the drawing $\Gamma$.
By Corollary~\ref{cor-nicedex}, we can assume that the drawing $\Gamma$ is $\{T_{1,1},T_{1,2},T_{2,1}, T_{2,2}\}$-flat.
Therefore, $\crg(H-e)\le \crg(G-e)<c$ by Lemma~\ref{lemma-redurem}.
\end{proof}

\subsection{Expansions in shelled sashes}

Next, let us consider a converse operation to reduction.

\begin{definition}[expansion]\label{def-expand}
Let $G$ be a graph, let $\Gamma$ be a drawing of $G$ in the plane with crossings, let $G'$ be the planarization of $\Gamma$,
and let $\PP$ be a shelled sash in $G'$ whose support does not contain any crossing vertices.
Let $(\PP_1,\PP_2,\PP_3)$ be a 3-repetition in $\PP$.  Let $S$ be the support of the subsash $\PP_{1,2}$ of $\PP$
between the first path of $\PP_1$ and the last path of $\PP_2$.  Let $Q_1$ and $Q_2$ be the
first and the last path of $\PP_3$.  The \emph{$(\PP,\PP_1,\PP_2,\PP_3)$-expansion}
of $G$ is the graph $M$ obtained from $G$ by removing the vertices and edges of the support of $\PP_3$ not contained
in $Q_1\cup Q_2$, adding a copy $S'$ of $S$, and identifying the two border paths of this copy of $S$ with
$Q_1$ and $Q_2$.  The drawing of $G$ is naturally turned into a drawing of $M$, by drawing $S$ in the
face previously occupied by the support of $\PP_3$ with a drawing homeomorphic to the one induced by $\Gamma$.
The \emph{$(\PP_1,\PP_2,\PP_3)$-expansion} of $\PP$ is the sash $\PP_M$ in the planarization of $M$
obtained from $\PP$ by replacing $\PP_3$ by the copy of $\PP_{1,2}$ contained in $S'$.
Let $\PP'_1$ and $\PP'_2$ be the copies of $\PP_1$ and $\PP_2$ in $S'$; then
$((\PP_1,\PP_2),(\PP'_1,\PP'_2))$ is a nested $(2,2)$-repetition in $\PP_M$, which we call the \emph{site}
of the $(\PP,\PP_1,\PP_2,\PP_3)$-expansion.  Let us remark that $G$ is the $(\PP_M,\PP'_1,\PP'_2)$-reduction of $M$.
\end{definition}
First, let us note that expansion does not affect the breadth of breath-uniform sashes.
\begin{lemma}\label{lemma-presbreadth-expand}
Let $G$ be a graph, let $G'$ be the planarization of a drawing~of~$G$ in the plane with crossings,
let $\PP$ be a weakly breadth-uniform sash in $G'$ whose support does not contain any crossing vertices,
and let $k$ be the breadth~of~$\PP$.  Let $(\PP_1,\PP_2,\PP_3)$ be a 3-repetition in $\PP$ of order at least $k$.
Let $M$ be the $(\PP,\PP_1,\PP_2,\PP_3)$-expansion of $G$ and let $\PP_M$ be the $(\PP_1,\PP_2,\PP_3)$-expansion of $\PP$.
Then $\PP_M$ is weakly breadth-uniform and its breadth is $k$.
\end{lemma}
\begin{proof}
Let $((\PP_1,\PP_2),(\PP'_1,\PP'_2))$ be the site of the expansion, let $R_{1,1}$ and $R_{1,2}$ be the first and the last path of $\PP'_1$,
and let $R_{2,1}$ and $R_{2,2}$ be the first and the last path of $\PP'_2$.

Each tile of $\PP_M$ is isomorphic to a tile of $\PP$, and thus it has breadth exactly~$k$.  Hence, it suffices to show
that $\PP_M$ has breadth at least $k$.  Let $Q_1$ and $Q_2$ be the first and the last path of $\PP_M$ and let $S$ be the support of $\PP_M$.
Suppose for a contradiction that there exists a set $Z\subseteq E(S)$ of size less than $k$ separating $Q_1-V(Q_2)$ from $Q_2-V(Q_1)$ in $S-V(Q_1\cap Q_2)$.

The subsash of $\PP_M$ between $Q_1$ and $R_{1,2}$ is isomorphic to the subsash of $\PP$ between $Q_1$ and the last path of $\PP_3$,
and thus it has breadth $k$.  Hence, $S-V(Q_1\cap Q_2)-Z$ contains a path $K_1$ from $Q_1-V(Q_2)$ to $R_{1,2}-V(Q_1\cap Q_2)$.
Similarly, $S-V(Q_1\cap Q_2)-Z$ contains a path $K_2$ from $Q_2-V(Q_1)$ to $R_{2,1}-V(Q_1\cap Q_2)$.  Moreover, the subsash of $\PP_M$ between $R_{1,1}$ and $R_{2,2}$
is isomorphic to the subsash of $\PP$ between the first path of $\PP_1$ and the last path of $\PP_2$,
and thus it has breath $k$.  Consequently, $S-V(Q_1\cap Q_2)-Z$ contains a path $K$ from $R_{1,1}$ to $R_{2,2}$.
Finally, since $\PP'_1$ and $\PP'_2$ have length at least $k$, for $i\in \{1,2\}$, there exists a path $K'_i$ of $\PP'_i$
disjoint from $Z$.  Thus, $K_1\cup (K'_1-V(Q_1\cap Q_2))\cup K\cup (K'_2-V(Q_1\cap Q_2))\cup K_2$ is a connected subgraph of $S-V(Q_1\cap Q_2)-Z$
intersecting both $Q_1-V(Q_2)$ and $Q_2-V(Q_1)$; this is a contradiction.
\end{proof}

Lemma~\ref{lemma-redurem} (with $M$ playing the role of $G$) now easily implies that the expansion operation does not decrease crossing number.
\begin{lemma}\label{lemma-expand-crg}
Let $G$ be a graph, let $G'$ be the planarization of an optimal drawing of $G$ in the plane with $c_0=\crg(G)$ crossings,
let $\PP$ be a weakly breadth-uniform sash in $G'$ whose support does not contain any crossing vertices,
and let $k$ be the breadth of $\PP$.  Let $(\PP_1,\PP_2,\PP_3)$ be a 3-repetition in $\PP$ of order at least $\max(k, 4c_0)$ and let
$M$ be the $(\PP,\PP_1,\PP_2,\PP_3)$-expansion of $G$.  Then $\crg(M)=c_0$.
\end{lemma}
\begin{proof}
Let $((\PP_1,\PP_2),(\PP'_1,\PP'_2))$ be the site of the expansion and let $\PP_M$ be the $(\PP_1,\PP_2,\PP_3)$-expansion of $\PP$.

Clearly $\crg(M)\le c_0$.  Suppose for a contradiction that $\crg(M)<c_0$, and thus there exists a drawing $\Gamma_M$ of $M$ in the plane with less than $c_0$ crossings.
Then there exists $m'\in \{1,3,\ldots,4c_0-3\}$ such that the edges of the $m'$-th tile of $\PP_1$, $\PP_2$, $\PP'_1$, and $\PP'_2$ are not crossed in $\Gamma_M$.
Let $T_1$, $T_2$, $T'_1$, and $T'_2$ denote these tiles.  By Corollary~\ref{cor-nicedex}, we can assume that $\Gamma_M$ is $\{T_1,T_2,T'_1,T'_2\}$-flat.
Moreover, the sash $\PP_M$ is weakly breadth-uniform by Lemma~\ref{lemma-presbreadth-expand}.
Since $G$ is the $(\PP_M,\PP'_1,\PP'_2)$-reduction of $M$, Lemma~\ref{lemma-redurem} implies that $\crg(G)<c_0$; this is a contradiction.
\end{proof}

The argument showing that the expansion operation preserves criticality is more involved, requiring us to use different transplantations depending
on the placement of a twist.  Let us start with an analogue of Lemma~\ref{lemma-redurem}.
\begin{figure}
\begin{center}
\includegraphics[width=\textwidth]{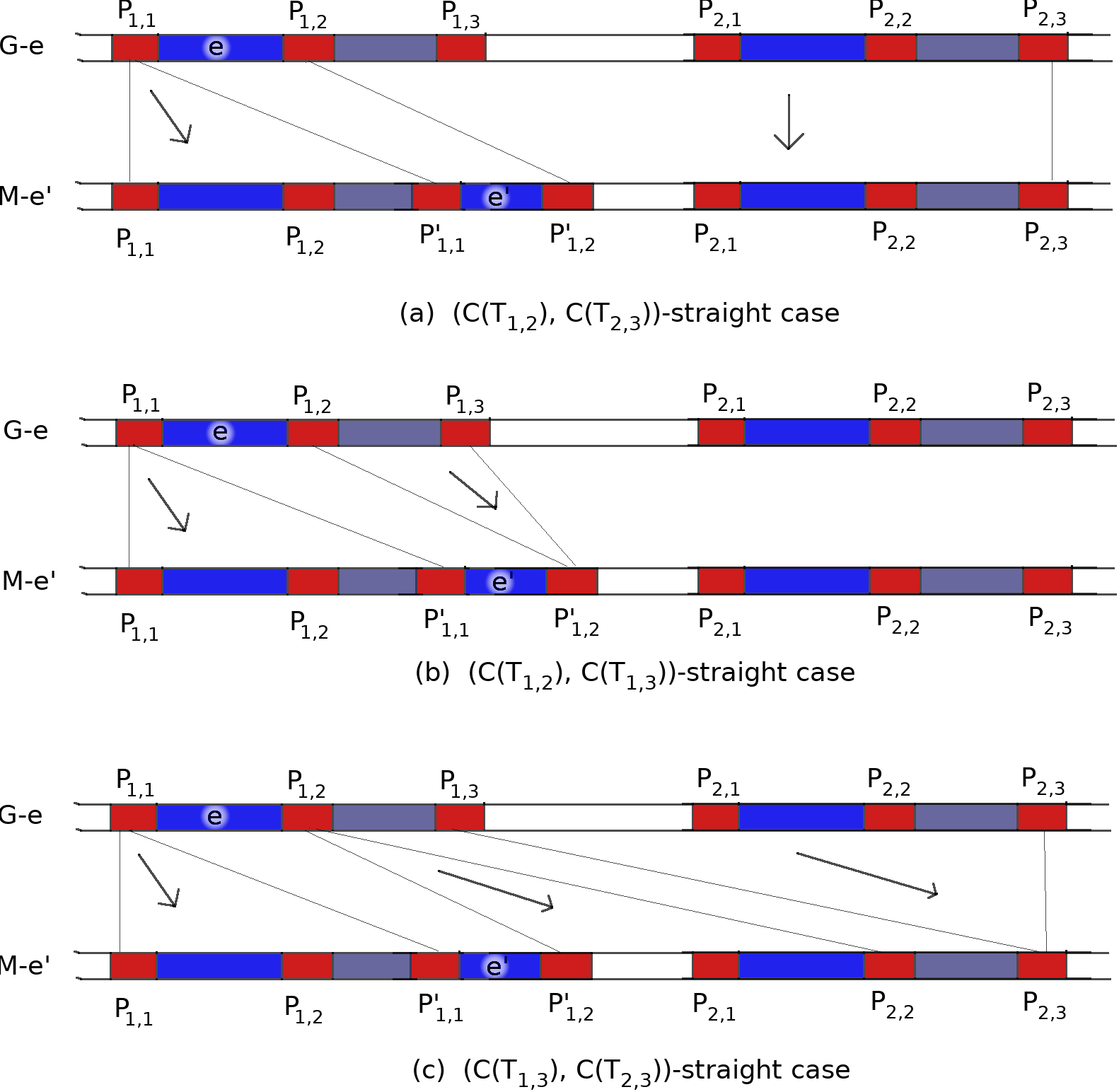}
\end{center}
\caption{An illustration of the surgeries from the proof of Lemma~\ref{lemma-exprem}.}\label{fig-exprem}
\end{figure}

\begin{lemma}\label{lemma-exprem}
Let $G$ be a $2$-connected graph and let $G'$ be the planarization of a drawing $\Gamma_G$ of $G$ in the plane with crossings.
Let $\PP$ be a weakly breadth-uniform shelled sash in $G'$ whose support does not contain any crossing vertices.
Let $((\PP_{1,1},\PP_{1,2},\PP_{1,3}), (\PP_{2,1},\PP_{2,2},\PP_{2,3}))$ be a nested $(2,3)$-repetition in $\PP$ of order $m$ for an integer $m\ge 2$.
Let us fix an integer $m'\in\{1,\ldots,m-1\}$ and for $i\in \{1,2\}$ and $j\in \{1,2,3\}$, let $T_{i,j}$ be the $m'$-th tile of $\PP_{i,j}$.
Let $M$ be the $(\PP,\PP_{1,1},\PP_{1,2},\PP_{1,3})$-expansion of $G$ with the drawing $\Gamma_M$ of $M$ naturally arising from the drawing of $G$,
let $\PP_M$ be the $(\PP_{1,1},\PP_{1,2},\PP_{1,3})$-expansion of $\PP$,
and let $((\PP_{1,1},\PP_{1,2}),(\PP'_{1,1},\PP'_{1,2}))$ be the site of the expansion.
Let $S$ be the support of the subsash of $\PP$ between the last path of $\PP_{1,1}$ and the first path of $\PP_{1,2}$,
and let $S'$ be the support of the isomorphic subsash of $\PP_M$ between the last path of $\PP'_{1,1}$ and the first path of $\PP'_{1,2}$.
Let $e'$ be an edge of $S'$ and let $e$ be the corresponding edge of $S$.
If $G-e$ has a $\{T_{i,j}:i\in\{1,2\},j\in\{1,2,3\}\}$-flat drawing $\Gamma$ with less than $c$ crossings, then $\crg(M-e')<c$.
\end{lemma}
\begin{proof}
Let $T'_{1,1}$ and $T'_{1,2}$ be the $m'$-th tiles of $\PP'_{1,1}$ and $\PP'_{1,2}$, respectively.
Let us distinguish three cases; in each of them, the described transplantations are followed by the single-tile $(T_{1,1}\to (\PP_M,T_{1,1}, T'_{1,1}))$-transplantation,
and the final results are sketched in Figure~\ref{fig-exprem}.
\begin{itemize}
\item If $\Gamma$ is $(C(T_{1,2}), C(T_{2,3}))$-straight, then we perform the $((\PP,T_{1,2},T_{2,3})\to (\PP_M, T'_{1,2}, T_{2,3})$-transplantation.
See Figure~\ref{fig-exprem}(a) for an illustration.
\item If $\Gamma$ is $(C(T_{1,2}), C(T_{1,3}))$-straight, then we perform the $((\PP,T_{1,2},T_{1,3})\to (\PP_M, T'_{1,2}, T'_{1,2})$-transplantation.
See Figure~\ref{fig-exprem}(b) for an illustration.
\item If $\Gamma$ is $(C(T_{1,2}), C(T_{2,3}))$-twisted and $(C(T_{1,2}), C(T_{1,3}))$-twisted, then it is $(C(T_{1,3},C(T_{2,3})))$-straight.
We perform the $((\PP,T_{1,3},T_{2,3})\to (\PP_M, T_{2,3}, T_{2,3})$-transplantation, and the single-tile  $(T_{1,2}\to (\PP_M, T'_{1,2}, T_{2,2}))$-transplantation.
See Figure~\ref{fig-exprem}(b) for an illustration.
\end{itemize}
Using the assumed drawings $\Gamma$ and $\Gamma_M$ with the transplantations,
in all three cases, we obtain a drawing of a graph isomorphic to $M-e'$ as the result of the transplantations, 
and we conclude that $\crg(M-e')<c$ by Observation~\ref{obs-prescrossing}.
\end{proof}

We can now argue about the crossing-criticality of the expansions.
\begin{corollary}\label{cor-expcrit}
Let $G$ be a $2$-connected $c\,$-crossing-critical graph and let $G'$ be the planarization of a drawing of $G$ in the plane with the smallest number $c_0$ of crossings.
Let $\PP$ be a weakly breadth-uniform shelled sash in $G'$ whose support does not contain any crossing vertices and let $k$ be the breadth of $\PP$.
For an integer $m$, let $((\PP_{1,1},\PP_{1,2},\PP_{1,3}), (\PP_{2,1},\PP_{2,2},\PP_{2,3})$ be a nested $(2,3)$-repetition in $\PP$ of order $m$.
Let $M$ be the $(\PP,\PP_{1,1},\PP_{1,2},\PP_{1,3})$-expansion of $G$.  If $m\ge \max(k,8\fnlab{thm-crit-cno}(c))$, then $M$ is $c\,$-crossing-critical.
\end{corollary}
\begin{proof}
By Theorem~\ref{thm-crit-cno}, we have $c_0\le \fnlab{thm-crit-cno}(c)$, and thus $m\ge 4c_0$.  By Lemma~\ref{lemma-expand-crg}, we have $\crg(M)=c_0$.
Hence, we just need to argue that $\crg(M-e')<c$ for every edge $e'\in E(M)$.
Let $\PP_M$ be the $(\PP_{1,1},\PP_{1,2},\PP_{1,3})$-expansion of $\PP$, and let $((\PP_{1,1},\PP_{1,2}),(\PP'_{1,1},\PP'_{1,2}))$ be the site of the expansion.
By Lemma~\ref{lemma-presbreadth-expand}, $\PP_M$ is weakly breadth-uniform and is breadth is $k$.  

There exists $s\in \{1,4c+1\}$ such that, letting $\QQ'_{1,j}$ for $j\in \{1,2\}$ denote the subsash of $\PP'_{1,j}$ of length $4c$ starting with the $s$-th path of $\PP'_{1,j}$,
the edge $e'$ is contained neither in the support of $\QQ'_{1,1}$ nor in the support of $\QQ'_{1,2}$.  For $i\in\{1,2\}$ and $j\in\{1,2,3\}$, let $\QQ_{i,j}$ be
the subsash of $\PP_{i,j}$ of length $4c$ starting with the $s$-th path of $\PP_{i,j}$.  Note that $M$ is the $(\PP,\QQ_{1,1},\QQ_{1,2},\QQ_{1,3})$-expansion~of~$G$.
Let $S$ be the support of the subsash of $\PP$ between the last path of $\QQ_{1,1}$ and the first path of $\QQ_{1,2}$,
and let $S'$ be the support of the isomorphic subsash of $\PP_M$ between the last path of $\QQ'_{1,1}$ and the first path of $\QQ'_{1,2}$.

If $e'\not\in E(S')$, then let $e$ be the corresponding edge of $G$ and consider a drawing $\Gamma_1$ of $G-e$ with less than $c$ crossings.
There exists a tile $T$ of $\QQ_{1,3}$ whose edges are not crossed in $\Gamma_1$, and by Lemma~\ref{lemma-nicedex}, we can assume that the drawing $\Gamma_1$ is $T$-flat.
Let $T_1$ and $T_2$ be the tiles of $\QQ'_{1,1}$ and $\QQ'_{1,2}$, respectively, corresponding to $T$.
Note that $M-e'$ is obtained from $G-e$ by the single-tile $(T\to (\PP_M,T_1,T_2))$-expansion, and thus $\crg(M-e')<c$ by Observation~\ref{obs-prescrossing}.

Hence, suppose that $e'\in E(S')$.  In this case, let $e$ be the corresponding edge of $S$, and consider a drawing $\Gamma_2$ of $G-e$ with less than $c$ crossings.
There exists $m'\in\{1,3,\ldots, 4c-3\}$ such that for $i\in \{1,2\}$ and $j\in \{1,2,3\}$, the $m'$-th tile $T_{i,j}$ of $\QQ'_{i,j}$ is not crossed in $\Gamma_2$.
By Corollary~\ref{cor-nicedex}, we can assume that the drawing $\Gamma_2$ is $\{T_{i,j}:i\in\{1,2\},j\in\{1,2,3\}\}$-flat.  By Lemma~\ref{lemma-exprem},
we have $\crg(M-e')<c$.
\end{proof}

\section{The structural theorem}\label{sec:main}

Lemma~\ref{lemma-neck} and Corollary~\ref{cor-expcrit} motivate the following definition.

\begin{definition}[nothing-new expansion]\label{def:expansion}
Let $G$ be a $2$-connected $c\,$-crossing-critical graph.
Let $\Gamma$ be a drawing of $G$ in the plane with the minimum number of crossings,
let $G'$ be the planarization of $\Gamma$ and let $\PP$ be a sash in $G'$ whose support does not contain any crossing vertices.
\begin{itemize}
\item Suppose $\PP$ is a necklace, $T$ is a tile of $\PP$, $M$ is a $\PP$-expansion of $G$ obtained by replacing
a vertex of $\PP$ by the copy of $T$, and $\Gamma_M$ is the drawing of $M$ naturally arising from $\Gamma$.
Then we say that $(M,\Gamma_M)$ is a \emph{nothing-new expansion of $(G,\Gamma)$}.
\item Suppose that $\PP$ is shelled and weakly breadth-uniform and let $k$ be the breadth of $\PP$.
Moreover, suppose that $((\PP_{1,1},\PP_{1,2},\PP_{1,3}), (\PP_{2,1},\PP_{2,2},\PP_{2,3}))$ is a nested $(2,3)$-repetition in $\PP$ of order
$\max(k,8\fnlab{thm-crit-cno}(c))$.  Let $M$ be the $(\PP,\PP_{1,1},\PP_{1,2},\PP_{1,3})$-expansion of $G$
and let $\Gamma_M$ be the drawing of $M$ naturally arising from $\Gamma$.  In this case, we also say that
$(M,\Gamma_M)$ is a \emph{nothing-new expansion of $(G,\Gamma)$}.
\end{itemize}
In both cases, the \emph{magnitude} of the nothing-new expansion is defined as the number of vertices of the support of $\PP$.
\end{definition}

By the pigeonhole principle, it is easy to see that a suitable nested repetition can be found in any sufficiently long sash.
More precisely, note the following observation.
\begin{observation}\label{obs-repet}
There exists a function $f_{\ref{obs-repet}}:\mathbb{N}^4\to \mathbb{N}$ such that the following claim holds.
Let $G$ be a $2$-connected graph with maximum edge multiplicity at most $c$ and let $G'$ be the planarization of a drawing of $G$ in the plane with crossings.
Let $\PP$ be a sash in $G'$ and let $w$ be the maximum number of vertices of a tile of $\PP$.  For all positive integers $a$ and $m$,
if the length of $\PP$ is at least $f_{\ref{obs-repet}}(a,m,w,c)$, then $\PP$ contains an $a$-repetition of order $m$.
\end{observation}
\begin{proof}
Let $q(w,c)$ be the number of isomorphism classes of tiles with at most $w$ vertices and with maximum edge multiplicity at most $c$.
Then, the number of isomorphism classes of subsashes of $\PP$ of length $m$ is $q'(w,c,m)=q^{m-1}(w,c)$.  Let us set $f_{\ref{obs-repet}}(a,m,w,c)=4+a\cdot q'(w,c,m)$.
If the length of $\PP$ is at least $f_{\ref{obs-repet}}(a,m,w,c)$, then $\PP$ contains $a\cdot q'(w,c,m)$ pairwise disjoint internal subsashes of length $m$.
By the pigeonhole principle, at least $a$ of them are pairwise isomorphic, and thus they form an $a$-repetition of order $m$.
\end{proof}

By iterating the observation, we get the desired result.
\begin{corollary}\label{cor-repet}
There exists a function $f_{\ref{cor-repet}}:\mathbb{N}^3\to \mathbb{N}$ such that the following claim holds.
Let $G$ be a $2$-connected graph with maximum edge multiplicity at most $c$ and let $G'$ be the planarization of a drawing of $G$ in the plane with crossings.
Let $\PP$ be a sash in $G'$ and let $w$ be the maximum number of vertices of a tile of~$\PP$.  For a positive integer $m$,
if the length of $\PP$ is at least $f_{\ref{cor-repet}}(m,w,c)$, then there exists a repetition $(\PP_1,\PP_2)$ in $\PP$
such that $\PP_1$ (and thus also $\PP_2)$ contains a nested $(2,2)$-repetition of order $m$.
\end{corollary}
\begin{proof}
Let $f_0(m,w,c)=m$ and for $i\ge 1$, let $f_i(m,w,c)=f_{\ref{obs-repet}}(2,f_{i-1}(m,w,c),w,c)$.
Finally, let us define $f_{\ref{cor-repet}}(m,w,c)=f_3(m,w,c)$.

If the length of $\PP$ is at least $f_{\ref{cor-repet}}(m,w,c)$, then by Observation~\ref{obs-repet},
it contains a repetition $(\PP_1,\PP_2)$ of order $f_2(m,w,c)$.  Applying Observation~\ref{obs-repet} to $\PP_1$,
we conclude that it contains a $2$-repetition $(\PP_{1,1},\PP_{1,2})$ of order $f_{\ref{obs-repet}}(a,m,w,c)$.  Applying Observation~\ref{obs-repet} one more time,
we obtain a $2$-repetition of order $m$ in $\PP_{1,1}$.  This $2$-repetition together with the corresponding $2$-repetition
in $\PP_{1,2}$ forms a nested $(2,2)$-repetition of order $m$ in $\PP_1$.
\end{proof}

We are now ready to state our main result.

\begin{theorem}
\label{thm-nn}\label{thm:mainexpansion}%
For all positive integers $c\le c_0$, there exists a positive integer~$n$ such that the following holds.
Let $G$ be a $2$-connected graph and let $\Gamma$ be a drawing of $G$ in the plane.  Then
the following claims are equivalent:
\begin{enumerate}[(i)]
\item\label{it:is_crit} The graph $G$ is $(c,c_0)\,$-crossing-critical and the drawing $\Gamma$ has $c_0$ crossings.
\item\label{it:is_constr} There exists a sequence $(G_0,\Gamma_0)$, \ldots, $(G_l,\Gamma_l)$
such that
\begin{itemize}
\item $G_0$ is a $2$-connected $(c,c_0)$-crossing-critical graph with at most $n$ vertices and $\Gamma_0$ is a drawing
of $G_0$ in the plane with $c_0$ crossings,
\item for $i\in\{1, \ldots, l\}$, $(G_i,\Gamma_i)$ is a nothing-new expansion of $(G_{i-1},\Gamma_{i-1})$ of magnitude at most $n$, and
\item $G=G_l$ and $\Gamma=\Gamma_l$.
\end{itemize}
\end{enumerate}
\end{theorem}
\begin{proof}
Let $m(k)=\max(k,8\fnlab{thm-crit-cno}(c))$.
Let $f(k,w')=f_{\ref{cor-repet}}(m(k),w',c)$ and for every positive integer $w$, let $\ell(w)$ be the value of $\ell$ from Lemma~\ref{lemma-elink} for $c$, $w$, and the function $f$.
Let $g(w)=\max(f_{\ref{obs-repet}}(2,2,w,c),\ell(w))$
and let $w_0$ and $n_0$ be the corresponding values from Corollary~\ref{cor-tiles}.  Finally, let $n=\max(n_0,g(w_0)w_0)$.

To prove that \ref{it:is_constr} implies \ref{it:is_crit}, we need to show by induction that for $i=0,\ldots, l$, the graph
$G_i$ is $2$-connected and $(c,c_0)\,$-crossing-critical and its drawing $\Gamma_i$ has $c_0$ crossings.
For $i=0$, this holds by the assumptions.  Suppose that $i\ge 1$.  The drawing $\Gamma_{i-1}$
has the same number $c_0$ of crossings as its nothing-new expansion $\Gamma_i$.  Moreover,
$\crg(G_i)=\crg(G_{i-1})=c_0$ and $G_i$ is $c$-crossing-critical by Lemma~\ref{lemma-neck}
or by Lemma~\ref{lemma-expand-crg} and Corollary~\ref{cor-expcrit}.

Let us now consider the implication from \ref{it:is_crit} to \ref{it:is_constr}.
We prove the claim by induction on the number of vertices of $G$.
If $|V(G)|\le n$, then \ref{it:is_constr} holds with $l=0$, $G=G_0$, and $\Gamma_0=\Gamma$.

Hence, suppose that $|V(G)|>n\ge n_0$.  By Corollary~\ref{cor-tiles}, there exists
a sash $\PP$ in the planarization of $\Gamma$ whose support does not contain any crossing vertices,
$\PP$ is either a necklace or shelled, and for some integer $w_1\le w_0$, all the tiles of $\PP$ have size at most $w_1$
and the length of $\PP$ is exactly $g(w_1)$.  In particular, the support $S$ of $\PP$ has at most $g(w_1)w_1\le g(w_0)w_0\le n$ vertices.

Suppose first that $\PP$ is a necklace.  By Observation~\ref{obs-multip}, every edge of $G$ has multiplicity at most $c$.
Since $g(w_1)\ge f_{\ref{obs-repet}}(2,2,w_1,c)$, the necklace $\PP$ contains a $2$-repetition $(\PP_1,\PP_2)$ of order $2$.
Let $H$ be the $\PP$-reduction of $G$ obtained by contracting the tile forming $\PP_2$, and let $\Gamma_H$ be the drawing of $G'$
obtained from $\Gamma$ in the natural way.  Then $(G,\Gamma)$ is a nothing-new expansion of $(H,\Gamma_H)$
of magnitude at most $|V(S)|\le n$.  Moreover, $H$ is $(c,c_0)$-crossing-critical by Lemma~\ref{lemma-neck},
and thus \ref{it:is_constr} follows by the induction hypothesis applied to $H$ and $\Gamma_H$.

Therefore, we can assume that $\PP$ is shelled.  Since the length of $\PP$ is $g(w_1)\ge \ell(w_1)$, Lemma~\ref{lemma-elink}
implies that for some integers $w'$ and $k$, the graph $G'$ contains a weakly breadth-uniform shelled sash $\PP'$ of
length $f(k,w')$, breadth~$k$, and with tiles of size at most $w'$, whose support $S'$ is contained in the support $S$ of $\PP$.
Since $f(k,w')=f_{\ref{cor-repet}}(m(k),w',c)$, there exists a repetition $(\PP_1,\PP_2)$ in the sash $\PP'$
such that $\PP_1$ contains a nested $(2,2)$-repetition $((\PP_{1,1,1},\PP_{1,1,2}),(\PP_{1,2,1},\PP_{1,2,2}))$ of order $m(k)$.
Let $((\PP_{2,1,1},\PP_{2,1,2}),(\PP_{2,2,1},\PP_{2,2,2}))$ be the corresponding nested $(2,2)$-repetition in $\PP_2$.
Let $H$ with drawing $\Gamma_H$ be the $(\PP',\PP_{1,2,1},\PP_{1,2,2})$-reduction of $G$ with drawing $\Gamma$,
and let $\PP_H$ be the $(\PP_{1,2,1},\PP_{1,2,2})$-reduction of $\PP'$.  Let $\PP'_{1,2,1}$ be the subsash of $\PP_H$
corresponding to $\PP_{1,2,1}$ (and $\PP_{1,2,2}$). Then $((\PP_{1,1,1},\PP_{1,1,2},\PP'_{1,2,1}),(\PP_{2,1,1},\PP_{2,1,2},\PP_{2,2,1}))$
is a nested $(2,3)$-repetition of order $m(k)$ in $\PP_H$, and $G$ with the drawing $\Gamma$ is obtained
from $H$ with the drawing $\Gamma_H$ by the $(\PP_H,\PP_{1,1,1},\PP_{1,1,2},\PP_{1,2,1})$-expansion.
We conclude that $(G,\Gamma)$ is a nothing-new expansion of $(H,\Gamma_H)$ of magnitude at most $|V(S')|\le |V(S)|\le n$.
Moreover, $H$ is $(c,c_0)$-crossing-critical by Lemma~\ref{lemma-reducross} and Corollary~\ref{cor-reducrit}.
Therefore, \ref{it:is_constr} follows by the induction hypothesis applied to $H$ and $\Gamma_H$.
\end{proof}

\section{Enumeration of crossing-critical graphs}\label{sec:alg}

Naturally, for any fixed positive integer $c$, Theorem~\ref{thm:mainexpansion} gives an efficient enumeration procedure
to generate all $2$-connected $c\,$-crossing-critical graphs (or $(c,c_0)$-crossing-critical graphs for any $c_0\ge c$) of at most a given order~$n$.
By straightforwardly following the statement, we could generate each such graph with all of its (non-homeomorphic) optimal drawings. This may not be
desirable, since there can exist many optimal drawings for each graph.  Let us now describe how this can be avoided, i.e., how to
only generate non-isomorphic $(c,c_0)$-crossing-critical graphs with an output-sensitive complexity.

We say that a graph $H$ is an \emph{$a$-alteration} of a graph $G$ if $H$ is obtained from~$G$ by removing at most $a$ vertices and incident edges, then adding
at most $a$ new vertices and at most $a$ edges.  We say that the $a$-alteration is \emph{special} if
\begin{itemize}
\item the removed vertices of $G$ have degree at most $a$ and induce a connected subgraph of $G$, and
\item each end of an added edge is either a new vertex or a neighbor of a removed vertex.
\end{itemize}
\begin{corollary}\label{cor:mainexpansion}
For all positive integers $c\le c_0$, there exists a positive integer~$a$ such that the following holds.
Let $G$ be a $2$-connected graph.  If $G$ is $(c,c_0)\,$-crossing-critical,
then there exists a sequence $G_0$, \ldots, $G_l=G$ of $2$-connected $(c,c_0)$-crossing-critical graphs
such that $|V(G_0)|\le a$ and for $i\in \{1, \ldots, l\}$, the graph $G_i$ is a special $a$-alteration of $G_{i-1}$
and $|V(G_i)|>|V(G_{i-1}|$.
\end{corollary}
\begin{proof}
Let $n$ be the value from Theorem~\ref{thm:mainexpansion} and let $a=3cn$.
We let $G_0$, \ldots, $G_l$ be the graphs from Theorem~\ref{thm:mainexpansion}\ref{it:is_constr}.
Note that for $i\in\{1,\ldots,l\}$, the graph $G_i$ is obtained from $G_{i-1}$ by picking a planar subgraph $S$
with at most $n$ vertices (the support of the sash in which we perform the nothing-new expansion), deleting a connected subgraph
formed by vertices of $S$ that have no neighbors outside of $S$, and adding a new planar subgraph $R$ with at most $n$ vertices, some of which are identified with the
neighbors of the deleted vertices.  Since $G_i$ is $c$-crossing-critical,
$R$ has maximum edge multiplicity at most $c$ by Observation~\ref{obs-multip}, and thus $|E(R)|\le 3cn=a$.  Therefore, $G_i$ is a special $a$-alteration of $G_{i-1}$.
\end{proof}

Let us remark that the number of connected subgraphs $K$ of $G$ induced by at most $a$ vertices of $G$ of degree (in~$G$) at most $a$
is $O_a(|V(G)|)$, since after fixing a vertex $v_0$ of $K$, every other vertex of $K$ is reachable from $v_0$ by a path of
at most $a$ vertices of $G$ of degree at most $a$, and the number of such paths is bounded by a function of $a$. 
Moreover, the number of neighbors of $K$ outside of $K$ is at most $a^2$.
Hence, the following claim holds.
\begin{observation}\label{obs-fewalt}
Let $a$ be a fixed positive integer.  Then every graph $G$ has $O_a(|V(G)|)$ special $a$-alterations.
\end{observation}

With some (likely substantial) additional effort, it should be possible to strengthen Corollary~\ref{cor:mainexpansion} by adding further constraints on the admissible
alterations so that they preserve crossing-criticality.  However, this is not needed for the enumeration:
As we have seen in Theorem~\ref{thm:bounded-pw}, crossing-critical graphs have bounded path-width, and thus
also bounded tree-width.  Hence, we can use the meta-algorithmic theory for graphs of bounded tree-width
to efficiently generate just the special $a$-alterations that are $(c,c_0)$-crossing-critical.

More precisely, for every fixed integer $k$, it is well known (e.g.~\cite{grohe2001computing}) that there exists a Monadic Second-Order Logic (MSOL) sentence $\varphi_{\le\!k}$
such that $G\models \varphi_{\le\!k}$ iff $\crg(G)\le k$.  Briefly, this sentence expresses the fact that it is possible to
at most $2k$-times subdivide an edge of $G$ and identify pairs of the arising vertices of degree two into vertices of degree four
so that the resulting graph $G'$ is planar ($G'$ contains neither $K_5$ nor $K_{3,3}$ as a minor); i.e., we guess the planarization $G'$
of a drawing of $G$ in the plane with at most $k$ crossings.  Since some edges of $G$ may need to be subdivided multiple times, writing
out the sentence $\varphi_k$ explicitly is not straightforward; however, it is rather easy to show its existence through the theory of MSOL transductions~\cite{transductions}.

Using the sentence $\varphi_{\le\!k}$, it is also possible to write a formula $\varphi^-_{\le\!k}$ with one free edge-valued variable
such that $G,e\models \varphi^-_{\le\!k}$ iff $\crg(G-e)\le k$.
Thus, there also exists a MSOL sentence $\varphi_{(c,c_0)}$ expressing that the graph is $(c,c_0)\,$-crossing-critical,
defined as
$$\varphi_{\le c_0}\land\lnot \varphi_{\le c_0-1}\land(\forall e\in E)\,\varphi^-_{\le c-1}(e).$$
Additionally, we can test 2-connectivity in MSOL; let $\psi_{(c,c_0)}$ be the sentence expressing that the graph is $2$-connected and $(c,c_0)$-crossing-critical.

Consider any fixed graph $R$, a fixed finite sequence $r_1$, \ldots, $r_t$ of its vertices, and positive integers $m$ and $d$.
Using, e.g., the theory of MSOL transductions, one can easily see that any MSOL sentence $\varphi$
can be turned into a MSOL formula $\varphi^{(R,m,d,r_1,\ldots,r_t)}$ with $m+t$ free vertex variables
such that $$G,x_1, \ldots,x_m, y_1,\ldots, y_t\models \varphi^{(R,m,d,r_1,\ldots,r_t)}$$
iff the subgraph $G[\{x_1,\ldots,x_m\}]$ is connected, each of the vertices $x_1$, \ldots, $x_m$ has degree
at most $d$ in $G$, the vertices $y_1$, \ldots, $y_t$ are not contained in $\{x_1,\ldots,x_m\}$
but each of them has a neighbor in this set, and the graph $G'$ obtained from $G-\{x_1,\ldots,x_m\}$ by 
adding the graph $R$ and the edges $r_1y_1$, \ldots, $r_ty_t$ satisfies $\varphi$.  Note that 
for any $a\ge \max (m,d,|V(R)|, |E(R)|+t)$, the graph $G'$ is a special $a$-alteration of $G$;
and if $m < |V(R)|$, then $|V(G')|>|V(G)|$.

If $G$ is a graph of tree-width bounded by a fixed integer $b$, we can use the meta-algorithmic theory
of Courcelle~\cite{courcelle} to efficiently enumerate all choices of $x_1, \ldots,x_m$, $y_1,\ldots, y_t\in V(G)$
such that $G,x_1, \ldots,x_m, y_1,\ldots, y_t\models\psi^{(R,m,d,r_1\ldots,r_t)}_{(c,c_0)}$.
By performing this for $d=a$, all graphs $R$ and all finite sequences $r_1,\ldots,r_t\in V(R)$ such that
$a\ge \max (|V(R)|, |E(R)|+t)$, and all non-negative integers $m<|V(R)|$,
we obtain the following conclusion.
\begin{lemma}\label{lemma-enum}
Let $a$, $b$, and $c\le c_0$ be positive integers.  There exists an algorithm that,
given a graph $G$ of tree-width at most $b$ and edge multiplicity at most $c$,
enumerates all $2$-connected $(c,c_0)$-crossing-critical
special $a$-alterations $G'$ of $G$ such that $|V(G')|>|V(G)|$, in time $O_{a,b,c_0}(|V(G)|+s)$,
where $s$ is the number of such alterations $G'$.
\end{lemma}
Note that to achieve this complexity, the algorithm should return just descriptions
of the alterations (the sets of removed and added vertices and added edges) rather than the $a$-alterations themselves.
By Observation~\ref{obs-fewalt}, and further neglecting the fixed parameters $a$, $b$ and $c_0$,
the time complexity of this algorithm is actually $O(|V(G)|)$, or $O(|V(G)|^2)$ if we need to construct the $a$-alterations.

Let us fix positive integers $c\le c_0$ and let $a$ be the value from Corollary~\ref{cor:mainexpansion}.
Furthermore, let $b=\fnlab{thm:bounded-pw}(c)$, so that all $c$-crossing-critical graphs have tree-width at most $b$.
Consider the auxiliary directed graph $\vec{S}$ whose vertices are all non-isomorphic $2$-connected
$(c,c_0)\,$-crossing-critical graphs with at most $n$ vertices,
and $(G,G')$ is an edge of $\vec{S}$ iff $|V(G')|>|V(G)|$ and $G'$ is a special $a$-alteration
of $G$.  By Corollary~\ref{cor:mainexpansion}, all vertices of this graph are reachable
from the vertices corresponding to the graphs with at most $a$ vertices.
Furthermore, by Observation~\ref{obs-fewalt} each vertex of $G$ of $\vec{S}$
has $O(|V(G)|)$ outneighbors, and according to Lemma~\ref{lemma-enum} and the comments following it, these outneighbors
can be enumerated in time $O(|V(G)|^2)$.

Hence, we can enumerate all $2$-connected $(c,c_0)\,$-crossing-critical graphs
by searching this auxiliary graph $\vec{S}$.  Note that we need to ensure that we do not process
each graph more than once.  One way to do it is as follows:  When we process an edge $(G,G')$
of $\vec{S}$, we compute an automorphism-invariant canonical code $c(G')$ for $G'$; since $G'$ has bounded
treewidth, this can be done in time $O(|V(G')|)=O(V(G))$ as shown in~\cite{canonical}.
We keep a hash table of such canonical codes of already processed graphs, and we only process $G'$
if $c(G')$ is not in this table.  Since $G$ has $O(|V(G)|)$ outneighbors, the total overhead for all outneighbors
of $G$ is $O(|V(G)|^2)$, i.e., on the same order as the time needed to enumerate the outneighbors.
For a somewhat more involved approach that avoids the need to store the canonical codes of the generated
graphs (and thus reduces the space complexity to $O(n^2)$), see~\cite{mckay1998isomorph}.

Note that the total size $s$ of the output (the list of all $2$-connected $(c,c_0)$-crossing-critical graphs with at most $n$
vertices) is $\sum_{G\in V(\vec{S})} |V(G)|$, and thus the time complexity of the search is
$$O\left(\sum_{G\in V(\vec{S})} |V(G)|^2\right)=O(sn).$$
Hence, we obtain the following result.

\begin{theorem}\label{thm-gen}
For any positive integers $c\le c_0$, there exists an algorithm that, given a positive integer $n$,
lists all non-isomorphic $2$-connected $(c,c_0)$-crossing-critical graphs with at most $n$ vertices
in time $O(sn)$, where $s$ is the total size of the output.
\end{theorem}

Lastly, we remark that one could also avoid using the aforementioned Courcelle's meta-algorithmic result
(and thus make the procedure look ``more concrete''), for instance, 
by using the very recent crossing-number algorithm by Lokshtanov et al.~\cite{FPTcr-SODA25}
applied to the list enumerating all special $a$-alterations $G_0$ of a graph $G$ as in Lemma~\ref{lemma-enum}.
However, although the algorithm of \cite{FPTcr-SODA25} runs in linear time for fixed $c$ and $c_0$, testing
whether a candidate graph $G_0$ is $(c,c_0)$-crossing-critical requires linear number of calls to the algorithm
(plus one call for every edge of $G_0$) and this would altogether rise the time complexity by a multiplicative factor of~$O(n)$.

\section{Conclusion}
\label{sec:conclusion}

To summarize, we have shown a structural characterization and an enumeration
procedure for all $2$-connected $c\,$-crossing-critical graphs, using
bounded-size replication steps over an implicit finite set of basic
$c\,$-crossing-critical graphs.
The characterization can be used to describe all $c\,$-crossing-critical graphs
(without the connectivity assumption), as discussed in Section~\ref{ssec:2con}

With this characterization and other tools from this paper at hand,
one can expect a significant progress in the crossing number research, both
from mathematical and algorithmic perspectives.
For example, the tools can be used to quite easily derive that for every positive integer $c$, there are only
finitely many $3$-regular $c\,$-cros\-sing-critical graphs, a claim that has been so far proved only
via the Graph minors theorem of Robertson and Seymour.

One can similarly hope for a progress on some of the long-standing open questions
in the area of crossing-critical graphs, though this may be hampered by the lack
of information on the basic graphs.  For example, while it seems at first plausible that
it can lead to asymptotic improvements to the bound of Theorem~\ref{thm-crit-cno},
we can notice that the largest crossing number can (and indeed, must) occur already
among the basic graphs, since the crossing number is not affected by the expansions.
Some ideas (such as the existence of long sashes in crossing-critical graphs)
could still be of use in this context, though.

More immediately, we suspect that using our structural results, it should be
possible to address another long-standing open question
\cite{richter1993minimal}, whether for a fixed
integer $c$, the number of $5$-regular $c\,$-crossing-critical graphs is bounded
(or whether an infinite family of such graphs exists).

\bibliography{cycles,extra}


\end{document}